\documentclass{amsart}

\usepackage{amssymb,amsmath,amscd,graphicx,amsthm,color,bm,tikz-cd,mathtools,arydshln,relsize,upgreek}
\usepackage[center]{caption}
\newtheorem{theorem}{Theorem}
\newtheorem{lemma}[theorem]{Lemma}
\newtheorem{proposition}[theorem]{Proposition}
\newtheorem{corollary}[theorem]{Corollary}

\theoremstyle{definition}

\theoremstyle{remark}
\newtheorem{remark}[theorem]{Remark}

\numberwithin{equation}{section}
\numberwithin{theorem}{section}

\def\A{{\mathcal A}}
\def\AA{{\mathbb A}}
\def\B{{\mathcal B}}

\def\C{{\mathbb C}}
\def\CC{{\mathcal C}}
\def\D{{\mathfrak d}}
\def\DD{{\mathcal D}}

\def\FF{{\mathcal F}}

\def\G{{\mathcal G}}
\def\GCC{{\G\CC}}

\def\Hs{\mathbf H}
\def\Is{\mathbf I}

\def\N{\mathcal N}
\def\NA{{\tt A}}
\def\NB{{\tt B}}
\def\NF{{\tt F}}
\def\NG{{\tt G}}
\def\NK{{\tt K}}
\def\Nphi{{\mathsf\Phi}}

\def\O{{\mathcal O}}
\def\P{{\mathcal P}}
\def\PP{{\mathcal P}}
\def\Ps{\mathbf P}

\def\RR{{\mathcal R}}

\def\Rs{\mathbf R}
\def\S{{\mathbf S}}

\def\TE{{\mathcal T}}

\def\Ts{\mathbf T}

\def\UU{\mathrm U}
\def\UC{\overline{\A}}

\def\Ws{\mathbf W}

\def\Z{{\mathbb Z}}

\def\bfG{\mathbf \Gamma}

\def\nablac{{\nabla}}

\def\pu{\varnothing}

\def\b{\mathfrak b}

\def\bgamma{\pmb\gamma}

\def\ec{{\tt c}}
\def\ef{{\tt f}}
\def\eg{{\tt g}}
\def\ek{{\tt k}}

\def\er{{\rm r}}
\def\es{{\tt s}}

\def\la{\langle}
\def\ra{\rangle}

\def\fy{\varphi}

\def\g{\mathfrak g}
\def\gammar{{\gamma^\er}}
\def\gammac{{\gamma^\ec}}

\def\gl{\mathfrak g\mathfrak l}

\def\h{\mathfrak h}

\def\ls{\le}

\def\n{\mathfrak n}
\def\nar{\setlength\arraycolsep{2pt}}
\def\one{\mathbf 1}

\def\sl{\mathfrak {sl}}
\def\smup{{\scriptstyle\Upsilon}}

\def\wx{{\widetilde{X}}}
\def\wY{{\widetilde{Y}}}
\def\u{{\bf u}}
\def\x{{\bf x}}

\def\zero{\mathbf 0}

\def\Ad{\operatorname{Ad}}
\def\CG{{\rm CG}}

\def\End{\operatorname{End}}

\def\Kil{\langle {\cdot},{\cdot}\rangle}

\def\Mat{\operatorname{Mat}}

\def\Poi{{\{{\cdot},{\cdot}\}}}

\def\Tr{\operatorname{Tr}}

\def\deg{{\operatorname{deg}}}
\def\diag{\operatorname{diag}}

\def\sign{{\operatorname{sign}}}

\def\:{{:\ }}

\begin{document}

\title{A generalized cluster structure on $GL_n$ via birational Poisson maps}

\author{Misha Gekhtman}

\address{Department of Mathematics, University of Notre Dame, Notre Dame,
IN 46556}
\email{mgekhtma@nd.edu}

\author{Michael Shapiro}
\address{Department of Mathematics, Michigan State University, East Lansing,
MI 48823}
\email{mshapiro@math.msu.edu}

\author{Alek Vainshtein}
\address{Department of Mathematics \& Department of Computer Science, University of Haifa, Haifa,
Mount Carmel 31905, Israel}
\email{alek@cs.haifa.ac.il}

\begin{abstract}
In a recent work, we constructed a rational map from a simple Lie group $\G$ to itself that intertwines the standard 
Poisson--Lie structure on $\G$ with a Poisson homogeneous one defined by a pair of quasi-triangular solutions to the classical Yang--Baxter equation(CYBE) known as R-matrices. We also showed, in the case of $SL_n$, that if the combinatorial Belavin--Drinfeld data associated with these R-matrices satisfies certain {\em aperiodicity conditions}, the map is, in fact, birational and can be used to obtain an initial cluster for an exotic cluster structure on $SL_n$ via the pullback of Berenstein--Fomin--Zelevinsky cluster variables. The same strategy was later used by the first author and 
D.~Voloshyn to describe generalized cluster structures 
compatible with the Poisson dual of the Poisson--Lie bracket defined by a quasi-
triangular R-matrix.

In this paper we further promote the use of birational Poisson maps in constructing generalized cluster structures by applying it in the situation when the aperiodicity condition is not satisfied. To this end, we describe a generalized cluster structure on $GL_n$ compatible with the Poisson homogeneous bracket defined by 
two Cremmer--Gervais solutions to the
CYBE related via conjugation by the longest element of the Weyl group. The key ingredient to our construction is a birational map that 
connects the bracket under consideration with two other Poisson brackets: the Poisson dual 
to Cremmer--Gervais Poisson--Lie bracket on $GL_{n-1}$ and the bracket 
on a
certain space of complex rational functions of one variable closely related to cluster algebraic interpretation of 
Coxeter--Toda flows. New notions of a {\em 
regular pullback of a seed\/} and of an {\it almost-cluster structure\/} whose detailed description are given in a separate note also play an important role in our construction.
\end{abstract}


\maketitle

\tableofcontents

\section{Introduction} 
A search for cluster structures in rings of regular functions on algebraic varieties has been among the most active directions of research in the theory of cluster algebras starting from the time they were discovered by Fomin and Zelevinsky.   Two approaches to constructing cluster structures can be clearly distinguished in the literature. One relies on the combinatorics of Weyl and braid groups and has been successfully applied to a number of varieties, particularly those  appearing in Lie theory: from Grassmannians~\cite{Scott} and double Bruhat cells of a simply-connected simple Lie group~\cite{BFZ} in the early days of cluster algebras to the recent works on cluster structures on braid 
varieties~\cite{CGGLSS}. The second approach was  initiated by the authors in~\cite{GSV1} where the notion Poisson structures compatible with cluster transformations was introduced and the strategy of constructing a cluster structure in the ring of regular functions on a Poisson variety starting from a {\em regular log-canonical} chart was first implemented in the case of Grassmannians. The Poisson bracket in this case was the Poisson homogeneous bracket with respect to the natural action of $GL_n$ equipped with the {\em standard Poisson--Lie bracket}, and the resulting cluster structure and the one from~\cite{Scott} turned out to coincide. Similarly, in~\cite[Ch.~4.3]{GSVb} we have shown that the cluster structure of~\cite{BFZ} on double Bruhat cells of a simply-connected simple Lie group $\G$ is compatible with the standard Poisson--Lie structure on $\G$.
 
 But what happens if one modifies the Poisson bracket on the underlying variety? Does a compatible cluster structure still exist, and if so, how is it related to the original one? We posed this question, in the case of simple Poisson--Lie groups, in~\cite{GSVM} and, while observing that already in $SL_2$ there exists a Poisson--Lie bracket not compatible with any cluster structure, formulated a conjecture on the existence of compatible cluster structures for Poisson–-Lie brackets corresponding to quasi-triangular solutions of  the classical Yang--Baxter equation  (CYBE) classified by Belavin and Drinfeld in~\cite{BD}.  These solutions to CYBE, called  R-matrices, are  parametrized by discrete data consisting of an isometry between two subsets of positive roots in the root system of the Lie algebra of $\G$ and a continuous parameter that can be described as an element of the tensor square of the Cartan subalgebra that satisfies a system of linear equations governed by the discrete data. The discrete data determines a {\em Belavin--Drinfeld (BD) class\/} of R-matrices and corresponding Poisson--Lie brackets, and continuous data specifies 
a particular R-matrix and a bracket within this class. Given two such brackets on $\G$ associated with representatives of two BD classes, one can define a Poisson--Lie group 
$\G\times \G$ equipped with the direct product Poisson structure and then construct a Poisson-homogeneous structure on $\G$ with respect to the action of  $\G\times \G$ by right and left multiplication. 

In the series of papers culminating in~\cite{GSVple},~\cite{GSVuni}, we worked on an extended version of the conjecture 
of~\cite{GSVM}: {\em for any Poisson-homogeneous bracket associated with an arbitrary pair of BD data, 
 there exists a  compatible regular complete, possibly generalized, cluster structure in the ring of regular functions on $\G$.} In~\cite{GSVple},~\cite{GSVuni}, we proved the conjecture for $SL_n$ in the case of {\em aperiodic\/} BD data. Here aperiodicity means that a certain map $\delta$ naturally induced by the discrete part of the BD data on simple positive roots has no periodic orbits.
The geometric meaning of the aperiodicity condition was clarified in~\cite{GSVuni} where we adopted a new approach to constructing a regular log-canonical chart on $\G$ for a Poisson-homogeneous bracket under consideration, which constitutes the first and crucial step in uncovering a cluster structure we are looking for. In contrast with~\cite{GSVple}, where an {\em ad hoc\/} ansatz  was used to identify an initial log-canonical family of functions and the proof consisted of a long direct computation, in~\cite{GSVuni} we constructed and utilized a rational Poisson map $h^{r,r'}$ between two copies of $\G$ endowed with two different Poisson-homogeneous structures. One is $\Poi_{r,r'}$ determined by a pair of R-matrices from two arbitrary BD classes. The other, $\Poi_{r,r'}^{\rm st}$, corresponds to two R-matrices from the {\em standard\/} BD class whose Cartan parts match those of $r, r'$. The  map $h^{r,r'}$ maps 
$(\G, \Poi_{r,r'}^{\rm st})$ to $(\G, \Poi_{r,r'})$. In the $SL_n$ case with aperiodic BD data, $h^{r,r'}$ is birational and we are able to pull back  a conveniently chosen  cluster in the standard cluster structure on $\G$ to obtain an initial seed for a cluster structure compatible with  $\Poi_{r,r'}$. 

In~\cite{GSVpest}, we considered an example of a non-aperiodic  BD data in $SL_6$ to illustrate the emergence of generalized cluster exchange relations in this case and to sketch a conjectural construction of a compatible generalized cluster structure. We will revisit this construction in future publication.
It is important, however, to emphasize a key point illustrated by the example considered in~\cite{GSVpest}. Namely, in the non-aperiodic case for $SL_n$, periodic orbits of the map $\delta$  on the set of positive simple roots induced by the BD data are  divided into equivalence classes. Each family gives rise to an infinite periodic staircase matrix whose generalized characteristic polynomial provides coefficients of the corresponding generalized cluster exchange relation.

In this paper we shed additional light on generalized cluster structures in $SL_n$ arising in a non-aperiodic case by providing a comprehensive analysis of a {\em generalized cluster structure $\G\CC_n$ on $SL_n$ (and its extensions to $GL_n$ and $\Mat_n$)
compatible with $\Poi=\Poi_{r,r'}$, where $r$ and $r'$ are two Cremmer--Gervais solutions to the CYBE related  via conjugation by the longest element of the Weyl group}. 
We chose this example because it represents a ``purely'' non-aperiodic case in a sense that the corresponding map $\delta$ has only periodic orbits which form a single equivalence class.  In a way, this choice of a non-aperiodic pair $(r,r')$ mirrors the progress of our systematic study of the aperiodic case which started, in~\cite{GSVPNAS, GSVMem}, with the description of the exotic cluster structure compatible with the Cremmer--Gervais Poisson--Lie bracket.

This work further promotes an approach based on using birational Poisson maps to recover a new (generalized) cluster structure from known ones. 
Following~\cite{GSVuni}, this approach was also applied in~\cite{GV} to describe generalized cluster structures compatible with Poisson dual $\Poi^\dag_r$ of the Poisson--Lie bracket defined by a quasi-triangular solution $r$ to the CYBE. We use the construction in \cite{GV} as a crucial ingredient in defining a birational map $\Psi$ that connects the bracket $\Poi$ defined above to two other Poisson brackets: the Poisson dual  $\Poi^\dag$ to Cremmer--Gervais Poisson--Lie bracket on $SL_{n-1}$, 
and the bracket $\Poi^T$ on a certain space $\bar\RR_n$ of complex rational functions of one variable closely related to the one we studied in~\cite{GSV_Acta} in the context of cluster algebraic interpretation of B\"acklund--Darboux transformations of the Coxeter--Toda flows. Using a notion of a {\em regular pullback of a seed\/}, we apply  $\Psi$ to obtain an initial 
seed for $\G\CC_n$
from initial seeds of structures compatible with $\Poi^\dag$ and $\Poi^T$, where the former (generalized) was constructed 
in~\cite{GV} and the latter is a slight modification of the one in~\cite{GSV_Acta}.  We expect the map 
$(SL_n,\Poi) \to (SL_{n-1},\Poi^\dag)$ featured in our construction of $\Psi$ to have an analogue for every equivalence class of periodic orbits of $\delta$ for any non-aperiodic BD data, with the target space  $(SL_{k},\Poi^\dag)$, where $k$ is determined by the data but 
$\Poi^\dag$ is the same as in the current paper. The question of what the Poisson space $(\bar\RR_n, \Poi^T)$ should be replaced with so that an analogue of the map $\Psi$ could be defined for any non-aperiodic BD data is one of the main directions of our future research.

The paper is organized as follows. Section~\ref{sec:prelim} begins with a brief overview of the necessary background on Poisson--Lie groups and Poisson-homogeneous structures on simple Lie groups arising from the 
Belavin--Drinfeld classification of quasi-triangular $r$-matrices, and on generalized cluster structures compatible with Poisson brackets. We also discuss, in Section~\ref{seedpullback},  the notion of a {\em regular pullback of a seed\/} which plays an important role in our main construction;  a more detailed treatment of this subject is contained in a separate paper~\cite{GSVpullback}. In Section~\ref{mainresult} we define the Poisson homogeneous bracket 
$\Poi$ at the center of our investigation---it corresponds to the non-aperiodic pair of the ``opposite'' Cremmer--Gervais 
data---and formulate our main result, Theorem~\ref{thm:main} on the existence of a generalized cluster structure $\G\CC_n$  compatible with $\Poi$. The rest of Section~\ref{sec:prelim} outlines the proof of the main theorem that hinges on definition and properties of a birational map $\Psi$ that connects the bracket $\Poi$ with two other Poisson brackets: the Poisson dual  $\Poi^\dag$ to Cremmer--Gervais Poisson--Lie bracket on $GL_{n-1}$, and the bracket $\Poi^T$ on a certain space $\bar\RR_n$ of complex rational functions of one variable closely related to the one we studied in~\cite{GSV_Acta} in the context of cluster algebraic interpretation of B\"acklund--Darboux transformations of Coxeter--Toda flows. Using a regular pullback via $\Psi$ of initial clusters for the structures compatible with $\Poi^\dag$ and $\Poi^T$, 
%the former generalized constructed in~\cite{GV}, the latter a slight modification of the one in~\cite{GSV_Acta}, 
we obtain an initial seed for $\G\CC_n$. The proof of the required modification of the construction of~\cite{GSV_Acta} is contained in the Appendix.

The results that constitute steps in  the proof of the main theorem are stated in Section~\ref{outline}. Their detailed proofs are presented in Sections~\ref{sec:twomaps}--\ref{regcomplet}. These include several additional technical claims that we feel may prove useful beyond the specific context of this paper, e.g.~Proposition~\ref{frozenchar} that elucidates a Poisson-geometric nature for frozen variables in a cluster structure compatible with a Poisson bracket.

Our research was supported in part by the NSF research grants DMS \#2100785 (M.~G.) and  DMS \#2100791 (M.~S.), and ISF grants \#876/20 and \#2848/25 (A.~V.).  
 While working further on this project, we benefited from support and hospitality from the following institutions: International Centre for Mathematical Sciences, Edinburgh (M.~G., M.~S., A.~V., Summer 2023; M.~S., Summer 2024, Spring 2025), Institut des Hautes Etudes Scientifiques,
Bures–sur–Yvette (A.~V., Fall 2023),
Mathematical Institute of the University of Heidelberg (M.~G., Summer 2024, supported by the Mercator fellowship), Max Planck Institute for Mathematics in the Sciences (M.~G., Summers 2024/25), Michigan State University (A.~V., Spring 2025),
Istituto Nazionale di Alta Matematica Francesco Severi and University l'Aquila (A.~V., Spring 2025), Instituto de 
Matem\'atica Pura e Aplicada and Pontificia Universidade Cat\'olica do Rio de Janeiro (M.~S., Summer 2025).
 %University of Haifa (M.~G., M.~S., Summer 2022), Michigan State University (A.~V., Fall 2022), University of Notre Dame (A.~V., Spring 2023), Max Planck Institute for Mathematics, Bonn (A.~V., Spring  2023), and Research in Pairs Program at the Mathematisches Forschungsinstitut Oberwolfach 
%(M.~G., M.~S., A.~V., Summer 2023), where the project was completed. 
 We are grateful to all these institutions for their hospitality and outstanding working conditions they provided. 
 Special thanks are due to P.~Pushkar' and D.~Voloshyn for valuable discussions.

\section{Preliminaries and Main Result}\label{sec:prelim}
\subsection{Poisson--Lie groups} 
A reductive complex Lie group $\G$ equipped with a Poisson bracket $\Poi$ is called a {\em Poisson--Lie group\/}
if the multiplication map $\G\times \G \ni (X,Y) \mapsto XY \in \G$
is Poisson. Perhaps, the most important class of Poisson--Lie groups
is the one associated with quasitriangular Lie bialgebras defined in terms of  {\em classical R-matrices\/} 
(see, e.~g., \cite[Ch.~1]{CP}, \cite{r-sts} and \cite{Ya} for a detailed exposition of these structures).

Let $\g$ be the Lie algebra corresponding to $\G$, $\Kil$ be an invariant nondegenerate form on $\g$,
 and let $\mathfrak{t}\in \g\otimes\g$ be the corresponding Casimir element.
For an arbitrary element $r=\sum_i a_i\otimes b_i\in\g\otimes\g$ denote
\[
[[r,r]]=\sum_{i,j} [a_i,a_j]\otimes b_i\otimes b_j+\sum_{i,j} a_i\otimes [b_i,a_j]\otimes b_j+
\sum_{i,j} a_i\otimes a_j\otimes [ b_i,b_j]
\]
and $r^{21}=\sum_i b_i\otimes a_i$.
A {\em classical R-matrix} is an element $r\in \g\otimes\g$ that satisfies
{\em the classical Yang-Baxter equation (CYBE)\/} $[[r, r]] =0$
together with the condition $r + r^{21} = \mathfrak{t}$.
The Poisson--Lie bracket on $\G$ that corresponds to $r$ can be written as
\begin{equation}\label{sklyabra}
\begin{aligned}
\{f_1,f_2\}_r &= \langle R_+(\nabla^L f_1), \nabla^L f_2 \rangle - \langle R_+(\nabla^R f_1), \nabla^R f_2 \rangle\\
&= \langle R_-(\nabla^L f_1), \nabla^L f_2 \rangle - \langle R_-(\nabla^R f_1), \nabla^R f_2 \rangle,
\end{aligned}
\end{equation} 
where $R_+,R_- \in \End \g$ are given by $\langle R_+ \eta, \zeta\rangle = \langle r, \eta\otimes\zeta \rangle$, 
$-\langle R_- \zeta, \eta\rangle = \langle r, \eta\otimes\zeta \rangle$ for any $\eta,\zeta\in \g$ and  
$\nabla^L$, $\nabla^R$ are the right and the left gradients of functions on $\G$ with respect to $\Kil$ 
defined by
\begin{equation*}
\left\langle \nabla^R f(X),\xi\right\rangle=\left.\frac d{dt}\right|_{t=0}f(e^{t\xi}X),  \quad
\left\langle \nabla^L f(X),\xi\right\rangle=\left.\frac d{dt}\right|_{t=0}f(Xe^{t\xi})
\end{equation*}
for any $\xi\in\g$, $X\in\G$. 

The double of $\g$ is 
$D(\g)=\g  \oplus \g$ equipped with an invariant nondegenerate bilinear form
$\langle\langle (\xi,\eta), (\xi',\eta')\rangle\rangle = \langle \xi, \xi'\rangle - \langle \eta, \eta'\rangle$. 
Define subalgebras $\D_\pm$ of $D(\g)$ by
$\D_+=\{( \xi,\xi) : \xi \in\g\}$ and $\D_-=\{ (R_+(\xi),R_-(\xi)) : \xi \in\g\}$.
%where $R_\pm\in \End\g$ is given by $R_\pm=\frac{1}{2} ( R \pm \Id)$. 
%The operator $R_D= \pi_{\D_+} - \pi_{\D_-}$ can be used to 
Using right and left gradients with respect to $\langle\langle {\cdot} ,{\cdot} \rangle\rangle$, 
one can define 
a Poisson--Lie structure on $D(\G)=\G\times \G$, the {\em Drinfeld double\/} of the group $\G$, see \cite{GSVdouble}
for details.

The diagonal subgroup $\{ (X,X)\ : \ X\in \G\}$ is a Poisson--Lie subgroup of $D(\G)$ (whose Lie algebra is $\D_+$) naturally isomorphic to $(\G,\Poi_r)$.
The group $\G^*$  whose Lie algebra is $\D_-$ is a Poisson-Lie subgroup of $D(\G)$ called the {\em dual Poisson--Lie group of $\G$}.
The map $D(\G) \to \G$ given by $(X,Y) \mapsto U=X^{-1} Y$ induces another Poisson bracket on $\G$, see~\cite{r-sts}; we denote this bracket $\Poi^\dag_r$. 
The image of the restriction of this map to $\G^*$ is denoted 
$\G^\dag=(\G,\Poi^\dag_r)$.
As proved in \cite[Lemma~5.5]{GSVdouble} 
(see equation~(5.6) there), this bracket is given by
\begin{equation}\label{dualbra}
\{f_1,f_2\}^{\dagger}_r =\langle R_+([\nabla f_1,U]),[\nabla f_2,U]\rangle-\langle[\nabla f_1,U],\nabla f_2{\cdot} U\rangle.
% \langle R_+(\nabla^Lf_1-\nabla^R f_1), \nabla^L f_2 - \nabla^R f_2\rangle - \langle \nabla^Lf_1-\nabla^R f_1, \nabla^L f_2\rangle.
\end{equation}

 As is explained in detail in~\cite{GSVple, GSVuni},
 it is natural to consider the following generalization of the 
bracket \eqref{sklyabra}.
Let $r, \bar r$ be two classical R-matrices, and $R_+, \bar R_+$ be the corresponding operators, then we write
\begin{equation}\label{sklyabragen}
\{f_1,f_2\}_{r,\bar r} = \langle \bar R_+(\nabla^L f_1), \nabla^L f_2 \rangle -\langle R_+(\nabla^R f_1), \nabla^R f_2 \rangle.
\end{equation} 
By \cite[Proposition 12.11]{r-sts}, the above expression defines a Poisson bracket, which is not Poisson--Lie unless $r=\bar r$,
in which case $\{f_1,f_2\}_{r,r}$ evidently coincides with $\{f_1,f_2\}_{r}$. 
The bracket \eqref{sklyabragen} defines a Poisson homogeneous structure on $\G$ with respect to the left and right multiplication by Poisson--Lie groups $(\G,\Poi_{\bar r})$ and
$(\G,\Poi_{r})$, respectively. It will be one of the main objects of our study in this paper.

The classification of classical R-matrices for simple complex Lie groups was given by Belavin and Drinfeld in \cite{BD}.
Let $\G$ be a simple complex Lie group, $\Phi$ be the root system associated with its Lie algebra $\g$, $\Phi^+$ be the set of positive roots, and $\Pi\subset \Phi^+$ be the set of positive simple roots. 
A {\em Belavin--Drinfeld triple} $\bfG=(\Gamma_1,\Gamma_2, \gamma)$ (in what follows, a {\em BD triple\/})
consists of two subsets $\Gamma_1,\Gamma_2$ of $\Pi$ and an isometry $\gamma\:\Gamma_1\to\Gamma_2$ nilpotent in the 
following sense: for every $\alpha \in \Gamma_1$ there exists $m\in\mathbb{N}$ such that $\gamma^j(\alpha)\in \Gamma_1$ 
for $j\in [0,m-1]$, but $\gamma^m(\alpha)\notin \Gamma_1$.  We call the BD triple 
$\bfG^*=\{\Gamma_2,\Gamma_1, \gamma^*\}$ with $\gamma^*\: \Gamma_2\to \Gamma_1$ defined via 
$\gamma^*(\gamma(\alpha)) = \alpha$ for $\alpha\in \Gamma_1$ the {\em adjoint\/} of $\bfG$. 
Recall that the longest element $w_0$ of the Weyl group of $\g$ acts on $\Phi$ and takes positive simple roots to negative simple roots. We define the {\em $w_0$-conjugate\/} of $\bfG$ %denoted by $\bfG^{w_0}$ 
via $\bfG^{w_0}=\{-w_0\Gamma_1,-w_0\Gamma_2, \gamma^{w_0}\}$ with $\gamma^{w_0}\: -w_0\Gamma_1\to-w_0\Gamma_2$
given by $\gamma^{w_0}(-w_0\alpha) = -w_0\gamma(\alpha)$ for $\alpha\in \Gamma_1$.

 The isometry $\gamma$ yields an isomorphism, also denoted by $\gamma$, between the Lie subalgebras $\g^{\Gamma_1}$ 
and $\g^{\Gamma_2}$ that correspond to $\Gamma_1$ and $\Gamma_2$. It is uniquely defined by the property 
$\gamma e_\alpha = e_{\gamma(\alpha)}$ for $\alpha\in \Gamma_1$, where $e_\alpha$ is the Chevalley generator corresponding to 
the root $\alpha$. The isomorphism $\gamma^*\: \g^{\Gamma_2} \to \g^{\Gamma_1}$ is defined as the adjoint to $\gamma$ with respect to the form $\Kil$.  It corresponds to the isometry $\gamma^*$ and 
is given by $\gamma^* e_{\gamma(\alpha)}=e_{\alpha}$ for $\alpha\in \Gamma_1$.
 Both $\gamma$ and $\gamma^*$ can be extended to maps of $\g$ to itself by applying first the orthogonal
projection on $\g^{\Gamma_1}$ (respectively, on $\g^{\Gamma_2}$) with respect to $\Kil$; clearly, the extended
maps remain adjoint to each other. Note that the restrictions of $\gamma$ and $\gamma^*$ to the positive and the negative nilpotent subalgebras $\n_+$ and $\n_-$ of $\g$ are Lie algebra homomorphisms 
of $\n_+$ and $\n_-$  to themselves, and $\gamma(e_{\pm\alpha})=0$
for all $\alpha\in\Pi\setminus\Gamma_1$. Further, if $\g$ is simply connected $\gamma$ can be lifted to $\bgamma=\exp\gamma$;  note that $\bgamma$ is  defined only on subgroups $\N_+$ and $\N_-$ that correspond to $\n_+$ and $\n_-$,
and is a group homomorphism.

 By the classification theorem, each classical R-matrix is equivalent to an R-matrix $r^\bfG$ from a {\it Belavin--Drinfeld class\/} defined by a BD triple $\bfG$. The operator $R^\bfG_+$ corresponding
to a member of this class is given by
\begin{equation}
\label{Rplusgamma}
R^\bfG_+=R_0^\bfG+\frac1{1-\gamma}\pi_{>}-\frac{\gamma^*}{1-\gamma^*}\pi_{<},
\end{equation}
where $\pi_{>}$, $\pi_{<}$ are projections of  
$\g$ onto $\n_+$ and $\n_-$ and $R_0^\bfG$ acts on $\h$ (see~\cite{GSVple} for more details).

\subsection{Generalized cluster structures and compatible Poisson brackets}
Following \cite{GSVdouble, GSVnewdouble}, we remind the definition of a generalized cluster structure represented by a quiver with multiplicities. Let $Q=(Q,d_1,\dots,d_N)$ be
a quiver on $N$ mutable and $M$ frozen vertices with positive integer multiplicities $d_i$ at mutable vertices. 
A vertex is called {\it special\/} if its multiplicity is greater than~1. 
A frozen vertex is called {\it isolated\/} if it is not connected to any other vertices. 
%Let $\F$ be the field of rational functions in $N+M$ independent variables with rational coefficients. 
%There are $M$  distinguished variables corresponding to frozen vertices; they are denoted $x_{N+1},\dots,x_{N+M}$ and called {\em stable}, or {\em frozen\/} variables. 
%The {\it coefficient group\/} is a free multiplicative abelian group of Laurent monomials in stable variables, 
%and its integer group ring is $\bar\AA=\Z[x_{N+1}^{\pm1},\dots,x_{N+M}^{\pm1}]$ (we write $x^{\pm1}$ instead of $x,x^{-1}$).

An {\em extended seed\/} %in $\F$ 
is a triple $\Sigma=(\x,Q,\P)$, where $\x=(x_1,\dots,x_N, x_{N+1},\dots,\allowbreak x_{N+M})$ are
cluster and frozen variables attached to the vertices of $Q$
 %a transcendence basis of $\F$ over the field of fractions of  $\bar\AA$ 
and $\P$ is a set of $N$ {\em strings}. The $i$th string is a collection of 
monomials $p_{ir}\in\AA=\Z[x_{N+1},\dots,x_{N+M}]$,
 $0\le r\le d_i$, such that  
$p_{i0}=p_{id_i}=1$; it is called {\em trivial\/} if $d_i=1$, and hence both elements of the string are equal to one.
The monomials $p_{ir}$ are called {\em exchange coefficients}.

Given a seed as above, the {\em adjacent cluster\/} in direction $k$, $1\le k\le N$,
is defined by $\x'=(\x\setminus\{x_k\})\cup\{x'_k\}$,
where the new cluster variable $x'_k$ is given by the {\em generalized exchange relation}
\begin{equation}\label{genger}
x_kx'_k=\sum_{r=0}^{d_k}p_{kr}u_{k;>}^r v_{k;>}^{[r]}u_{k;<}^{d_k-r}v_{k;<}^{[d_k-r]}
\end{equation}
with
%where $u_{k;>}$ and $u_{k;<}$, $1\le k\le N$, are %{\em cluster $\tau$-monomials\/}  defined by
\begin{equation*}
u_{k;>}=\prod_{k\to i\in Q} x_i,\qquad  u_{k;<}=\prod_{i\to k \in Q}x_i,
\end{equation*}
%where the products are taken over all edges between $k$ and mutable vertices,
and 
%{\em stable $\tau$-monomials\/} $v_{k;>}^{[r]}$ and $v_{k;<}^{[r]}$, $1\le k\le N$, $0\le r\le d_k$, defined by
\begin{equation*}\label{stable}
v_{k;>}^{[s]}=\prod_{N+1\le i\le N+M}x_i^{\lfloor sb_{ki}/d_k\rfloor},\qquad
v_{k;<}^{[s]}=\prod_{N+1\le i\le N+M}x_i^{\lfloor sb_{ik}/d_k\rfloor},
\end{equation*}
where $b_{ki}$ is the number of edges from $k$ to $i$ and $b_{ik}$ is the number of edges from $i$ to $k$.
The strings remain unchanged, except for the $k$-th one that is reversed:  $p'_{kr}=p_{k,d_k-r}$. 

The standard definition of the {\it quiver mutation\/} in direction $k$ is modified as follows: 
if both vertices $i$ and $j$
in a path $i\to k\to j$ are mutable, then this path contributes $d_k$ edges $i\to j$ to the mutated quiver $Q'$; if one of the vertices $i$ or $j$ is frozen then the path contributes $d_j$ or $d_i$ edges $i\to j$ to $Q'$. The multiplicities at the vertices do not change. Note that isolated vertices remain isolated in $Q'$.

Given an extended seed $\Sigma=(\x,Q,\P)$, we say that a seed
$\Sigma'=(\x',Q',\P')$ is {\em adjacent\/} to $\Sigma$ (in direction
$k$) if $\x'$, $Q'$ and $\P'$ are as above. 
Two such seeds are {\em mutation equivalent\/} if they can
be connected by a sequence of pairwise adjacent seeds. 
The set of all seeds mutation equivalent to $\Sigma$ is called the {\it generalized cluster structure\/}  associated with $\Sigma$ and denoted by $\GCC=\GCC(\Sigma)$.

Fix a ground ring $\widehat{\AA}$ such that $\AA\subseteq\widehat\AA\subseteq\bar\AA=\Z[x_{N+1}^{\pm1},\dots,x_{N+M}^{\pm1}]$. The {\it generalized upper cluster algebra\/}
$\UC(\GCC)$ is the intersection of the rings of Laurent polynomials over $\widehat{\AA}$ in cluster variables taken over all seeds in $\GCC$. Let $V$ be a rational variety over $\C$, $\C(V)$ be the function field of $V$, and $\O(V)$ be the ring of regular functions on $V$. A generalized cluster structure $\GCC$
in $\C(V)$ is an embedding of $\x$ into $\C(V)$ that can be extended to a field isomorphism between 
%$\F_\C=\F\otimes\C$ 
the ambient field tensored with $\C$ and $\C(V)$. 
It is called {\it regular on $V$\/} if any cluster variable in any cluster belongs to $\O(V)$, and {\it complete\/} if 
$\UC(\GCC)$ tensored with $\C$ is isomorphic to $\O(V)$. The choice of the ground ring is discussed 
in~\cite[Section 2.1]{GSVdouble}.

Let $\Poi$ be a Poisson bracket on the ambient field, and $\GCC$ be a generalized cluster structure. 
We say that the bracket and the generalized cluster structure are {\em compatible\/} if any extended
cluster $\widetilde{\x}=(x_1,\dots,x_{N+M})$ is {\em log-canonical\/} with respect to $\Poi$, that is,
$\{x_i,x_j\}=\omega_{ij} x_ix_j$,
where $\omega_{ij}\in\Z$ are constants for all $i,j$, $1\le i,j\le N+M$. Instrumental in the study of compatibility
are certain Laurent monomials in cluster and frozen variables called {\it $y$-variables\/}. They are defined for any mutable vertex $k\in Q$ via
\begin{equation}\label{yvar}
y_k=\frac{u_{k;>}v_{k;>}}{u_{k;<}v_{k;<}}.
\end{equation}

\subsection{Regular pullbacks of a seed and almost-cluster structures}\label{seedpullback} 
Let %$A$ be an affine space, $V$ be a quasi-affine variety and 
$\Psi :\AA^{L'}\to \AA^{L}$ be a dominant rational map. Assume that the algebra of regular functions on $\AA^L$ is endowed with a cluster structure $\CC$ defined by an
initial seed $\Sigma=(\x,Q)$. Let $x_k$ be the cluster variable attached to a vertex $k$ of $Q$; put $\tilde x_k=x_k\circ\Psi$. Clearly, $\tilde x_k$ is a rational function on $\AA^{L'}$; let $u_k$ and $z_k$ be the numerator 
and the denominator of $\tilde x_k$, assuming they are coprime. Further, let $\{w_p\}_{p=1}^P$ be the family of all irreducible factors of the denominators $z_k$ over all vertices of $Q$, so that each $z_k$ can be written
as $z_k=\prod_{p=1}^Pw_p^{\lambda_{pk}}$ for some nonnegative integer $\lambda_{pk}$; the polynomials $w_p$ are called
{\it distinguished\/}.
The {\it discrepancy\/} at vertex $k$ 
relative to $w_p$ is defined via
\begin{equation}\label{discrep}
\delta_{pk}=\sum_{i}b_{ik}\lambda_{pi}-\sum_{j}b_{kj}\lambda_{pj}.
\end{equation}

Define a quiver $\Psi^* Q$ as a copy of $Q$ with $P$ additional frozen vertices corresponding to $w_p$, $p\in[1,P]$. There are $|\delta_{pk}|$ arrows between the frozen vertex $p$ and the vertex $k$ in  $\Psi^*Q$; they are directed from $k$ to $p$ if $\delta_{pk}<0$, and in the opposite direction otherwise. 
We thus get a seed $\Psi^*\Sigma=(\u,\Psi^* Q)$ by attaching $u_k$ to the vertex $k$ and $w_p$ to the frozen 
vertex $p$. This seed is called the {\it regular pullback\/} of $\Sigma$. Depending on the choice of the ground ring, 
it defines a cluster structure on $\AA^{L'}$ or on a quasi-affine variety $U\subset \AA^{L'}$ given by non-vanishing of some of $w_p$.

The following fact is a direct consequence the definition.

\begin{proposition}\label{regpullback}
The $y$-variables for the seed $\Psi^*\Sigma$ coincide with the pullback by $\Psi$ of $y$-variables 
for the seed $\Sigma$.
\end{proposition}

Note that in general pullbacks and mutations of seeds do not commute. If they do, the pullback is called coherent. Sufficient conditions for the existence of a coherent regular pullback are given in~\cite{GSVpullback}.

The construction of the regular pullback can be extended to the case of generalized cluster structures. In this case the quiver $\Psi^* Q$ may have 2-cycles of a special type, namely,  opposite arrows between a vertex of multiplicity greater than one and a frozen vertex. %In the corresponding exchange relation the first and the last terms do not need to be coprime, since they might involve certain powers of the same frozen variable. 
Expression~\eqref{discrep} for the discrepancy at vertices of multiplicity $d_k>1$ is modified by multiplying
the contributions of mutable neighbors by $d_k$ and retaining the contributions of frozen neighbors. In the contrast to the case of a regular pullback of  a cluster structure, the discrepancy $\delta_{pk}$ defines only the difference between the number of arrows between $k$ and $p$ in two opposite directions. To find the actual number of arrows, the additional information is needed.
This additional information consists of a function $\chi_{pk}(r)$, $r\in[0,d_k]$, and a slope $\tau_{pk}$ that are defined 
via the degrees of polynomials $w_p$ in the various terms of the generalized exchange relation~\eqref{genger}, 
see~\cite[Section 4]{GSVpullback} for details. If the coherency conditions are satisfied, the seed $\Psi^*\Sigma=(\u,\Psi^* Q,\Psi^*\P)$ defines an {\it almost-cluster\/} structure.   The first and the last terms of exchange relations in
such structures do not need to be coprime, since they might involve certain powers of the same frozen variable. 
The upper cluster algebra for an almost-cluster structure is defined, similarly to the case of cluster structures, as the intersection of the rings of Laurent polynomials over the ring of polynomials in frozen variables taken over all seeds.

A similar situation occurs when a regular generalized cluster structure on a quasi-affine variety $V$ is lifted to the ambient affine space $\AA^L$. In this case the role of distinguished polynomials $w_p$ is played by irreducible polynomials that define the divisors whose union is the complement of $V$ in $\AA^L$. In~\cite[Theorem 6.2]{GSVpullback} we give sufficient conditions for a regular complete generalized cluster structure on $V$
to have a coherent complete regular pullback to $\AA^L$.

\subsection{Main result}\label{mainresult}
The study of cluster structures on $\G$ compatible with the bra\-cket~\eqref{sklyabragen} for $\g = \sl_n$ 
in \cite{GSVple, GSVuni} revealed that there are two essentially different types of pairs of BD triples
$(\bfG,\bar\bfG)$. 
The condition distinguishing these two types can be formulated in terms of 
the map $\delta=(-w_0)\bar\gamma^*(-w_0)\gamma$: we say that a pair $(\bfG,\bar\bfG)$ violates the aperiodicity 
condition if there exists $\alpha\in\Gamma_1$ and natural $m>0$ such that $\delta^m(\alpha)=\alpha$;
otherwise the pair $(\bfG,\bar\bfG)$ is called {\em aperiodic\/} (see
\cite[Section 3.1]{GSVple} for a more detailed discussion of these two types in terms of acyclicity 
of corresponding BD graphs). Papers~\cite{GSVple, GSVuni} provide a comprehensive study of cluster structures on $SL_n$ compatible with aperiodic pairs of BD triples 
The case of pairs that violate the aperiodicity condition is much more complicated and leads to
generalized cluster structures. The simplest way to get such a pair $(\bfG,\bar\bfG)$ is to choose
$\bar\bfG=\bfG^{w_0}$. In the self-dual case $\bfG=\bfG^{w_0}$ this produces generalised cluster structures 
compatible with the Poisson--Lie structure. The smallest case of this type occurs for $n=6$ and is considered
in detail in~\cite{GSVpest}.  

In this paper we study the pair $(\bfG,\bfG^{w_0})$ for the Cremmer--Gervais BD triple $\bfG=\bfG^{\CG}$. Namely,
let $\Pi=\{\alpha_1,\dots,\alpha_{n-1}\}$ be the set of positive simple roots for the $\sl_n$-case; the Cremmer--Gervais triple is given by $\Gamma_1=\{\alpha_1,\dots,\alpha_{n-2}\}$, 
$\Gamma_2=\{\alpha_2,\dots,\alpha_{n-1}\}$, and $\gamma\:\alpha_i\mapsto\alpha_{i+1}$ for $i\in[1,n-2]$. Note that in this 
case $\bfG^{w_0}=\bfG^*$. Recall that the study of the aperiodic pair $(\bfG^\CG,\bfG^\CG)$ for the Cremmer--Gervais triple provided one of the first examples of exotic cluster structures on $GL_n$, see \cite{GSVPNAS, GSVMem}.
Let $\Poi$ denote the bracket~\eqref{sklyabragen} with $r=r^\bfG$ and $\bar r=r^{\bfG^{w_0}}$ for $\bfG=\bfG^\CG$. We retain the same notation for the lift of this bracket to a Poisson bracket on $GL_n$ for which $\det X$
is a Casimir. The main result of the paper is

\begin{theorem}\label{thm:main}
For any $n\ge4$, there exists a regular complete generalized cluster structure $\G\CC_n$ on $GL_n$ compatible with $\Poi$.
\end{theorem}

\begin{remark}\label{threestruct}
The ground ring for the generalized cluster structure mentioned in the Theorem is $\C[x_{1n},x_{n1},\det X^{\pm1}]$. As usual, the same construction (with evident minor modifications) provides a regular complete
generalized structure on $SL_n$ over the ground ring $\C[x_{1n},x_{n_1}]$. Extending this structure to
$\Mat_n$ (over the ground ring $\C[x_{1n},x_{n1},\det X]$) is somewhat more subtle and will be adressed below in 
Section~\ref{outline}.6.
\end{remark}

\begin{remark} Theorem~\ref{thm:main} remains valid for $n=3$ as well. 
The case $n=3$ is special, and certain constructions that work for $n\ge4$ should be modified. It is discussed in detail in a companion paper~\cite{GSVpullback}.
\end{remark}

In fact, for $n\ge4$, we construct $\G\CC_n$ as an amalgamation of the two cluster structures closely related to the two previously known ones: the generalized cluster
structure on $GL_{n-1}^\dag$ compatible with the bracket~\eqref{dualbra} for $r=r^{\bfG^\CG}$ 
studied in \cite{GV} and 
the cluster structure on rational functions studied in \cite{GSV_Acta}. To make the exposition self-contained, we provide below a concise description of these structures.

\subsection{Generalized cluster structure on $GL_n^\dag$ compatible with the Cremmer--Gervais bracket}
\label{sec:dualcs}
A regular complete generalized cluster structure  on $GL_n^\dag$ compatible with the the bracket~\eqref{dualbra} for arbitrary BD data was built in~\cite{GV}. In this paper, we need a specialization of this construction for the Cremmer--Gervais BD data. Denote by $\Poi^\dag$ the bracket~\eqref{dualbra} for $r=r^{\bfG^\CG}$.  The generalized structure $\G\CC^\dag_n$ 
on $GL_n^\dag$ built in~\cite{GV} is defined by the quiver $Q_n^\dag$ with one special vertex of multiplicity $n$, initial extended cluster $\DD_n$ and a unique coefficient string $P_n$. 

The quiver $Q_n^\dag$  for $n\ge 3$ consists of two equilateral
triangles each comprised of $n(n-1)/2$ vertices. The triangles are placed base to base making a lozenge, see Fig.~\ref{Q5dag}  for the case $n=5$; for the sake of visualization we call them the left triangle and the right triangle. The leftmost vertex of the lozenge is the special vertex, the rightmost one is the frozen vertex, all other vertices are mutable. The vertices in the left triangle are labeled by pairs $(k,l)$ with $1\le k,l\le n-1$, $k+l\le n$, the leftmost vertex being $(1,1)$, $k$ increasing along the upper side of
the triangle and $l$ increasing along its lower side, so that the vertices of the base traversed bottom up are $(m,n-m)$ for 
$1\le m\le n-1$. The vertices in the right triangle are labeled by pairs $\la i,j\ra=\la n+1-k,n+2-k-l\ra$ with 
$1\le k,l\le n-1$, $k+l\le n$,  the rightmost vertex
being $\la n,n\ra$ and corresponding to $k=l=1$, $k$ increasing along the lower side of the triangle and $l$ increasing along its upper side, so that the vertices of the base traversed bottom up are $\la m,2\ra$ for $2\le m\le n$. This numbering is inherited from the numbering of the vertices of the quiver for the generalized cluster structure on the Drinfeld double of
$GL_n$ derived in~\cite{GSVdouble}, which accounts for the difference in the numberings in the left and right triangles.

The arrows in both triangles form a complete triangular mesh with arrows directed south, northeast, and northwest, 
with the omission of arrows along the upper side of the left triangle and the lower side of the right triangle. The two
bases are connected with the arrows $\la m,2\ra\to(m-1,n-m+1)$ for $2\le m\le n$ directed west and $(m-1,n-m+1)\to\la m+1,2\ra$ 
for $2\le m\le n-1$ directed northeast. Additionally, there is a path of length $2(n-2)$ formed by arrows
$(1,l)\to(l+1,1)\to(1,l+1)$ for $1\le l\le n-2$ in the left triangle and a path of the same length formed by arrows
$\la n,j-1\ra\to\la j-1,j-1\ra\to \la n,j\ra$ for $1\le j\le n-2$ in the right triangle. For the sake of better visualization, these two paths are shown with dashed lines 
in~Fig.~\ref{Q5dag}. Finally, there are an additional frozen vertex $(0,0)$, $n-1$ isolated frozen vertices that are not shown 
in Fig.~\ref{Q5dag}, and two additional arrows: from $(1,1)$ in the left triangle to the additional frozen vertex $(0,0)$ and from 
this vertex to $\la 2,2\ra$ in the right triangle.

\begin{figure}[ht]
\begin{center}
\includegraphics[width=12cm]{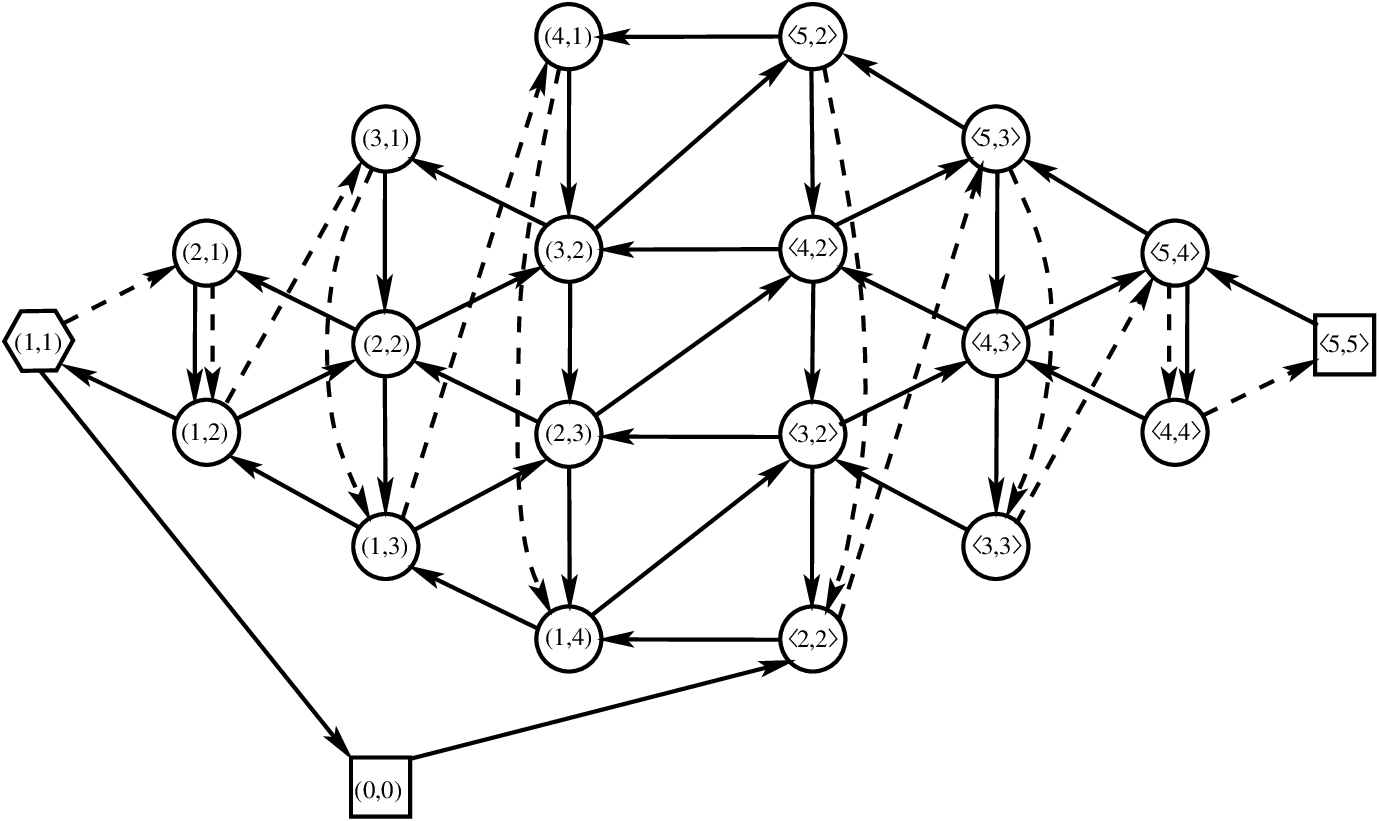}
\caption{Quiver $Q_5^\dag$}
\label{Q5dag}
\end{center}
\end{figure}

Following~\cite[Section~7.2]{GV}, the initial extended cluster $\DD_n$ consists of three subfamilies of functions. 
Let $U$ be an element of $GL_n^\dag$. 
Define $n\times n$ matrices $\Phi_{kl}(U)$, $k,l\in[1,n-1]$, $k+l\le n$, by taking the first $k$ columns of $U^0$, appending the first $l$ columns of $U^1=U$, and then appending the first column of the matrices $U^p$ for $2\le p\le n-k-l+1$. The first subfamily consists of $n(n-1)/2$ functions $\fy_{kl}(U)$ attached to the vertices $(k,l)$ of the left triangle and defined
via
\begin{equation}\label{phidef}
\fy_{kl}(U)=s_{kl}^n\det\Phi_{kl}(U),
\end{equation}
 where the sign $s_{kl}^n$ is given by
\begin{equation}\label{signdef}
s_{kl}^n=\begin{cases}
(-1)^{k(l+1)} \ & \ \text{for $n$ even},\\
(-1)^{(n-1)/2 + k(k-1)/2 + l(l-1)/2} \ & \ \text{for $n$ odd}.
\end{cases}
\end{equation}   
The second subfamily consists of $n(n-1)/2$ functions $g_{ij}(U)$ attached to the vertices $\la i,j\ra $ of the right triangle 
and defined via $g_{i2}(U)=\det U_{[i,n]}^{[2,n-i+2]}$ and 
\begin{equation}\label{gdef}
g_{ij} (U) =
 \sum_{I_1,I_2,\dots, I_{j-2}}
\det U^{I_1}_{[i,n]} \det U_{\gamma^*(I_1)\cup [n-i+j,n]}^{I_2} 
\prod_{m=2}^{j-2}\det U_{\gamma^*(I_m)\cup n}^{I_{m+1}}
\end{equation} 
for $j>2$, where the sum is taken over all subsets $I_1,\dots,I_{j-2}$ of $I_{j-1}=[2,n]$ of sizes $|I_1|=n-i+1$ and $|I_m|=n-j+m$ for
$2\le m\le j-2$ (see Section~\ref{badgproof} below for the derivation of the above expression from the general procedure
for constructing functions $g_{ij}(U)$ for an arbitrary BD data presented in~\cite{GV}).
The third subfamily consists
of $n$ functions $c_r(U)$, $1\le r \le n$, that are just signed coefficients of the characteristic polynomial of $U$:
\begin{equation}\label{cdef}
\det(\one_n+\lambda U)=1+\sum_{r=1}^n (-1)^{r(n-1)}\lambda^r c_r(U).
\end{equation}
Functions $c_r(U)$ for $1\le r\le n-1$ are attached to the isolated vertices, $c_n(U)=\det U$ is attached to the
additional frozen vertex $(0,0)$. The only string in $\PP_n$ is $(1, c_1(U),\allowbreak \dots,c_{n-1}(U),1)$, so that the exchange
relation at the special vertex for the seed $(\DD_n, Q_n^\dag, \PP_n)$ reads
$$
\fy_{11}(U)\fy_{11}'(U)=\sum_{r=0}^n \fy_{21}^r(U)\fy_{12}^{n-r}(U)c_r(U)
$$
with $c_0(U)=1$.

The following theorem is proved in~\cite{GV}.
\begin{theorem} \label{dualgroupcs}
For any $n\ge3$, the initial seed $(\DD_n,Q_n^\dag, \PP_n)$ defines a regular complete generalized cluster structure
$\GCC_n^\dag$ on $GL_n$ compatible with the bracket $\Poi^\dag$.
\end{theorem}

\begin{remark} \label{grfordual}
Note that the ground ring for the corresponding generalized upper cluster algebra $\UC(\GCC_n^\dag)$ is
$\C[g_{nn},\det U^{\pm1}]$.
\end{remark}

Clearly, a (generalized) cluster structure compatible with a Poisson bracket is also compatible with its negative, therefore $\GCC_n^\dag$ is compatible with the bracket $-\Poi^\dag$ as well.

\subsection{Cluster structure on rational functions}\label{sec:todacs} 
Denote by $\RR_n$ the space of  complex rational functions of one 
variable of the form 
$M(\lambda)=q(\lambda)/p(\lambda)$, where $p(\lambda)$ is a monic polynomial of degree  $n\ge2$ and $q(\lambda)$ is a polynomial
of degree at most $n-1$ coprime with $p(\lambda)$. The cluster structure $\CC_n^T$ on $\RR_n$ compatible with the bracket
\begin{equation}\label{todabra}
\{M(\lambda),M(\mu)\}^T=-(\lambda M(\lambda)-\mu M(\mu))\frac{M(\lambda)-M(\mu)}{\lambda-\mu}
\end{equation}
is built in~\cite{GSV_Acta}. Below we describe a slight modification of this construction related to the initial one via
a quasi-isomorphism. The quiver $Q^T_n$ consists of two rows of $n$ vertices each. The vertices in the upper and lower rows are
labeled  from right to left by $(m,+)$ and $(m,-)$, $1\le m\le n$, respectively. The leftmost vertices in each row, $(n,+)$
and $(n,-)$, are frozen, all other vertices are mutable. There are horizontal arrows $(m,+)\to (m-1,+)$ 
and $(m,-)\to (m-1,-)$ 
for $2\le m\le n$, double vertical arrows $(m,-)\Rightarrow(m,+)$, $1\le m\le n-1$, directed upwards,  double inclined arrows $(m,+)\Rightarrow(m+1,-)$, $1\le m\le n-2$, directed downwards, and an ordinary 
inclined arrow $(n-1,+)\to(n,-)$ directed downwards, see Fig.~\ref{Q5T}.

\begin{figure}[ht]
\begin{center}
\includegraphics[width=9cm]{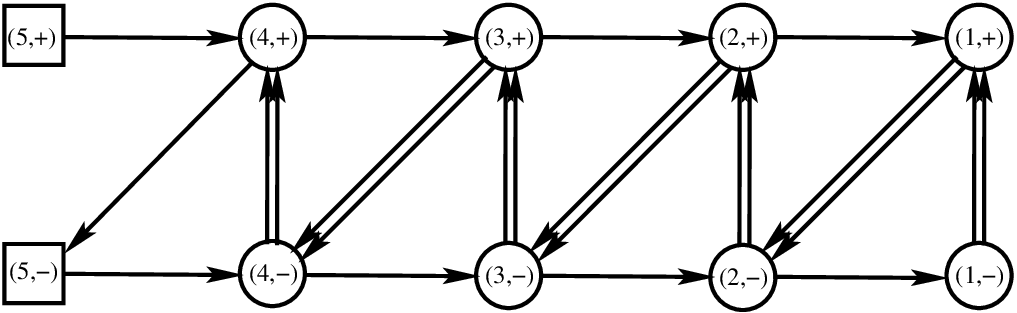}
\caption{Quiver $Q_5^T$}
\label{Q5T}
\end{center}
\end{figure}

To build the initial extended cluster $\TE_n$, 
write the Laurent expansion
$$
M(\lambda)=\sum_{i=0}^\infty \frac{h_i}{\lambda^{i+1}}
$$
and define for each $m\in[1,n]$ two $m\times m$ Hankel matrices $H_m^-=(h_{\alpha+\beta-2})_{\alpha,\beta=1}^m$ and
$H_m^+=(h_{\alpha+\beta-1})_{\alpha,\beta=1}^m$. The family $\TE_n$ consists of $2n-1$ functions $t_m^-=\det H_m^-$, $1\le m\le n$, and 
$t_m^+=\det H_m^+$, $1\le m\le n-1$, assigned to vertices $(m,-)$ and $(m,+)$, respectively, and the
function $t_n^+/t_n^-$ assigned to vertex $(n,+)$. Note that the bracket~\eqref{todabra} in terms of 
the coefficients $h_i$ reads
\begin{equation}\label{todabramoment}
\{h_i,h_j\}^T=\sum_{k=i}^{j-1}h_{k+1}h_{i+j-k-1}
\end{equation}
for $i<j$.

\begin{remark} 1. The quiver described above should be compared to the quiver (5.28) in~\cite{GSV_Acta} 
reflected about its center. The frozen variable attached to the frozen vertex $(n,+)$ differs from the one attached to the corresponding frozen vertex in the quiver~(5.28), which is
equal to $t_n^-/t_n^+$. This difference accounts for the reversal of the horizontal arrow 
$(n,+)\to(n-1,+)$.

2. The definition of the space $\RR_n$ in~\cite{GSV_Acta} contains an additional technical requirement $p(0)\ne0$, which is not needed in our situation. In~\cite{GSV_Acta}, where the cluster structure on $\RR_n$ was used to give a cluster theory interpretation to B\"acklund--Darboux transformations between Coxeter--Toda flows on $GL_n$, it was used to distinguish
rational functions that are scalar multiples of Weyl functions of elements of Coxeter double Bruhat cells.

3. Expression~\eqref{todabramoment} for the bracket is borrowed from~\cite{FaGe}; 
%with the following two modifications: the negative sign is added in the right hand side, which corresponds to the negative sign in~\eqref{todabra}, and 
a misprint in the upper summation limit is corrected.
\end{remark}

The construction presented above can be easily extended to the space $\bar\RR_n$ of  complex rational functions of one 
variable of the form  $\bar M(\lambda)=\bar q(\lambda)/p(\lambda)$, where $p(\lambda)$ is the same as before and $\bar q(\lambda)$ is a polynomial of degree at most $n$ coprime with $p(\lambda)$. The Laurent expansion for $\bar M(\lambda)$ reads
\begin{equation}\label{moments}
\bar M(\lambda)=\sum_{i=0}^\infty\frac{\bar h_i}{\lambda^i}.
\end{equation}
Clearly, for $i>0$ coefficients $\bar h_i$ coincide with $h_{i-1}$ computed for the function $M(\lambda)=\bar M(\lambda)-\bar h_0\in\RR_n$. We extend the bracket $\Poi^T$ to $\bar h_i$ via
\begin{equation}\label{bartodabramoment}
\{\bar h_i,\bar h_j\}^T=\sum_{k=i}^{j-1}\bar h_{k+1}\bar h_{i+j-k-1}
\end{equation}
for $i<j$, which immediately yields
\begin{equation}\label{bartodabra}
\{\bar M(\lambda),\bar M(\mu)\}^T=-(\lambda \bar M(\mu)-\mu \bar M(\lambda))\frac{\bar M(\lambda)-\bar M(\mu)}{\lambda-\mu}.
\end{equation}

The quiver $\bar Q_n^T$ contains $n$ vertices $(m,+)$, $1\le m\le n$, and $n+1$ vertices $(m,-)$, $1\le m\le n+1$. The two frozen vertices are $(n,+)$ and $(n+1,-)$. All arrows of the quiver $Q_n^T$ are present in
$\bar Q_n^T$, except for the inclined arrow $(n-1,+)\to(n,-)$ which is replaced by a double inclined arrow
$(n-1,+)\Rightarrow(n,-)$. Additionally, there is a vertical arrow $(n,-)\to(n,+)$ and a horizontal arrow
$(n+1,-)\to(n,-)$. We draw $\bar Q_n^T$ in a lozenge-like way, see Fig.~\ref{barQ5T}.

\begin{figure}[ht]
\begin{center}
\includegraphics[width=11cm]{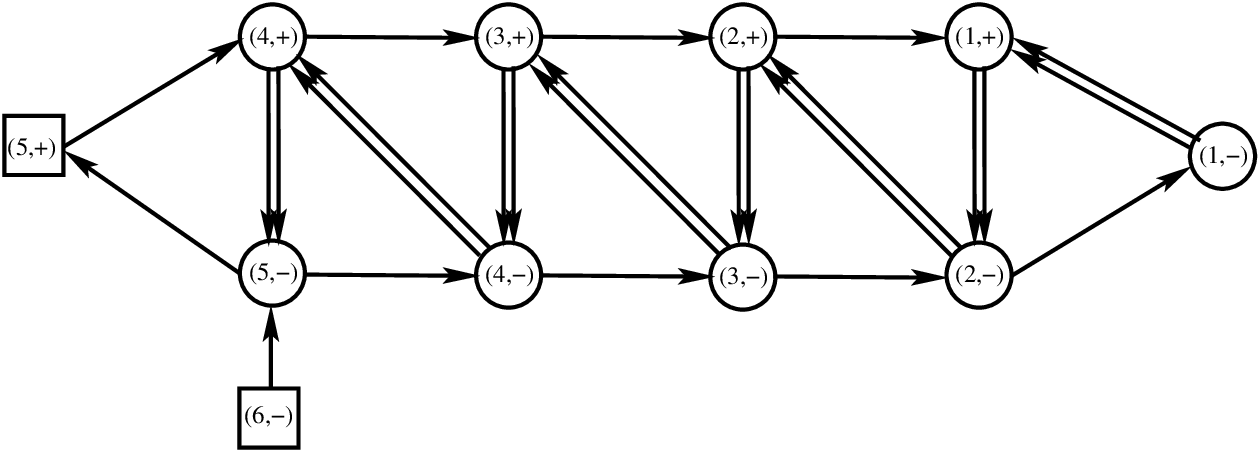}
\caption{Quiver $\bar Q_5^T$}
\label{barQ5T}
\end{center}
\end{figure}

Similarly to above, we define Hankel matrices $\bar H_m^-=(\bar h_{\alpha+\beta-2})_{\alpha,\beta=1}^m$ and 
$\bar H_m^+=(\bar h_{\alpha+\beta-1})_{\alpha,\beta=1}^m$ and the corresponding determinants $\bar t_m^-=
\det \bar H_m^-$, $m\in [1,n+1]$, and $\bar t_m^+=(-1)^{m}\det\bar H_m^+$, $m\in[1,n]$. The initial cluster $\bar\TE_n$ consists of $2n$ functions $\bar t_m^\pm$, 
$1\le m\le n$, assigned to vertices $(m,+)$ and $(m,-)$, respectively, and the function 
$\bar t_{n+1}^-/\bar t_n^+$ assigned to vertex $(n+1,-)$.

 \begin{theorem}\label{todacs} 
For any $n\ge 2$, the initial seed $(\bar\TE_n, \bar Q_n^T)$ defines a regular complete cluster structure $\bar\CC_n^T$ on 
$\bar \RR_n$ compatible with the bracket $\Poi^T$.
\end{theorem}

\begin{remark} \label{grfortoda}
Note that the ground ring for the corresponding upper cluster algebra $\UC(\bar\CC_n^T)$ is
$\C[\bar t_{n+1}^-/\bar t_n^+,(\bar t_n^+)^{\pm1}]$.
\end{remark}

The proof of this theorem is given in the Appendix.

\subsection{Outline of the proof}\label{outline} In what follows in this section we assume that $n\ge4$, unless explicitly stated otherwise. To prove Theorem~\ref{thm:main} we perform the following steps.

1. We build two maps: $\Psi' : GL_n\to GL_{n-1}$ and $\Psi'' : GL_n\to\bar\RR_{n-1}$ and prove the following theorem.

\begin{theorem}\label{twomaps}
{\rm (i)}  The map $\Psi': (GL_n,\Poi)\to \widetilde{GL}_{n-1}^{\dag}=(GL_{n-1},-\Poi^\dag)$ is Poisson.

{\rm (ii)} The map $\Psi'': (GL_n,\Poi)\to \bar\RR_{n-1}^T=(\bar\RR_{n-1},\Poi^T)$ is Poisson.

{\rm (iii)} The combined map $\Psi : (GL_n,\Poi)\to \widetilde{GL}_{n-1}^{\dag}\times_\C \bar\RR_{n-1}^T$ with the fibre product defined by functions $\det U$ for $U\in GL_{n-1}$ and $\bar q(0)$ for $\bar q(\lambda)/p(\lambda)\in
\bar\RR_{n-1}$ is Poisson.
\end{theorem}

2. We consider the cluster structures $\GCC_{n-1}^\dag$ and $\bar\CC_{n-1}^T$ described in Sections~\ref{sec:dualcs} and~\ref{sec:todacs}, respectively. The corresponding quivers are shown schematically 
in Fig.~\ref{twoquivers}.

\begin{figure}[ht]
\begin{center}
\includegraphics[width=12cm]{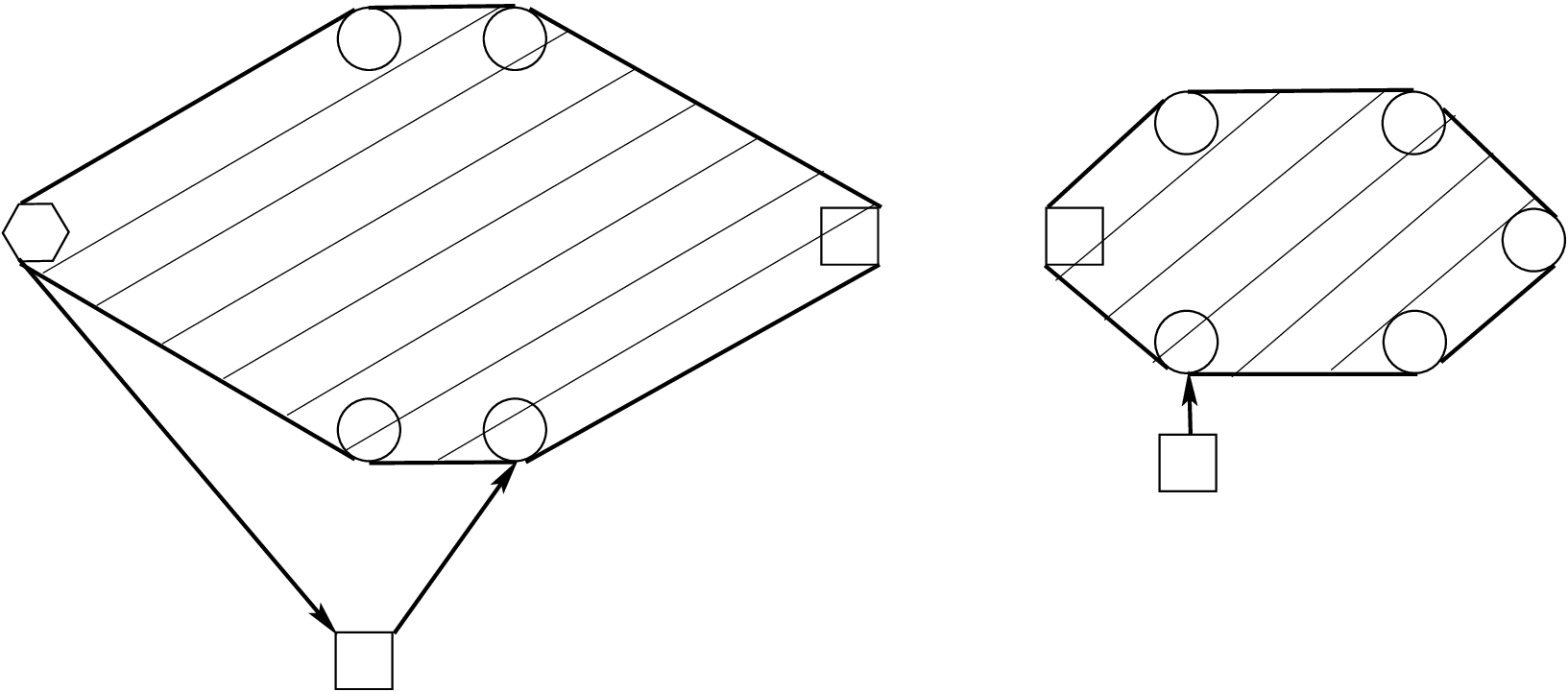}
\caption{Quivers $Q_{n-1}^\dag$ (on the left) and $\bar Q_{n-1}^T$ (on the right)}  
\label{twoquivers}
\end{center}
\end{figure}

To find the functions $\fy_{kl}(\Psi'(X))$, $g_{ij}(\Psi'(X))$, $c_p(\Psi'(X))$, 
and $\bar t_m^\pm(\Psi''(X)$ for $X\in GL_n$ we need the following objects.

First, we define $(n-1)(n-2)/2$ matrices $\Nphi_{kl}(X)$ for $k,l\in[1,n-2]$ and $k+l\le n-1$ 
of sizes $(n-k-l)(n+1)\times (n-k-l)(n+1)$. Next, we define $(n-1)(n-2)/2$ matrices $\NG_{ij}(X)$ 
for $i\in[2,n-1]$ and $j\in [2,i]$ of sizes $(n+j-1)\times (n+j-1)$. Finally, we define three more matrices:
$\NA(X)$ and $\NB(X)$ of size $(n+1)\times(n+1)$ and $\NF(X)$ of size $(2n-1)\times(2n-1)$. The entries of all the
matrices are either zeros or entries of $X$. Explicit constructions of the matrices are given in Section~\ref{explicit} below.

Put $\phi_{kl}(X)=\det\Nphi_{kl}(X)$, $\eg_{ij}(X)=\det \NG_{ij}(X)$, and define $\ec_p(X)$ as the 
coefficient at $\lambda^p$ of $\det(\NA(X)+\lambda \NB(X))$ with the sign $(-1)^{pn}$ and $\ef_m(X)$ as the $m$-th main trailing minor of $\NF(X)$ starting from the lower right corner. It follows from above that all these functions are polynomials in the entries of $X$.

\begin{theorem}\label{expressions}
{\rm (i)} Functions $\fy_{kl}(\Psi'(X))$ for $k,l\in[1,n-2]$ and $k+l\le n-1$ can be written as
\begin{equation}\label{badphi}
 \fy_{kl}(\Psi'(X))=\es_{kl}\frac{\phi_{kl}(X)}{(x_{1n}\det X)^{n-k-l}}
\end{equation}
with
\[
\es_{kl}=\begin{cases} (-1)^{(n-2)/2 +k(k-1)/2+l(l-1)/2}\ & \ \text{for $n$ even},\\
(-1)^{(n-1)/2+k(k-1)/2+l(l-1)/2}    \ & \ \text{for $n$ odd}.
\end{cases}		
\]

{\rm (ii)} Functions $g_{ij}(\Psi'(X))$ for $i\in[2,n-1]$ and $j\in [2,i]$  can be written as
\begin{equation}\label{badg}
g_{ij}(\Psi'(X))=\hat\es_{ij}\frac{\eg_{ij}(X)}{x_{1n}^{j-1}\det X}
\end{equation}
with $\hat\es_{ij}=(-1)^{(n+j+1)(i+j)}$.
%In particular, $\esigma_{n-1,n-1}=(-1)^{n-1}$.

{\rm (iii)}  Functions $c_r(\Psi'(X))$  can be written as
\begin{equation}\label{badc}
c_r(\Psi'(X))=\frac{\ec_r(X)}{x_{1n}\det X}
\end{equation}
for $r\in[1,n-1]$. In particular, $c_{n-1}(\Psi'(X))=\det \Psi'(X)=x_{n1}/x_{1n}$.

{\rm (iv)} Functions $\bar t_m^\pm(\Psi''(X))$  can be written as
\begin{equation}\label{badbart}
\bar t_m^-(\Psi''(X))=\frac{\ef_{2m-1}(X)}{x_{1n}^{2m-1}},\qquad
\bar t_m^+(\Psi''(X))=\frac{\ef_{2m}(X)}{x_{1n}^{2m}},
\end{equation}
where $m\in [1,n]$ in the first case and $m\in [1,n-1]$ in the second case. 
In particular, $\bar t_n^-(\Psi''(X))/\bar t_{n-1}^+(\Psi''(X))=x_{n1}/x_{1n}$.
\end{theorem}

\begin{remark} Note that the map $\Psi$ is defined for $n=3$ as well, and Theorem~\ref{expressions} remains valid 
in this case.
\end{remark}

The following fact is checked via a straightforward calculation at each vertex.

\begin{corollary}\label{trivialsigns}
For any mutable vertex $v$ of $Q_{n-1}^\dag$, the product of the signs $\es_{kl}$ and $\hat\es_{ij}$ over all
vertices $(k,l)$ and $\langle i,j\rangle$ adjacent to $v$ equals one.
\end{corollary}

Consequently, we can modify $\DD_{n-1}$ via replacing each $\fy_{kl}$ by $\es_{kl}\fy_{kl}$ and each $g_{ij}$ by
$\hat\es_{ij}g_{ij}$, and Theorem~\ref{dualgroupcs} remains valid for 
the modified cluster $\widetilde{\DD}_{n-1}$.

3. We build regular pullbacks of the seeds $(\widetilde{\DD}_{n-1},Q_{n-1}^\dag,\PP_{n-1})$ and $(\bar\TE_{n-1},\bar Q_{n-1}^T)$ under $\Psi'$ and $\Psi''$, respectively, based on 
explicit expressions given in Theorem~\ref{expressions} and the constructions described in Section~\ref{seedpullback}.

\begin{proposition}\label{readytoglue}
{\rm (i)} The quiver $\Psi'^*Q_{n-1}^\dag$ is obtained from $Q_{n-1}^\dag$ by adding two frozen vertices $A$ and $B$ and four arrows: $(n-2,1)\to A\to (1,1)$ and $(n-2,1)\to B\to\la n-1,2\ra$.

{\rm(ii)} The quiver $\Psi''^*\bar Q_{n-1}^T$ is obtained from $\bar Q_{n-1}^T$ by adding an additional frozen vertex $A$ and an arrow $A\to (1,-)$.
\end{proposition}

The corresponding quivers are shown schematically in Fig.~\ref{twopullbacks}. The vertices of the quivers 
$\Psi'^*Q_{n-1}^\dag$ and $\Psi''^*\bar Q_{n-1}^T$ retain the same labeling as in $Q_{n-1}^\dag$ and $\bar Q_{n-1}^T$, respectively. Cluster variables assigned to the vertices of $\Psi'^*Q_{n-1}^\dag$ are $\phi_{kl}$ and $\eg_{ij}$, those assigned to the vertices of $\Psi''^*\bar Q_{n-1}^T$ are $\ef_m$. Note that since we consider the pullback to $GL_n$, the coefficients of the generalized exchange relation at $(1,1)$ are $\ec_r(X)/\det X$.

\begin{figure}[ht]
\begin{center}
\includegraphics[width=12cm]{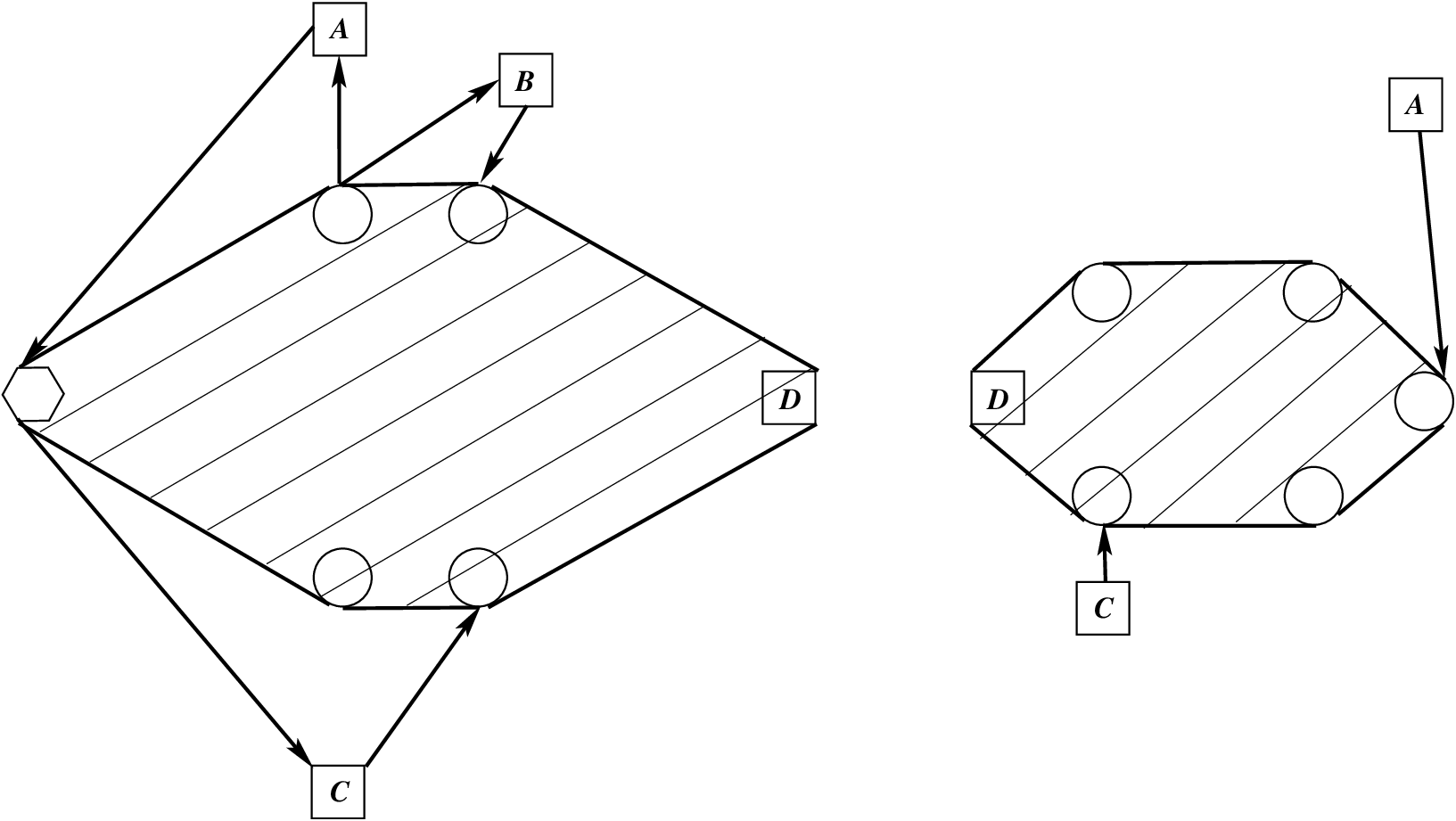}
\caption{Quivers $\Psi'^*Q_{n-1}^\dag$ (on the left) and $\Psi''^*\bar Q_{n-1}^T$ (on the right)}  
\label{twopullbacks}
\end{center}
\end{figure}

4. Observe that the same functions are assigned to vertices in both quivers marked by the same letters: 
$x_{1n}$ to vertex $A$, $x_{n1}$ to vertex $C$ and $\eg_{n-1,n-1}(X)=\ef_{2n-2}(X)$ to vertex $D$;
to justify the latter claim we note that $\NG_{n-1,n-1}(X)=\NF(X)_{[2,2n-1]}^{[2,2n-1]}$ (see the definitions of these matrices in Section~\ref{explicit} below). 
Therefore, we can glue the two quivers in Fig.~\ref{twopullbacks}
 into one quiver via identifying these three pairs of frozen vertices. 

Further, frozen vertices in a regular complete generalized cluster structure $\G\CC$ compatible with a Poisson bracket
$\Poi_{\G\CC}$ are characterized by the following property (see Section~\ref{prooffrozenchar} for the proof).

\begin{proposition}\label{frozenchar}
Let $\G\CC$ be a regular complete generalized cluster structure on a variety $V$ compatible with a Poisson bracket
$\Poi_{\G\CC}$. A cluster variable $f$ in $\G\CC$ is frozen if and only if $\tilde g=\{f,g\}_{\G\CC}/f$ is regular
for any regular function $g$.
\end{proposition}

Guided by Proposition~\ref{frozenchar}, we check that its condition holds for the variables $x_{1n}$ and $x_{n1}$, see Section~\ref{prooffrozenchar} below. However, we observe (by using computer-assisted computations) that for $n=4$, $\{\ef_6(X),x_{23}\}/\ef_6(X)$ is not a regular function.
This suggests that while $x_{1n}$, $x_{n1}$, and $\det X$ should remain frozen (recall that by construction $\det X$ is a Casimir of the bracket $\Poi$ on $GL_n$),
one cannot get a regular complete generalized cluster structure without unfreezing vertex $D$.
Indeed, this fact is proved later for arbitrary $n$ in Section~\ref{regcomplet}, see Remark~\ref{whyunfreeze}. 
In order to get such a structure, we unfreeze vertex $D$. But then the degrees of the monomials in the would be
exchange relation are out of balance. To fix the imbalance, we have to add arrows to and/or from $D$. Clearly, an arrow between $D$ and a mutable vertex would break the balance at this vertex, so the only possible new arrows are between $D$ and other frozen vertices, that is, $A$, $B$, and $C$. It is easy to see that there exists at most one choice of additional arrows that guarantees that after the mutation at $D$ the new function is regular. The following statement proven in Section~\ref{dexchangeproof} dictates this choice. 

\begin{proposition}\label{Dexchange}
Let $\eg(X)=\det\NG_{n-1,n-2}(X)_{[2,2n-3]}^{[2,2n-3]}$, then
\begin{equation}\label{dexchange}
\eg_{n-1,n-1}(X)\eg(X)=\ef_{2n-3}(X)\eg_{n-2,n-2}(X)+x_{n1}\ef_{2n-4}(X)\eg_{n-1,n-2}(X).
\end{equation}
\end{proposition}

Consequently, we connect $D$ with an additional arrow to vertex $C$.
The corresponding quiver $Q_n$ is shown schematically in
Fig.~\ref{onepullback}.

The initial cluster $\FF_n$ is defined as follows: functions $\phi_{kl}(X)$ are assigned 
to vertices $(k,l)$, functions $\eg_{ij}(X)$ are assigned to vertices $\la i,j\ra$, functions $\ef_{2i}(X)$
are assigned to vertices $(i,+)$, functions $\ef_{2i-1}(X)$ are assigned to vertices $(i,-1)$. Additionally, $\FF_n$ contains frozen variables $x_{1n}$ assigned to $A$, $\det X$ assigned to $B$ and $x_{n1}$ assigned to $C$. 
Finally, there are isolated frozen variables $\ec_r(X)$ that enter the only string $\PP_n=(1, \ec_1(X)/\det X, \dots,\ec_{n-2}(X)/\det X,1)$,  so that the generalized exchange relation reads
\begin{equation}\label{ger}
\phi_{11}(X)\phi'_{11}(X)=x_{1n}\phi_{12}^{n-1}(X)+\sum_{r=1}^{n-2}\frac{\ec_r(X)}{\det X}\phi_{21}^r(X)\phi_{12}^{n-1-r}(X)
+x_{n1}\phi_{21}^{n-1}(X).
\end{equation}

\begin{figure}[ht]
\begin{center}
\includegraphics[width=10cm]{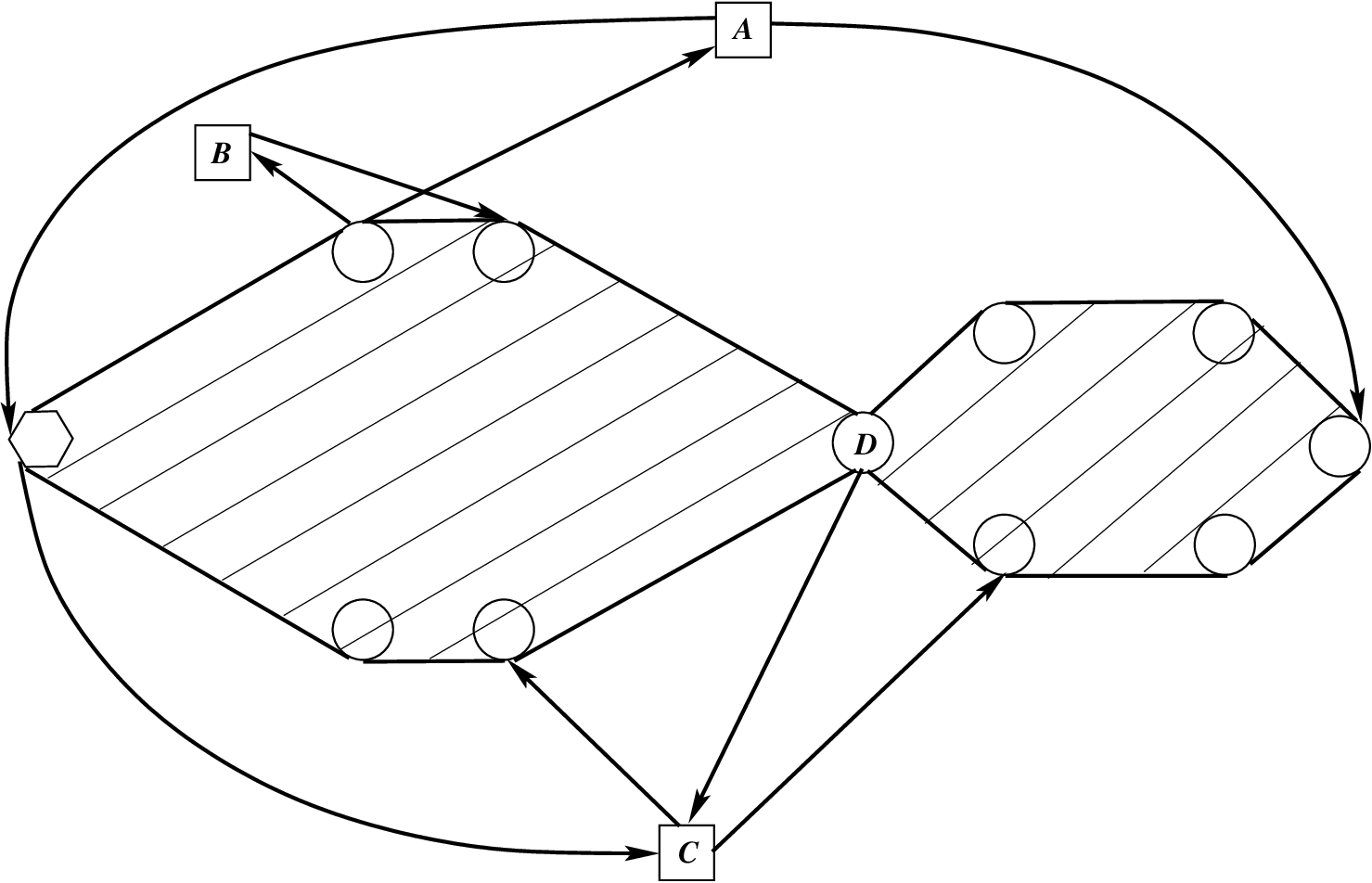}
\caption{Quiver $Q_n$: 
$x_{1n}$ assigned to $A$, $\det X$ assigned to $B$, $x_{n1}$ assigned to $C$, $\eg_{n-1,n-1}(X)=\ef_{2n-2}(X)$ assigned to $D$}
\label{onepullback} 
\end{center}
\end{figure}

We can now reformulate Theorem~\ref{thm:main} as follows.

\begin{theorem}\label{initialseed}
The generalized cluster structure $\G\CC_n=\G\CC(\FF_n,Q_n,\PP_n)$ is a regular complete generalized cluster structure on $GL_n$ compatible with $\Poi$.
\end{theorem}

\begin{remark} An analog of Theorem~\ref{initialseed} is valid for $n=3$. In fact, the components $\FF_3$ and $\PP_3$ of the initial seed are defined by the same formulas as above; the quiver $Q_3$ 
is different, see~\cite{GSVpullback} for details.
\end{remark}

5. For a diagonal matrix $D=\diag(d_1,\dots,d_n)$ denote by $t^D$ the diagonal matrix $\diag(t^{d_1},\dots,t^{d_n})$.
Let $D_n=\diag(n,\dots,2,1)$ and $\bar D_n=\diag(1,2,\dots,n)$.
We say that a function $g$ is homogeneous with respect to the action $X \mapsto t^{ D_n} X s^{\bar D_n}$ if
\begin{equation}
\label{eq:T-hom}
g(t^{D_n} X s^{\bar D_n}) = t^{\xi_g} s^{\bar\xi_g} g(X),
\end{equation}
where numbers $\xi_g$ and $\bar\xi_g$ depend only on $g$. 

The proof of the compatibility statement in Theorem~\ref{initialseed} is based 
on~\cite[Proposition~2.1]{GSVnewdouble} and employs homogeneity of the cluster variables in the initial cluster 
$\FF_n$, see Section~\ref{compatibility} below for details. As an immediate corollary we get that the extended exchange matrix has a full rank (see~\cite[Proposition~2.5]{GSVdouble}). 

Another corollary is the following statement.

\begin{proposition}
\label{prop:torus}
Every cluster variable in $\G\CC_{n}$ is a homogeneous function in matrix entries of $X$ and is also homogeneous with respect to the action $X \mapsto t^{D_n} X s^{\bar D_n}$.
\end{proposition}

\begin{remark} In fact, Proposition~\ref{prop:torus} describes the global toric action of $(\C^*)^2$ on $\G\CC_n$, 
see~\cite[Section 2.1]{GSVdouble} for relevant definitions. In line with the general conjecture, the rank of this action equals $|\Pi\setminus\Gamma_1|+ |\Pi\setminus\bar\Gamma_1|=2$.
\end{remark}

To proceed with the proof we invoke the following modification of~\cite[Proposition 2.1]{GSVMem}.

\begin{proposition}\label{regcomplete}
Let $V$ be a Zariski open subset in $\C^{n+m}$ and $\G\CC$ be a generalized cluster structure in $\C(V)$ with
$n$ mutable and $m$ frozen variables such that

{\rm (i)} the initial extended cluster consists of regular functions;

{\rm (ii)} every mutation of the initial seed produces a regular function;

{\rm (iii)} every regular function on $V$ belongs to the upper cluster algebra $\UC_\C(\G\CC)$

Then $\G\CC$ is a regular complete generalized cluster structure. 
\end{proposition}

Condition (i) follows from the construction above. Condition (ii) for all mutable vertices of $Q_n$ distinct from vertex $D$ follows immediately from Theorems~\ref{dualgroupcs} and~\ref{todacs}. Condition (ii) at vertex $D$ is
verified by Proposition~\ref{Dexchange}. 

To check condition (iii) in Proposition~\ref{regcomplete}, we use the following easy modification
of~\cite[Lemma 8.3]{GSVMem}.

\begin{proposition}\label{twolaurent} Let $R\subset \C(V)$ be a unique factorization domain and $F\in\C(V)$.
If there exist two Laurent representations $F=F_1/f_1=F_2/f_2$ such that $F_1$ and $F_2$ belong to the upper cluster algebra $\UC_\C \subseteq \C(V)$ over $R$ and $f_1$ and $f_2$ are monomials in cluster variables that are coprime in $\C[V]$, 
then $F$ itself belongs to $\UC_\C$.
\end{proposition}

We start with the following statement. 
%Let $U=\Psi'(X)$ and $\bar M=\Psi''(X)$. 

\begin{theorem}
\label{thm:InitLaurent}
%Each matrix entry $x_{ij}$, $i\in[2,\ldots, n-1]$, $j\in[1,\ldots, n]$, is, up to a multiplication by a monomial in 
%$x_{n1}$, $x_{1n}$, and $\ef_{2n-2}(X)$, a polynomial in matrix entries of $U$ and coefficients of $\bar M$.
Every matrix entry of $X$ can be expressed as a Laurent polynomial in terms of cluster variables in the initial seed 
$(Q_n,\FF_n,\PP_n)$ of $\G\CC_n=\G\CC(Q_n,\FF_n,\PP_n)$ as well as in seeds obtained from the initial one by a single mutation at any vertex other than vertices $\langle 2, 2\rangle,\ldots, \langle n-1,n-1\rangle = (n-1, +)$.
\end{theorem}

Concequently, we get a Laurent representation for all matrix entries of $X$.   
To finish the proof of completeness, we need to get another Laurent representation in some seed that does not contain 
$\ef_{2n-2}(X)$ as a  cluster variable. We achieve this goal via a long sequence of mutations that starts with a mutation at $D$ and results in a seed which evokes parallels to our construction in \cite{GSVuni} and in which  Laurentness can be established by combining properties of rational Poisson maps  in \cite{GSVuni} with
those of normal forms under the conjugation by triangular elements that we studied  in \cite{GSVdouble}.

To this end, we define yet another family of $n(n-1)/2-1$ matrices $\NK_{ij}(X)$ for $1\le j<i\le n$, 
$(i,j)\ne(n,1)$, of size $(n-j+2)\times(n-j+2)$ whose entries are either zeros or entries of $X$, and put
$\ek_{ij}(X)=\det \NK_{ij}(X)$. Additionally, we put $\ek_{ii}(X)=\det X_{[i,n]}^{[i,n]}$ for $i\in [1,n]$, so that
$\ek_{11}(X)=\det X$.

\begin{theorem}
\label{fish_to_cuttlefish}
{\rm (i)} There exists a sequence of mutations $\Ws_n$ that starts with the initial seed in $\G\CC_{n}$ and results in a seed with the set of mutable cluster variables given by 
\[
\{\phi_{kl}(X): k,l \geq 1, \ k+l \leq n-1 \}\cup \{\ek_{ij}(X) : 1\le j \leq i \leq n, (i,j)\ne (1,1), (n,1) \}.
\]

{\rm (ii)} All matrix entries of $X$ are Laurent polynomials in terms of this seed that do not contain any of the frozen variables in the denominator.
\end{theorem}

Theorem \ref{fish_to_cuttlefish} allows us to finish the proof of regularity and completeness for $\G\CC_{n}$.

6. To extend the presented construction of a generalized cluster structure to $\Mat_n$ as mentioned in Remark~\ref{threestruct}, we use once again the regular pullback construction described in~\cite{GSVpullback}. The resulting  almost-cluster structure is defined by 
%a quiver that allows 2-cycles of a special type, namely,  opposite arrows between a vertex of multiplicity greater than one and a frozen vertex. In the corresponding exchange relation the first and the last terms do not need to be coprime, since they might involve certain powers of the same frozen variable. Mutations of such quivers are studied in~\cite{GSVpullback}. 
the quiver $\widehat Q_n$ obtained from $Q_n$ by adding two opposite arrows: one from $(1,1)$ to $B$, and the other from $B$ to $(1,1)$, see Fig.~\ref{almostq}. 
Mutations of $\widehat Q_n$ are performed as explained 
in~\cite[Remark 4.6]{GSVpullback}, namely, if usual mutation rules produce $\Delta\ne0$ arrows between $(1,1)$ and $B$, then the quiver after the mutation has
%pair of the opposite arrows is replaced by 
$\Delta$ arrows in the corresponding direction and no arrows in the opposite direction, otherwise the quiver has a pair of opposite arrows between these vertices. For example, mutation at $(1,1)$ retains the pair of the opposite arrows and adds arrows from $B$ to $(2,1)$ and from $(1,2)$ to $B$. A subsequent mutation at $(2,1)$ replaces the pair of opposite arrows by $n-1$ arrows from $B$ to $(1,1)$. 
The functions attached to the vertices of $\widehat Q_n$ are exactly the same as for $Q_n$. The only string $\widehat\P$ is given by $\widehat\P=(1,\ec_1(X),\dots,\ec_{n-2}(X),1)$, so that~\eqref{ger} is replaced by
\begin{multline*}
\phi_{11}(X)\hat\phi'_{11}(X)\\
=\det Xx_{1n}\phi_{12}^{n-1}(X)+\sum_{r=1}^{n-2}\ec_r(X)\phi_{21}^r(X)\phi_{12}^{n-1-r}(X)
+\det Xx_{n1}\phi_{21}^{n-1}(X).
\end{multline*}
%The upper cluster algebra for an almost-cluster structure is defined, similarly to the case of cluster structures, as the intersection of the rings of Laurent polynomials over the ring of polynomials in frozen variables taken over all seeds.

\begin{figure}[ht]
\begin{center}
\includegraphics[width=10cm]{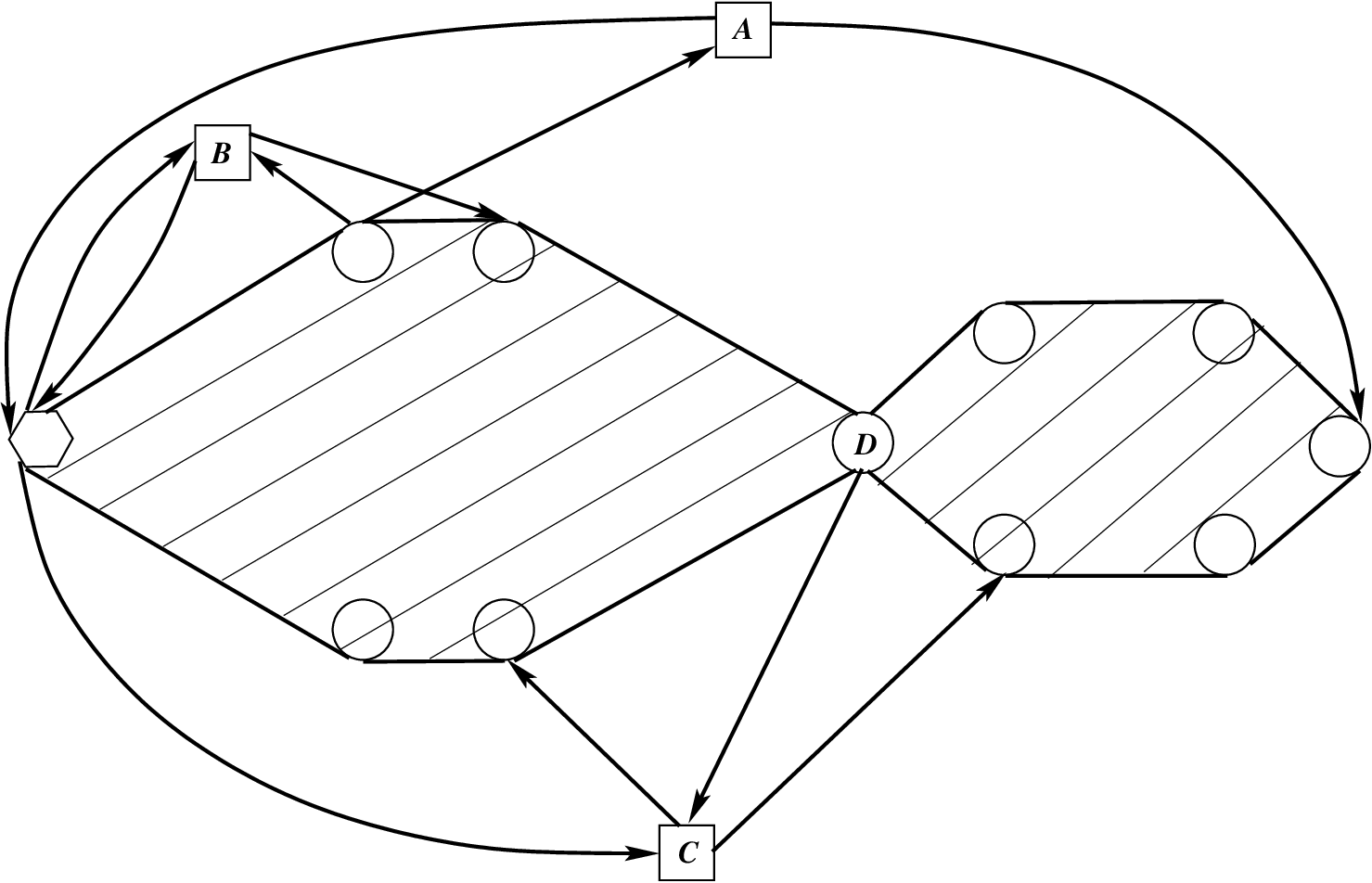}
\caption{Quiver $\widehat Q_n$}
\label{almostq} 
\end{center}
\end{figure}

\begin{theorem}\label{initialalmostseed}
The almost-cluster structure $\widehat{\G\CC}_n=\widehat{\G\CC}(\FF_n,\widehat Q_n,\widehat\PP_n)$ is a regular complete  almost-cluster structure on $Mat_n$ compatible with $\Poi$.
\end{theorem}

\section{Map $\Psi$ and the proof of Theorem~\ref{twomaps}}\label{sec:twomaps}
\subsection{Map $\Psi'$ and the proof of Theorem~\ref{twomaps}(i)}\label{twomapsiproof}
As it was explained in detail in~\cite{GSVple}, given a pair of BD-triples, pairs $(\Gamma_1^\er,\Gamma_1^\ec)$ and $(\Gamma_2^\er,\Gamma_2^\ec)$ define a partition of a matrix representing an element of $GL_n$ into blocks called $X$-blocks and $Y$-blocks, 
respectively, and maps $\gammar$ and $\gammac$ define gluing of the blocks into compound matrices having a staircase shape. When the pair of
triples is aperiodic, the corresponding compound matrices are finite, and their trailing minors form the initial cluster for the cluster structure compatible with the Poisson bracket defined by the pair of triples. If the aperiodicity condition is violated, the corresponding compound matrices are infinite, and arising cluster structures are generalized. An approach to the treatment
of this case is suggested in~\cite{GSVpest} based on utilizing a periodic staircase structure of the obtained infinite compound matrices. In~\cite[Section 6]{GSVpest} we studied the simplest example of such a generalized cluster structure for $GL_6$. While the validity of this approach for the general case remains conjectural, it provides a necessary intuition.

In our case, there is only one $X$-block and only one $Y$-block, and they both coincide with the initial matrix. The
blocks are glued into an infinite compound matrix as shown in Fig.~\ref{pbs}. The periodic staircase structure is indicated in the same figure by dashed lines. The blocks of the periodic staircase structure are $(n+1)\times (n+1)$ matrices $\NA=\NA(X)$ and $\NB=\NB(X)$ so that $\NA_{[1,n]}^{[3,n+1]}=X^{[1,n-1]}$, $\NB_{[1,n]}^{[1,1]}=X^{[n,n]}$, 
and $\NB_{[2,n+1]}^{[2,n+1]}=X$; all other entries of $\NA$ and $\NB$ are zeros.

\begin{figure}[ht]
\begin{center}
\includegraphics[width=5cm]{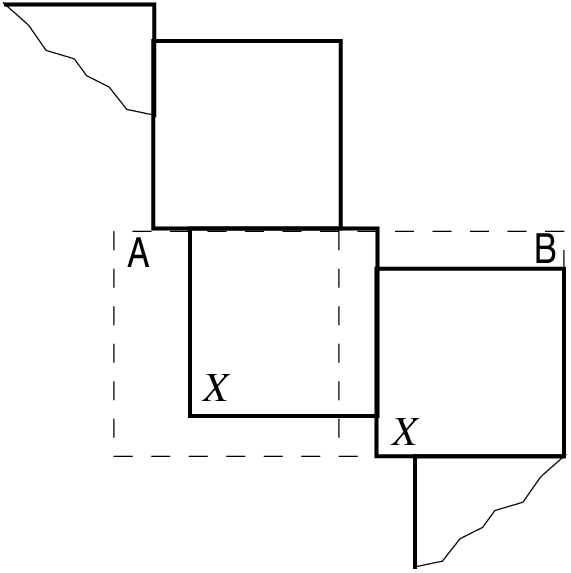}
\caption{Periodic staircase structure for the infinite compound matrix}
\label{pbs}
\end{center}
\end{figure}

By~\cite[Theorem 3.2]{GSVpest}, we expect that the coefficients of the generalized exchange relation are
the coefficients of the characteristic polynomial $\det (\lambda\NA+\NB)$. Assuming $x_{1n}x_{n1}\ne0$, one can use elementary row operations to rewrire it as
\begin{equation}\label{ABvsU}
\begin{aligned}
\det(\lambda\NA+\NB)&=\det\left(\lambda
\nar
\begin{bmatrix}
0 & 0      &  x_{11}  & \dots  & x_{1,n-1} \\
0 & 0      &         &        &          \\
\vdots  & \vdots &          &   X_L  &     \\
0 & 0      &         &        &          \\
0 & 0      & 0       & \dots  & 0       
\end{bmatrix}
+\begin{bmatrix}
x_{1n} & 0      &  0  & \dots  & 0 \\
0 & 0      &         &        &          \\
\vdots  & \vdots &          &   X_R  &     \\
0 & 0      &         &        &          \\
0 & x_{n1} & x_{n2} & \dots  & x_{nn}      
\end{bmatrix}
\right)\\
&=x_{1n}\det X\det(\lambda U+\one_{n-1}),
\end{aligned}
\end{equation}
 where $X_L$ and $X_R$ are the Schur complements of $x_{n1}$ and $x_{1n}$, respectively, that is, 
$(n-1)\times (n-1)$ matrices given by
\begin{equation*}
\begin{aligned}
X_L&=X_{[2,n]}^{[1,n-1]}-X_{[2,n]}^{[n,n]}x_{1n}^{-1}X_{[1,1]}^{[1,n-1]},\\
X_R&=X_{[1,n-1]}^{[2,n]}-X_{[1,n-1]}^{[1,1]}x_{n1}^{-1}X_{[n,n]}^{[2,n]},
\end{aligned}
\end{equation*}
and $U=X_LX_R^{-1}$. Note that 
\begin{equation}\label{detU}
\det U=\frac{x_{n1}}{x_{1n}},
\end{equation}
so that $U\in GL_{n-1}$. This suggests that a map $\Psi':GL_n\to GL_{n-1}$ defined 
via $\Psi':X\mapsto U$ can be utilized in constructing a generalized cluster structure on $GL_n$ compatible with 
$\Poi$.

We start the proof of Theorem~\ref{twomaps}(i) with a number of preliminary technical computations. First, we rewrite the definitions of
matrices $X_L$ and $X_R$ as
\begin{equation}\label{xlxr}
\begin{aligned}
X_L&=\nar\begin{bmatrix}0&\one_{n-1}\end{bmatrix}(X-XE_LX)\begin{bmatrix} \one_{n-1}\cr 0\end{bmatrix},\\
X_R&=\nar\begin{bmatrix}\one_{n-1}& 0\end{bmatrix}(X-XE_RX)\begin{bmatrix} 0\cr \one_{n-1}\end{bmatrix},
\end{aligned}
\end{equation}
where $E_L=e_nx_{1n}^{-1}e_1^T$, $E_R=e_1x_{n1}^{-1}e_n^T$,
$e_1$ and $e_n$ are the first and the last columns of $\one_n$. 
 By~\eqref{xlxr} we get
\[
\begin{aligned}
\delta X_L&=\nar\begin{bmatrix} 0 &\one_{n-1}\end{bmatrix}(\one_n-XE_L)\delta X(\one_n-E_LX)
\begin{bmatrix} \one_{n-1}\cr0\end{bmatrix},\\
\delta X_R&=\nar\begin{bmatrix} \one_{n-1} &0\end{bmatrix}(\one_n-XE_R)\delta X(\one_n-E_RX)
\begin{bmatrix}0\cr \one_{n-1}\end{bmatrix},
\end{aligned}
\]
so that
\[
\begin{aligned}
\delta U=\delta X_L X_R^{-1}-U\delta X_R X_R^{-1}&
=\nar\begin{bmatrix} 0 &\one_{n-1}\end{bmatrix}(\one_n-XE_L)\delta X(\one_n-E_LX)
\begin{bmatrix} X_R^{-1}\cr0\end{bmatrix}\\
&-\nar\begin{bmatrix} U &0\end{bmatrix}(\one_n-XE_R)\delta X(\one_n-E_RX)
\begin{bmatrix}0\cr X_R^{-1}\end{bmatrix}.
\end{aligned}
\]
 It follows that the gradients $\nabla (f\circ\Psi')$ at a point $X\in GL_n$ and 
$\nablac f$ at the point $U=\Psi'(X)\in GL_{n-1}$ are related via 
\begin{multline*}
\nabla(f\circ\Psi')=(\one_n-E_LX)\nar\begin{bmatrix} X_R^{-1}\cr0\end{bmatrix}\nablac f\begin{bmatrix} 0 &\one_{n-1}\end{bmatrix}
(\one_n-XE_L)\\
-(\one_n-E_RX)\nar\begin{bmatrix}0\cr X_R^{-1}\end{bmatrix}\nablac f\begin{bmatrix} U &0\end{bmatrix}(\one_n-XE_R)\\
=(\one_n-E_LX)\nar\begin{bmatrix} 0 & X_R^{-1}\nablac f\cr 0&0\end{bmatrix}(\one_n-XE_L)
-(\one_n-E_RX)\begin{bmatrix}0&0\cr X_R^{-1}\nablac fU&0\end{bmatrix}(\one_n-XE_R).
\end{multline*}
Note that~\eqref{xlxr} implies
\begin{equation}\label{XEX}
\begin{aligned}
X(\one_n-E_LX)&=(\one_n-XE_L)X=\nar\begin{bmatrix} 0 & 0\cr X_L &0\end{bmatrix},\\
X(\one_n-E_RX)&=(\one_n-XE_R)X=\nar\begin{bmatrix} 0 & X_R\cr 0 &0\end{bmatrix},
\end{aligned}
\end{equation}
so that the left and right gradients of $f\circ\Psi'$ can be written as
\[
\begin{aligned}
\nabla(f\circ\Psi') X&=(\one_n-E_LX)\nar\begin{bmatrix} X_R^{-1}\nablac f X_L &0\cr 0&0\end{bmatrix}
-(\one_n-E_RX)\begin{bmatrix}0&0\cr 0 &X_R^{-1}\nablac f X_L\end{bmatrix},\\
X\nabla(f\circ\Psi')&=\nar\begin{bmatrix}0& 0 &\cr 0& U\nablac f\end{bmatrix}(\one_n-XE_L)
-\begin{bmatrix}\nablac fU&0\cr 0&0\end{bmatrix}(\one_n-XE_R).
\end{aligned}
\]
Introducing
\begin{alignat*}{2}
\eta&=-x_{1n}^{-1}X_{[1,1]}^{[1,n-1]}X_R^{-1}\nablac f X_L, \qquad 
&&\zeta=-x_{n1}^{-1}X_{[n,n]}^{[2,n]}X_R^{-1}\nablac f X_L,\\
\xi&=-U\nablac f X_{[2,n]}^{[n,n]}x_{1n}^{-1}, \qquad 
&&\nu=-\nablac f U X^{[1,1]}_{[1,n-1]}x_{n1}^{-1},
\end{alignat*}
we rewrite the expressions for the left and right gradients of $f\circ\Psi'$ as
\begin{equation}\label{leftrightgrad}
\begin{aligned}
\nabla(f\circ\Psi') X&=\begin{bmatrix} X_R^{-1}\nablac f X_L &0\cr \eta&0\end{bmatrix}
-\begin{bmatrix}0&\zeta\cr 0 &X_R^{-1}\nablac f X_L\end{bmatrix},\\
X\nabla(f\circ\Psi')&=\begin{bmatrix}0& 0 &\cr \xi& U\nablac f\end{bmatrix}
-\begin{bmatrix}\nablac fU&\nu\cr 0&0\end{bmatrix}.
\end{aligned}
\end{equation}

We now move on to the computation of the bracket $\Poi$.  It is convenient to re-write $\Poi$ in a way where the Lie-theoretic definition of $\gamma$ is replaced by a more concrete transformation on the standard matrix
representation of 
$\gl_n$. Let $E_{kj}$ denote the $n\times n$ matrix having one at position $(k,j)$ and zeros at all other positions. Put $\bar S=E_{12}+\dots+E_{n-1,n}$ and $S=\bar S^T$
and define the action of $\gamma$ and $\gamma^*$ on $\gl_n$ via
\[
\gamma(A)=SA\bar S,\qquad \gamma^*(A)=\bar SAS.
\]
Clearly, for $n\times n$ matrices with an $(n-1)\times(n-1)$ block $M$ one has
\[ 
\gamma\begin{bmatrix} M&\star\cr \star &\star\end{bmatrix}=\begin{bmatrix} 0&0\cr 0 &M\end{bmatrix},\qquad
\gamma^*\begin{bmatrix} \star&\star\cr \star &M\end{bmatrix}=\begin{bmatrix} M&0\cr 0 &0\end{bmatrix},
\]
where the exact values of $\star$ are irrelevant.

By~\eqref{Rplusgamma}, $\bar R_+$  on $\gl_n$ can be written as
\[
\bar R_+=\bar R_0+\frac1{1-\gamma^*}\pi_>-\frac{\gamma}{1-\gamma}\pi_<
\]
with $\bar R_0$ given by
\begin{equation}\label{barR0}
\bar R_0(\eta)=\frac{n-1}{2n}\Tr\eta{\cdot}\one_n+\frac1n\left(\Tr\eta{\cdot}\bar D_n-\Tr(\bar D_n\eta){\cdot}\one_n\right)
-\frac{\gamma}{1-\gamma}\eta,
\end{equation}
where $\bar D_n=\diag(1,2,\dots,n)$, see~\cite[Lemma 4.1]{GSVMem}.

Consequently, the components of $\bar R_+(\nabla(f\circ\Psi') X)$ are
\begin{gather*}
\bar R_0\nar\left(\begin{bmatrix} X_R^{-1}\nablac f X_L& 0\cr \eta&0\end{bmatrix}
-\begin{bmatrix}0&\zeta\cr 0 &X_R^{-1}\nablac f X_L\end{bmatrix}\right)=
-\bar R_0(1-\gamma^*)\begin{bmatrix} 0 & 0\cr 0 & (X_R^{-1}\nablac f X_L)_0\end{bmatrix},\\
\frac1{1-\gamma^*}\nar\left(\begin{bmatrix} X_R^{-1}\nablac f X_L& 0\cr \eta&0\end{bmatrix}
-\begin{bmatrix}0&\zeta\cr 0 &X_R^{-1}\nablac f X_L\end{bmatrix}\right)_>=
-\begin{bmatrix} 0&\zeta\cr 0& (X_R^{-1}\nablac f X_L)_>\end{bmatrix},\\
\frac{\gamma}{1-\gamma}\nar\left(\begin{bmatrix} X_R^{-1}\nablac f X_L& 0\cr \eta&0\end{bmatrix}
-\begin{bmatrix}0&\zeta\cr 0 &X_R^{-1}\nablac f X_L\end{bmatrix}\right)_<=
-\begin{bmatrix} 0&0\cr 0& (X_R^{-1}\nablac f X_L)_<\end{bmatrix};
\end{gather*}
here and in what follows $M_>=\pi_>(M)$, $M_0=\pi_0(M)$, and $M_<=\pi_<(M)$. 
%are the upper triangular, diagonal, and lower triangular parts of $M$, respectively.

It follows from~\eqref{barR0} that for a diagonal
$(n-1)\times(n-1)$ matrix $D$ one has
\[
\bar R_0(1-\gamma^*)\begin{bmatrix} 0 & 0\cr 0& D\end{bmatrix}=
%\left(\frac12(1-\gamma^*)+\frac12(1+\gamma^*)\right)\begin{bmatrix} 0 & 0\cr 0& D\end{bmatrix}= 
\begin{bmatrix} 0 & 0\cr 0& D\end{bmatrix}-\frac1n\Tr D{\cdot}\one_n.
\]

Summing up the components of $\bar R_+(\nabla (f\circ\Psi')X)$ we arrive at
\begin{multline}\label{barRleft}
\bar R_+(\nabla (f\circ\Psi')X)=
-\nar\begin{bmatrix}0&\zeta\cr 0 &X_R^{-1}\nablac f X_L\end{bmatrix}
 +\frac1n\Tr (X_R^{-1}\nablac f X_L) {\cdot}\one_n
\\=-(\one_n-E_RX)\begin{bmatrix}0&0\cr 0 &X_R^{-1}\nablac f X_L\end{bmatrix}
 +\frac1n\Tr (X_R^{-1}\nablac f X_L) {\cdot}\one_n.
\end{multline}
Therefore, the first term in the right hand side of~\eqref{sklyabragen} is 
\begin{multline*}
\left\langle \bar R_+(\nabla(f_1\circ\Psi') X),\nabla(f_2\circ\Psi') X\right\rangle\\
=-\left\langle (\one_n-E_RX)\nar\begin{bmatrix}0&0\cr 0 &X_R^{-1}\nablac f_1 X_L\end{bmatrix},
(\one_n-E_LX)\begin{bmatrix} X_R^{-1}\nablac f_2 X_L &0\cr 0&0\end{bmatrix}\right\rangle\\
+\left\langle\nar\begin{bmatrix}0&\zeta(f_1)\cr 0 &X_R^{-1}\nablac f_1 X_L\end{bmatrix},
\begin{bmatrix}0&\zeta(f_2)\cr 0 &X_R^{-1}\nablac f_2 X_L\end{bmatrix}\right\rangle
+\left\langle\frac1n\Tr (X_R^{-1}\nablac f X_L) {\cdot}\one_n,\nabla(f_2\circ\Psi') X \right\rangle.
\end{multline*}
 It follows immediately from~\eqref{leftrightgrad} that the left gradient belongs to $\sl_n$, and hence the third term
in the right hand above vanishes.
Further, the second term is just $\langle X_R^{-1}\nablac f_1 X_L,X_R^{-1}\nablac f_2 X_L\rangle$.
Using~\eqref{XEX} we rewrite the first term as
\[
-\left\langle X^{-1}\nar\begin{bmatrix} 0&  X_R\cr 0&0\end{bmatrix}\begin{bmatrix} 0 & 0\cr 0& X_R^{-1}\nablac f_1 X_L\end{bmatrix},
X^{-1}\begin{bmatrix} 0& 0\cr X_L& 0\end{bmatrix}\begin{bmatrix} X_R^{-1}\nablac f_2 X_L&0\cr 0&0\end{bmatrix}\right\rangle,
\]
so that finally
\begin{multline}\label{ACGpart}
\left\langle \bar R_+(\nabla(f_1\circ\Psi') X),\nabla(f_2\circ\Psi') X\right\rangle\\
=-\left\langle X^{-1}\nar\begin{bmatrix} 0& \nablac f_1 X_L\cr 0&0\end{bmatrix},
X^{-1}\begin{bmatrix} 0& 0\cr U\nablac f_2 X_L&0\end{bmatrix}\right\rangle+\langle X_R^{-1}\nablac f_1 X_L,X_R^{-1}\nablac f_2 X_L\rangle.
\end{multline}

We proceed by computing $R_+(X\nabla (f\circ\Psi'))$ with $R_+$ given  on $\gl_n$ by
\[
R_+=R_0+\frac1{1-\gamma}\pi_>-\frac{\gamma^*}{1-\gamma^*}\pi_<
\]
and $R_0$ defined via
\begin{equation}\label{R0}
R_0(\eta)=\frac{n-1}{2n}\Tr\eta{\cdot}\one_n+\frac1n\left(\Tr\eta{\cdot} D_n-\Tr(D_n\eta){\cdot}\one_n\right)
-\frac{\gamma^*}{1-\gamma^*}\eta
\end{equation}
with $D_n=\diag(n,\dots,2,1)$ similarly to $\bar R_0$ above.

The components of $R_+(X\nabla(f\circ\Psi'))$ are
\begin{gather*}
 R_0\left(\nar\begin{bmatrix} 0 & 0\cr \xi& U\nablac f\end{bmatrix}
-\begin{bmatrix}\nablac f U&\nu\cr 0 & 0\end{bmatrix}\right)=
R_0\left(\nar\begin{bmatrix} 0 & 0\cr 0& (U\nablac f)_0\end{bmatrix}
-\begin{bmatrix}(\nablac f U)_0&0\cr 0 & 0\end{bmatrix}\right),\\
\frac1{1-\gamma}\left(\nar\begin{bmatrix} 0 & 0\cr \xi& U\nablac f\end{bmatrix}
-\begin{bmatrix}\nablac f U&\nu\cr 0 & 0\end{bmatrix}\right)_>=
-\frac1{1-\gamma}\nar\begin{bmatrix} 0&0\cr 0& [\nablac f, U]_>\end{bmatrix}-\begin{bmatrix}(\nablac fU)_>& \nu\cr 0&0\end{bmatrix},\\
\frac{\gamma^*}{1-\gamma^*}\left(\nar\begin{bmatrix} 0 & 0\cr \xi& U\nablac f\end{bmatrix}
-\begin{bmatrix}\nablac f U&\nu\cr 0 & 0\end{bmatrix}\right)_<=
-\frac1{1-\gamma^*}\nar\begin{bmatrix} [\nablac f,U]_<&0\cr 0& 0\end{bmatrix}+ \begin{bmatrix}(\nablac fU)_<& 0\cr 0&0\end{bmatrix}.
\end{gather*}
Note that
\[
\nar\begin{bmatrix} 0 & 0\cr \xi& U\nablac f\end{bmatrix}-\begin{bmatrix}\nablac f U&\nu\cr 0 & 0\end{bmatrix}=
-\begin{bmatrix} 0 & 0\cr \xi& [\nablac f,U]\end{bmatrix}-(1-\gamma)\begin{bmatrix}\nablac f U&\nu\cr 0 & 0\end{bmatrix},
\]
and both matrices on the right hand side belong to $\sl_n$. 
It follows from~\eqref{R0} that for a diagonal $(n-1)\times(n-1)$ matrix $D$ one has
\[
R_0(1-\gamma)\begin{bmatrix} D & 0\cr 0& 0\end{bmatrix}=
\begin{bmatrix} D & 0\cr 0& 0\end{bmatrix}-\frac1n\Tr D{\cdot}\one_n,
\]
and  hence
\begin{multline*}
R_0 \left(\nar\begin{bmatrix} 0 & 0\cr 0& (U\nablac f)_0\end{bmatrix}-\begin{bmatrix}(\nablac f U)_0&0\cr 0 & 0\end{bmatrix}\right)\\=
-R_0\begin{bmatrix} 0 & 0\cr 0& [\nablac f,U]_0\end{bmatrix}-\begin{bmatrix}(\nablac f U)_0&0\cr 0 & 0\end{bmatrix}
+\frac1n\Tr (\nabla f U){\cdot}\one_n.
\end{multline*}
 
Summing up the components of $R_+(X\nabla (f\circ\Psi'))$ we arrive at
\begin{multline}\label{almostthere2}
R_+(X\nabla (f\circ\Psi'))= -R_0\nar\begin{bmatrix} 0 & 0\cr 0 & [\nablac f, U]_0\end{bmatrix}
+\frac1n\Tr (\nabla f U){\cdot}\one_n\\
-\frac1{1-\gamma}\nar\begin{bmatrix} 0 & 0\cr 0 & [\nablac f, U]_>\end{bmatrix}
+\frac{\gamma^*}{1-\gamma^*}\begin{bmatrix} 0 & 0\cr 0 & [\nablac f, U]_< \end{bmatrix}
-\begin{bmatrix}\nablac f U&\nu\cr 0 & 0\end{bmatrix}.
\end{multline}
To compute the second term in the right hand side of~\eqref{sklyabragen} we treat separately the contributions of
the terms in~\eqref{almostthere2} to $\left\langle  R_+(X\nabla(f_1\circ\Psi') ),X\nabla(f_2\circ\Psi') \right\rangle$.
The contribution of the first term is the sum of
\[
\left\langle -R_0\nar\begin{bmatrix} 0 & 0\cr 0 & [\nablac f_1, U]_0\end{bmatrix},
-\begin{bmatrix} 0 & 0\cr 0 & [\nablac f_2, U]_0\end{bmatrix}\right\rangle=
\left\langle R_0[\nablac f_1, U]_0,[\nablac f_2, U]_0\right\rangle
\]
(note that $R_0$ in the right hand side above  is the operator on $\gl_{n-1}$ associated with the corresponding Cremmer--Gervais data)
and
\begin{multline*}
\left\langle -R_0\nar\begin{bmatrix} 0 & 0\cr 0 & [\nablac f_1, U]_0\end{bmatrix},
-(1-\gamma)\begin{bmatrix} (\nablac f_2 U)_0 & 0\cr 0 & 0\end{bmatrix}\right\rangle\\=
\left\langle \frac1{1-\gamma}\nar\begin{bmatrix} 0 & 0\cr 0 & [\nablac f_1, U]_0\end{bmatrix},
(1-\gamma)\gamma^*\begin{bmatrix} 0 & 0\cr 0& (\nablac f_2 U)_0 \end{bmatrix}\right\rangle\\=
\left\langle \nar\begin{bmatrix} 0 & 0\cr 0 & [\nablac f_1, U]_0\end{bmatrix},
\frac{(1-\gamma)\gamma^*}{1-\gamma^*}\begin{bmatrix} 0 & 0\cr 0& (\nablac f_2 U)_0 \end{bmatrix}\right\rangle\\=
-\left\langle [\nablac f_1, U]_0,(\nablac f_2 U)_0\right\rangle;
\end{multline*}
here the last equality follows from the fact that
$\gamma\gamma^*$ acts by multiplication by 
$\left[\begin{smallmatrix} 0&0\cr 0& \one_{n-1}\end{smallmatrix}\right]$ on the left.
To justify the first equality note that by~\cite[Lemma 4.2]{GSVple}, $R_0$ on the Cartan subalgebra of $sl_n$ is given by 
$R_0=\frac12+\frac12\left(\frac1{1-\gamma}-\frac1{1-\gamma^*}\right)$; further, one can readily check that on this subalgebra
$\frac{\gamma}{1-\gamma}+\frac1{1-\gamma^*}=0$, so that finally $R_0=\frac1{1-\gamma}$.  The contribution
of the second term vanishes since by~\eqref{leftrightgrad} the right gradient belongs to $\sl_n$.

The contribution of the third and the fourth terms is the sum of
\begin{multline*}
\left\langle -\frac1{1-\gamma}\nar\begin{bmatrix} 0 & 0\cr 0 & [\nablac f_1, U]_>\end{bmatrix}
+\frac{\gamma^*}{1-\gamma^*}\begin{bmatrix} 0 & 0\cr 0 & [\nablac f_1, U]_<\end{bmatrix},
-\begin{bmatrix} 0 & 0\cr 0 & [\nablac f_2, U]\end{bmatrix}\right\rangle\\=
\left\langle \frac1{1-\gamma}[\nablac f_1, U]_>-\frac{\gamma^*}{1-\gamma^*}[\nablac f_1, U]_<,[\nablac f_2, U]\right\rangle
\end{multline*}
and
\begin{multline*}
\left\langle -\frac1{1-\gamma}\nar\begin{bmatrix} 0 & 0\cr 0 & [\nablac f_1, U]_>\end{bmatrix}
+\frac{\gamma^*}{1-\gamma^*}\begin{bmatrix} 0 & 0\cr 0 & [\nablac f_1, U]_<\end{bmatrix},
-(1-\gamma)\begin{bmatrix} \nablac f_2 U & 0\cr 0 & 0\end{bmatrix}\right\rangle\\=
\left\langle \frac{(1-\gamma^*)\gamma}{1-\gamma}\nar\begin{bmatrix} [\nablac f_1, U]_>& 0\cr 0&0\end{bmatrix}
-\begin{bmatrix} [\nablac f_1, U]_<& 0\cr 0&0\end{bmatrix},\begin{bmatrix} \nablac f_2 U & 0\cr 0 & 0\end{bmatrix}\right\rangle\\=
\left\langle -\nar\begin{bmatrix} [\nablac f_1, U]& 0\cr 0&0\end{bmatrix}+\begin{bmatrix} [\nablac f_1, U]_0& 0\cr 0&0\end{bmatrix},
\begin{bmatrix} \nablac f_2 U & 0\cr 0 & 0\end{bmatrix}\right\rangle\\=
-\left\langle  [\nablac f_1, U],\nablac f_2 U\right\rangle+\left\langle  [\nablac f_1, U]_0,(\nablac f_2 U)_0\right\rangle;
\end{multline*}
here the second equality follows from the fact that
$\gamma^*\gamma$ acts by multiplication by 
$\left[\begin{smallmatrix} \one_{n-1}&0\cr 0&0\end{smallmatrix}\right]$ on the left. 

Finally, the contribution of the last term in the right hand side of~\eqref{almostthere2} are
\[
\left\langle -\nar\begin{bmatrix}\nablac f_1 U&\nu(f_1)\cr 0 & 0\end{bmatrix}, 
-\nar\begin{bmatrix}\nablac f_2 U&\nu(f_2)\cr 0 & 0\end{bmatrix}\right\rangle=
\left\langle \nablac f_1 U, \nablac f_2 U\right\rangle
\]
and
\begin{multline*}
\left\langle -\nar \begin{bmatrix}\nablac f_1 U&0\cr 0 & 0\end{bmatrix}(\one_n-XE_R),
 \begin{bmatrix}0&0\cr U\nablac f_2\end{bmatrix}(\one_n-XE_l)\right\rangle\\=
-\left\langle \nar \begin{bmatrix}\nablac f_1 U&0\cr 0 & 0\end{bmatrix}\begin{bmatrix}0& X_R \cr 0& 0\end{bmatrix}X^{-1},
\begin{bmatrix}0&0\cr U\nablac f_2 \end{bmatrix}\begin{bmatrix}0& 0\cr X_L& 0 \end{bmatrix}X^{-1}\right\rangle\\=
-\left\langle \nar \begin{bmatrix}\nablac f_1 X_L&0\cr 0 & 0\end{bmatrix}X^{-1},
\begin{bmatrix}0&0\cr U\nablac f_2X_L \end{bmatrix}X^{-1}\right\rangle
\end{multline*}
via~\eqref{XEX}. Summing up all the contributions we arrive at
\begin{multline}\label{CGpart}
\left\langle  R_+(X\nabla(f_1\circ\Psi') ),X\nabla(f_2\circ\Psi') \right\rangle\\=
\left\langle R_0[\nablac f_1, U]_0,[\nablac f_2, U]_0\right\rangle+
\left\langle \frac1{1-\gamma}[\nablac f_1, U]_>-\frac{\gamma^*}{1-\gamma^*}[\nablac f_1, U]_<,[\nablac f_2, U]\right\rangle\\-
\left\langle  [\nablac f_1, U],\nablac f_2 U\right\rangle+\left\langle \nablac f_1 U, \nablac f_2 U\right\rangle-
\left\langle \nar \begin{bmatrix}\nablac f_1 X_L&0\cr 0 & 0\end{bmatrix}X^{-1},
\begin{bmatrix}0&0\cr U\nablac f_2X_L \end{bmatrix}X^{-1}\right\rangle.
\end{multline}
Note that the first three terms in~\eqref{CGpart} form $\left\langle R_+[\nablac f_1,U], [\nablac f_2U]\right\rangle$, 
while the last two terms in~\eqref{CGpart} coincide with the right hand side of~\eqref{ACGpart}. Consequently, we get
\[
\{f_1\circ \Psi',f_1\circ \Psi'\}=\left\langle  [\nablac f_1, U],\nablac f_2 U\right\rangle-\left\langle R_+[\nablac f_1,U], [\nablac f_2,U]\right\rangle,
\]
which coincides with the expression for the bracket $-\Poi^\dag$, see~\cite[Eq.~(5.6)]{GSVdouble}. \qed

\subsection{Map $\Psi''$ and the proof of Theorem~\ref{twomaps}(ii)} Given $X\in GL_n$ and a rational function
$\bar M(\lambda)=\bar q(\lambda)/p(\lambda)\in\bar \RR_{n-1}$ with
\begin{equation}\label{qandp}
\bar q(\lambda)=\sum_{i=0}^{n-1}\bar q_i\lambda^i,\qquad p(\lambda)=\sum_{i=0}^{n-1}p_i\lambda^i,
\end{equation}
define a map $\Psi'': GL_n\to\bar \RR_{n-1}$ via $\Psi'': X\mapsto \bar M(\lambda)$ given by
\[
\bar q_i\circ\Psi''=\frac{x_{n,i+1}}{x_{1n}},\qquad p_i\circ\Psi''=\frac{x_{1,i+1}}{x_{1n}}, \qquad i\in[0,n-1]. 
\]

We start the proof of Theorem~\ref{twomaps}(ii) with computing the left and right gradients of the entries $x_{1k}$ and 
$x_{nj}$. Clearly,
$\nabla x_{1k}=E_{k1}$ and $\nabla x_{nj}=E_{jn}$, and hence
\begin{equation}\label{gradientsforx}
\begin{alignedat}{2}
\nabla  x_{1k}X&=e_kX_{[1,1]},\qquad\qquad &\nabla x_{nj}X=e_jX_{[n,n]},\\
X\nabla x_{1k}&=X^{[k,k]}e_1^T,\qquad\qquad &X\nabla x_{nj}=X^{[j,j]}e_n^T.
\end{alignedat}
\end{equation}
It follows immediately that $\gamma^*(X\nabla x_{1k})=0$ and $(X\nabla x_{1k})_>=0$, and similarly
$\gamma(X\nabla x_{nj})=0$ and $(X\nabla x_{nj})_<=0$. Consequently, we get
\begin{multline*}
 \{x_{1k},x_{1l}\}=\left\langle \bar R_+(\nabla x_{1k}X),\nabla x_{1l}X\right\rangle-
\left\langle R_+(X\nabla x_{1k}),X\nabla x_{1l}\right\rangle\\=
\left\langle\bar R_0 E_{kk},E_{ll}\right\rangle x_{1k}x_{1l}
-\left\langle R_0 E_{11},E_{ll}\right\rangle x_{1k}x_{1l}\\+
\left\langle \frac1{1-\gamma^*} (e_kX_{[1,1]})_>-\frac{\gamma}{1-\gamma}(e_kX_{[1,1]})_<,e_lX_{[1,1]}\right\rangle.
\end{multline*}
Assuming that $l<k$, we can rewrite the last term above as
\[
\sum_{s=1}^{l-1} x_{1s}x_{1,k+l-s}-\sum_{t=1}^{l-1} x_{1,k+t}x_{1,l-t}=0;
\]
here and in what follows we assume that $x_{1i}$ and $x_{ni}$ vanish whenever $i>n$.
Therefore, for $l<k$ we get
\[
\{x_{1k},x_{1l}\}=\left(\left\langle\bar R_0 E_{kk},E_{ll}\right\rangle-
\left\langle R_0 E_{11},E_{ll}\right\rangle\right)x_{1k}x_{1l}.
\]
Note that $E_{kk}\in gl_n$, so we use expressions~\eqref{barR0} and~\eqref{R0} for $\bar R_0$ and $R_0$,
respectively. 

Consequently, $\left\langle\bar R_0 E_{kk},E_{ll}\right\rangle=\frac{n-1}{2n}+\frac1n(l-k)$ and 
$\left\langle R_0 E_{11},E_{11}\right\rangle=\frac{n-1}{2n}$, so that $\{x_{1k},x_{1l}\}=\frac{l-k}nx_{1k}x_{1l}$ for $l<k$. For $l>k$
one has $\{x_{1k},x_{1l}\}=-\{x_{1l},x_{1k}\}=-\frac{k-l}nx_{1k}x_{1l}=\frac{l-k}nx_{1k}x_{1l}$, so that finally
\begin{equation}\label{firstfirst}
 \{x_{1k},x_{1l}\}=\frac{l-k}nx_{1k}x_{1l}.
\end{equation}

Similarly to above,
\begin{multline*}
 \{x_{nk},x_{nl}\}=\left\langle \bar R_+(\nabla x_{nk}X),\nabla x_{nl}X\right\rangle-
\left\langle R_+(X\nabla x_{nk}),X\nabla x_{nl}\right\rangle\\=
\left(\left\langle\bar R_0 E_{kk},E_{ll}\right\rangle
-\left\langle R_0 E_{nn},E_{nn}\right\rangle\right) x_{nk}x_{nl}\\+
\left\langle \frac1{1-\gamma^*} (e_kX_{[n,n]})_>-\frac{\gamma}{1-\gamma}(e_kX_{[n,n]})_<,e_lX_{[n,n]}\right\rangle,
\end{multline*}
and for $l<k$ the last term vanishes. Further, $\left\langle R_0 E_{nn},E_{nn}\right\rangle=
\frac{n-1}{2n}$, so that 
\begin{equation}\label{lastlast}
 \{x_{nk},x_{nl}\}=\frac{l-k}nx_{nk}x_{nl}.
\end{equation}

We proceed with
\begin{multline*}
 \{x_{1k},x_{nl}\}=\left\langle \bar R_+(\nabla x_{1k}X),\nabla x_{nl}X\right\rangle-
\left\langle R_+(X\nabla x_{1k}),X\nabla x_{nl}\right\rangle\\=
\left(\left\langle\bar R_0 E_{kk},E_{ll}\right\rangle
-\left\langle R_0 E_{11},E_{nn}\right\rangle\right) x_{1k}x_{nl}\\+
\left\langle \frac1{1-\gamma^*} (e_kX_{[1,1]})_>-\frac{\gamma}{1-\gamma}(e_kX_{[1,1]})_<,e_lX_{[n,n]}\right\rangle.
\end{multline*}
Assuming that $l\ls k$ we can rewrite the last term above as
\[
\sum_{s=1}^{l-1} x_{ns}x_{1,k+l-s}-\sum_{t=1}^{l-1} x_{n,k+t}x_{1,l-t}=\sum_{j=1}^{l-1}
(x_{1,k+l-j}x_{nj}-x_{1j}x_{n,k+l-j}).
\]
Further,  $\left\langle R_0 E_{11},E_{nn}\right\rangle=\frac{n-1}{2n}-\frac{n-1}n$, so that 
\begin{equation}\label{first>last}
 \{x_{1k},x_{nl}\}=\frac{n+l-k-1}nx_{1k}x_{nl}+\sum_{j=1}^{l-1}(x_{1,k+l-j}x_{nj}-x_{1j}x_{n,k+l-j})
\end{equation}
for $l\le k$. 

Finally, consider
\begin{multline*}
 \{x_{nk},x_{1l}\}=\left\langle \bar R_+(\nabla x_{nk}X),\nabla x_{1l}X\right\rangle-
\left\langle R_+(X\nabla x_{nk}),X\nabla x_{1l}\right\rangle\\=
\left(\left\langle\bar R_0 E_{kk},E_{ll}\right\rangle
-\left\langle R_0 E_{nn},E_{11}\right\rangle\right) x_{nk}x_{1l}\\+
\left\langle \frac1{1-\gamma^*} (e_kX_{[n,n]})_>-\frac{\gamma}{1-\gamma}(e_kX_{[n,n]})_<,e_lX_{[1,1]}\right\rangle\\-
\left\langle \frac1{1-\gamma} (X^{[k,k]}e_n^T)_>-\frac{\gamma*}{1-\gamma*}(X^{[k,k]}e_n^T)_<,X^{[l,l]}e_l^T\right\rangle.
\end{multline*}
Here the third term for $l<k$ equals
\[
\sum_{j=1}^{l-1}(x_{n,k+l-j}x_{1j}-x_{nj}x_{1,k+l-j})
\]
similarly to above, the fourth term equals $x_{nl}x_{1k}$, and 
 $\left\langle R_0 E_{nn},E_{11}\right\rangle=-\frac{n+1}{2n}+\frac{n-1}n$, so that
\[
\{x_{nk},x_{1l}\}=\frac{l-k+1}nx_{nk}x_{1l}+\sum_{j=1}^{l-1}(x_{n,k+l-j}x_{1j}-x_{nj}x_{1,k+l-j})-x_{nl}x_{1k}
\]
for $l<k$, which finally gives
\begin{equation}\label{first<last}
 \{x_{1k},x_{nl}\}=\frac{n+l-k-1}nx_{1k}x_{nl}+\sum_{j=1}^{k}(x_{1,k+l-j}x_{nj}-x_{1j}x_{n,k+l-j})
\end{equation}
for $l> k$.

From~\eqref{firstfirst}--\eqref{first<last} we immediately deduce
\begin{gather*}
\{p_k\circ\Psi'',p_l\circ\Psi''\}=0,\qquad \{\bar q_k\circ\Psi'',\bar q_l\circ\Psi''\}=0,\\
\{p_k\circ\Psi'',\bar q_l\circ\Psi''\}=
\begin{cases}
\sum\limits_{j=0}^{l-1}(p_{k+l-j}\bar q_j-p_j\bar q_{k+l-j})\circ\Psi'',\quad l\le k,\\
\sum\limits_{j=0}^k(p_{k+l-j}\bar q_j-p_j\bar q_{k+l-j})\circ\Psi'',\quad l>k,
\end{cases}
\end{gather*}
which can be summarized as
\begin{gather*}
\{p\circ\Psi''(\lambda),p\circ\Psi''(\mu)\}=0,\qquad \{\bar q\circ\Psi''(\lambda),\bar q\circ\Psi''(\mu)\}=0,\\
\{p\circ\Psi''(\lambda),\bar q\circ\Psi''(\mu)\}=\frac\mu{\lambda-\mu}
\left(p(\lambda)\bar q(\mu)-p(\mu)\bar q(\lambda)\right)\circ\Psi''.
\end{gather*}
Consequently,
\[
\{\bar M\circ\Psi''(\lambda),\bar M\circ\Psi''(\mu)\}=-\left(\lambda\bar M(\mu)-\mu\bar M(\lambda)\right)
\frac{\bar M(\lambda)-\bar M(\mu)}{\lambda-\mu}\circ\Psi'',
\]
which by~\eqref{bartodabra} means that $\{\bar M\circ\Psi''(\lambda),\bar M\circ\Psi''(\mu)\}=\{\bar M(\lambda),\bar M(\mu)\}^T\circ\Psi''$. \qed

\subsection{Proof of Theorem~\ref{twomaps}(iii)} To prove the remaining part of Theorem~\ref{twomaps} it suffices
to show that $\{u_{ij}\circ\Psi',p_k\circ\Psi''\}=\{u_{ij}\circ\Psi',\bar q_k\circ\Psi''\}=0$ for any entry $u_{ij}$ 
of $U$ and any coefficient $p_k$ and $\bar q_k$ of $\bar M(\lambda)=\Psi''(X)$.

We start with computing $\{u_{ij}\circ\Psi',x_{1k}\}$. We use the second equation in~\eqref{XEX}, similarly to how it was done in the derivation of~\eqref{ACGpart}, to rewrite~\eqref{barRleft} as 
\[
\bar R_+(\nabla (f\circ\Psi')X)=
-\nar X^{-1}\begin{bmatrix} 0 &\nablac f X_L\cr 0&0\end{bmatrix}
 +\frac1n\Tr (X_R^{-1}\nablac f X_L) {\cdot}\one_n.
\]
Consequently,
\begin{multline*}
\left\langle \bar R_+(\nabla (u_{ij}\circ\Psi')X),\nabla x_{1k}X\right\rangle\\=
-\left\langle X^{-1}\nar\begin{bmatrix} 0 &\nablac u_{ij} X_L\cr 0&0\end{bmatrix},e_kX_{[1,1]}\right\rangle
+\left\langle\frac1n\Tr (X_R^{-1}\nablac u_{ij} X_L) {\cdot}\one_n,e_kX_{[1,1]}\right\rangle\\=
-\left\langle \nar\begin{bmatrix} 0 &\nablac u_{ij} X_L\cr 0&0\end{bmatrix},E_{k1}\right\rangle+\frac1n u_{ij}x_{1k}.
\end{multline*}

Further, we use the same equation in~\eqref{XEX} to rewrite~\eqref{almostthere2} as
\begin{multline}\label{R0right}
R_+(X\nabla (f\circ\Psi'))= -R_0\nar\begin{bmatrix} 0 & 0\cr 0 & [\nablac f, U]_0\end{bmatrix}
 +\frac1n\Tr (\nabla f U){\cdot}\one_n\\
-\frac1{1-\gamma}\nar\begin{bmatrix} 0 & 0\cr 0 & [\nablac f, U]_>\end{bmatrix}
+\frac{\gamma^*}{1-\gamma^*}\begin{bmatrix} 0 & 0\cr 0 & [\nablac f, U]_< \end{bmatrix}
-\begin{bmatrix}0&\nablac f U\cr 0 & 0\end{bmatrix}X^{-1}.
\end{multline}
Consequently, the expression for $\left\langle R_+(X\nabla (u_{ij}\circ\Psi')),X\nabla x_{1k}\right\rangle$ has five terms. The contribution of the second term in~\eqref{R0right} is $\frac1n u_{ij}x_{1k}$. The contributions of the third and the fourth terms vanish since both matrices have a zero first row, while the only nonzero column of 
$X\nabla x_{1k}=X^{[k,k]}e_1^T$ is the first one. The contribution of the fifth term is
\[
-\left\langle \nar\begin{bmatrix} 0 &\nablac u_{ij} X_L\cr 0&0\end{bmatrix}X^{-1},X^{[k,k]}e_1^T\right\rangle=
-\left\langle \nar\begin{bmatrix} 0 &\nablac u_{ij} X_L\cr 0&0\end{bmatrix},E_{k1}\right\rangle.
\]
Finally, the contribution of the first term by~\eqref{R0} equals
\[
-\left\langle \frac1n \Tr\left(D_n\nar\begin{bmatrix} 0 & 0\cr 0 & [\nablac u_{ij}, U]_0\end{bmatrix}\right)\one_n,
X^{[k,k]}e_1^T\right\rangle-\left\langle \frac{\gamma^*}{1-\gamma^*}\nar\begin{bmatrix} 0 & 0\cr 0 & 
[\nablac u_{ij}, U]_0\end{bmatrix},X^{[k,k]}e_1^T\right\rangle,
\]
since $\Tr [\nablac u_{ij}, U]=0$. Note that $\left[\begin{smallmatrix} 0&0\cr0&[\nablac u_{ij}, U]\end{smallmatrix}
\right]=E_{j+1,j+1}-E_{i+1,i+1}$, so that 
\[
\Tr\left(D_n\nar\begin{bmatrix} 0 & 0\cr 0 & [\nablac u_{ij}, U]_0\end{bmatrix}\right)=i-j
\]
and
\begin{equation}\label{middle}
\frac{\gamma^*}{1-\gamma^*}\begin{bmatrix} 0 & 0\cr 0 & 
[\nablac u_{ij}, U]_0\end{bmatrix}=\begin{cases} \phantom{-}E_{i+1,i+1}+\dots+E_{jj}\quad\text{for $i<j$},\\
                                                 \phantom{-}0\qquad\qquad\qquad\qquad\quad\;\text{for $i=j$},\\
																								 -E_{j+1,j+1}-\dots-E_{ii}\quad\text{for $i<j$}.
																								\end{cases}
\end{equation}
Clearly, the upper left entry in the latter expression vanishes for any $i$ and $j$, so that its contribution
to $\left\langle R_+(X\nabla (u_{ij}\circ\Psi')),X\nabla x_{1k}\right\rangle$ vanishes, and the contribution of
the former expression is $\frac{i-j}n u_{ij}x_{1k}$. Summing up the contributions of all terms we arrive at 
$\{u_{ij}\circ\Psi',x_{1k}\}=\frac{j-i}n u_{ij}x_{1k}$, so that
\[
\{u_{ij}\circ\Psi',p_k\circ\Psi''\}=\left\{u_{ij}\circ\Psi',\frac{x_{1,k+1}}{x_{1n}}\right\}=0.
\]

In a similar way one can compute the bracket $\{u_{ij}\circ\Psi',x_{nk}\}$. Namely,
\[
\left\langle \bar R_+(\nabla (u_{ij}\circ\Psi')X),\nabla x_{nk}X\right\rangle=
-\left\langle \nar\begin{bmatrix} 0 &\nablac u_{ij} X_L\cr 0&0\end{bmatrix},E_{kn}\right\rangle+\frac1n u_{ij}x_{nk}.
\]
Further, the contribution of the second term in~\eqref{R0right} is $\frac1n u_{ij}x_{nk}$. The contributions of the third and the fourth terms vanish since both matrices have a zero last row, while the only nonzero column of 
$X\nabla x_{nk}=X^{[k,k]}e_n^T$ is the last one. The contribution of the fifth term is
$-\left\langle \left[\begin{smallmatrix} 0 &\nablac u_{ij} X_L\cr 0&0\end{smallmatrix}\right],E_{kn}\right\rangle$.
Finally, the contribution of the first term in~\eqref{R0right} is $\frac{i-j}n u_{ij}x_{1k}$, since the lowe right entry in~\eqref{middle} vanishes (because both $i$ and $j$ are at most $n-1$). Summing up the contributions of all terms we arrive at 
$\{u_{ij}\circ\Psi',x_{nk}\}=\frac{j-i}n u_{ij}x_{nk}$, so that
\[
\{u_{ij}\circ\Psi',q_k\circ\Psi''\}=\left\{u_{ij}\circ\Psi',\frac{x_{n,k+1}}{x_{1n}}\right\}=0.
\]
\qed

\section{Explicit expressions: the proof of Theorem~\ref{expressions}}\label{explicit}

\subsection{Proof of Theorem~\ref{expressions}(i)}\label{explicitphi} 
Define $(n-1)(n-2)/2$ matrices $\Nphi_{kl}(X)$ for $k,l\in[1,n-2]$ and $k+l\le n-1$ of sizes $(n-k-l)(n+1)\times (n-k-l)(n+1)$. The matrix $\Nphi_{kl}(X)$ consists of $n-k-l$ block
columns of width $n+1$. Each column is vertically subdivided into four tiers. The upper tier of the $i$-th column
is a zero block $\zero_{(i-1)n\times(n+1)}$. The second tier is $\begin{bmatrix} X_{I\cup n} & 0\end{bmatrix}$ with
$I=[2,n-1]$ for all columns except for the last one, for which $I=[l+1,n-1]$.  
Note that the leftmost element in the last row in each second tier is $x_{n1}$.
The third tier is
$\begin{bmatrix}0 & X_{1\cup\gamma(J)}\end{bmatrix}$ with $J=[2,n-1]$
for all columns except for the last two; for the second column from the right 
$J=[l+1,n-1]$, and for the last column $J=[k+1,n-1]$. Note that the rightmost element in the first row in each third tier is $x_{1n}$. 
Finally, the lower tier is once again a zero block of the corresponding size. Note that the upper tier is void for the first column, and the lower tier is void for the last column. For example, the matrix $\Nphi_{21}(X)$ for $n=6$ is given by
\begin{small}
\[
\nar
\begin{bmatrix}
x_{21} & {\cdot} & {\cdot} & {\cdot} & {\cdot} & x_{26} & 0 & 0 & 0 & 0 & 0 & 0 & 0 & 0 & 0 & 0 & 0 & 0 & 0 & 0 & 0\cr
x_{31} & {\cdot} & {\cdot} & {\cdot} & {\cdot} & x_{36} & 0 & 0 & 0 & 0 & 0 & 0 & 0 & 0 & 0 & 0 & 0 & 0 & 0 & 0 & 0\cr
x_{41} & {\cdot} & {\cdot} & {\cdot}& {\cdot} & x_{46}  & 0 & 0 & 0 & 0 & 0 & 0 & 0 & 0 & 0 & 0 & 0 & 0 & 0 & 0 & 0\cr
x_{51} &  {\cdot} & {\cdot} & {\cdot}& {\cdot} & x_{56} & 0 & 0 & 0 & 0 & 0 & 0 & 0 & 0 & 0 & 0 & 0 & 0 & 0 & 0 & 0\cr
x_{61} &  {\cdot} & {\cdot} & {\cdot}& {\cdot} & x_{66} & 0 & 0 & 0 & 0 & 0 & 0 & 0 & 0 & 0 & 0 & 0 & 0 & 0 & 0 & 0\cr
0 & x_{11} &  {\cdot} & {\cdot} & {\cdot}& {\cdot} & x_{16} & 0 & 0 & 0 & 0 & 0 & 0 & 0 & 0 & 0 & 0 & 0 & 0 & 0 & 0\cr
0 & x_{31} &  {\cdot} & {\cdot} & {\cdot}& {\cdot} & x_{36} &x_{21} & {\cdot} & {\cdot} & {\cdot}& {\cdot} &x_{26} & 0 & 0 & 0 & 0 & 0 & 0 & 0 & 0\cr
0 & x_{41} &  {\cdot} & {\cdot} & {\cdot}& {\cdot} &x_{46} & x_{31} &  {\cdot} & {\cdot} & {\cdot}& {\cdot} & x_{36} & 0 & 0 & 0 & 0 & 0 & 0 & 0 & 0\\
0 & x_{51} &  {\cdot} & {\cdot} & {\cdot}& {\cdot} & x_{56} & x_{41} &  {\cdot} & {\cdot} & {\cdot}& {\cdot} & x_{46} & 0 & 0 & 0 & 0 & 0 & 0 & 0 & 0\\
0 & x_{61} &  {\cdot} & {\cdot} & {\cdot}& {\cdot} &x_{66}& x_{51} &  {\cdot} & {\cdot} & {\cdot}& {\cdot} &x_{56}& 0 & 0 & 0 & 0 & 0 & 0 & 0 & 0\\
0 &0 & 0 & 0 & 0 & 0 & 0  & x_{61} &  {\cdot} & {\cdot} & {\cdot}& {\cdot} &x_{66} & 0 & 0 & 0 & 0 & 0 & 0 & 0 & 0\\
0 &0 & 0 & 0 & 0 & 0 & 0  & 0 & x_{11} & {\cdot} & {\cdot} & {\cdot}& {\cdot} & x_{16} & 0 & 0 & 0 & 0 & 0 & 0 & 0\\
0 &0 & 0 & 0 & 0 & 0 & 0  & 0 &x_{31} &  {\cdot} & {\cdot} & {\cdot}& {\cdot} &x_{36} & x_{21} & {\cdot} & {\cdot} & {\cdot}& {\cdot} & x_{26} & 0\\
0 &0 & 0 & 0 & 0 & 0 & 0  & 0 & x_{41} &  {\cdot} & {\cdot} & {\cdot}& {\cdot} & x_{46} & x_{31} &  {\cdot} & {\cdot} & {\cdot}& {\cdot} & x_{36} & 0\\
0 &0 & 0 & 0 & 0 & 0 & 0  & 0 & x_{51} & {\cdot} & {\cdot} & {\cdot}& {\cdot} & x_{56}  & x_{41} &  {\cdot} & {\cdot} & {\cdot}& {\cdot} & x_{46}& 0\\
0 &0 & 0 & 0 & 0 & 0 & 0  & 0 & x_{61} &  {\cdot} & {\cdot} & {\cdot}& {\cdot} &x_{66} & x_{51} & {\cdot} & {\cdot} & {\cdot}& {\cdot} & x_{56}& 0\\
0 &0 & 0 & 0 & 0 & 0 & 0  & 0 & 0 &0 & 0 & 0 & 0 & 0  &  x_{61} &  {\cdot} & {\cdot} & {\cdot}& {\cdot} & x_{66} & 0\\
0& 0 &0 & 0 & 0 & 0 & 0 & 0  & 0 & 0 &0 & 0 & 0 & 0 & 0  &   x_{11} &  {\cdot} & {\cdot} & {\cdot}& {\cdot} &x_{16}\\
0& 0 &0 & 0 & 0 & 0 & 0 & 0  & 0 & 0 &0 & 0 & 0 & 0 & 0  &   x_{41} &  {\cdot} & {\cdot} & {\cdot}& {\cdot} & x_{46}   \\
0& 0 &0 & 0 & 0 & 0 & 0 & 0  & 0 & 0 &0 & 0 & 0 & 0 & 0  &   x_{51} & {\cdot} & {\cdot} & {\cdot}& {\cdot} & x_{56}   \\
0& 0 &0 & 0 & 0 & 0 & 0 & 0  & 0 & 0 &0 & 0 & 0 & 0 & 0  &    x_{61} & {\cdot} & {\cdot} & {\cdot}& {\cdot} & x_{66}
\end{bmatrix}.
\]
\end{small}

Using $x_{n1}$ and $x_{1n}$ placed as explained above and applying column operations within each block column, we transform it into
\[
\nar\begin{bmatrix}
\star & & 0\cr
\vdots & (X_R)_I &\vdots\cr
\star & & 0\cr
x_{n1}& 0\dots0& 0\cr
0& 0\dots0& x_{1n}\cr
0 && \star\cr
\vdots & (X_L)_J &\vdots\cr
0 && \star
\end{bmatrix}=
\begin{bmatrix}
\star & & 0\cr
\vdots & \one_I &\vdots\cr
\star & & 0\cr
x_{n1}& 0\dots0& 0\cr
0& 0\dots0& x_{1n}\cr
0 && \star\cr
\vdots & U_J &\vdots\cr
0 && \star
\end{bmatrix}\diag (1, X_R, 1),
\]
where $U=X_LX_R^{-1}=\Psi'(X)$ and $\star$, as usual, stands for the entries that are inessential for further computations. Note that the entries
$x_{1n}$ and $x_{n1}$ are the only nonzero entries in the corresponding rows of the whole matrix $\Nphi_{kl}(X)$. 
Taking into account that $\det X_R=(-1)^{n+1}x_{n1}^{-1}\det X$, we get
\begin{multline*}
\phi_{kl}(X)=\det\Nphi_{kl}(X)\\=(x_{1n}\det X)^{n-k-l}\det
\begin{bmatrix} \one_{[2,n-1]} & \dots& 0& 0 \\ U_{[2,n-1]} &\ddots & 0& 0\\
  \qquad  \ddots & \one_{[2,n-1]} &0 & 0\\
 0&U_{[2,n-1]}& \one_{[2,n-1]}& 0\\
 0&\dots & U_{[l+1,n-1]} & \one_{[l+1,n-1]}\\
 0&\dots & 0&  U_{[k+1,n-1]}&
 \end{bmatrix},
\end{multline*}
where $\one$ is the unit matrix of size $(n-1)\times(n-1)$. 
Applying further block column operations, the matrix above can be reduced to
\[
\begin{bmatrix} 
\one_{[2,n-1]} &\dots & 0& 0 \\ 
\vdots  & \ddots& 0& 0\\
 0 & \one_{[2,n-1]} & 0& 0\\
0 &0 & \one_{[2,n-1]}& 0\\
0 &{\cdot}s & 0 & \one_{[l+1,n-1]}\\
(-1)^{j-1}U^{j}_{[k+1,n-1]} & {\cdot}s & - U^2_{[k+1,n-1]}&  U_{[k+1,n-1]}
 \end{bmatrix}
\]
with $j=n-k-l$, and hence
\begin{align*}
\phi_{kl}(X)&=(-1)^{\frac{j(j-1)}{2}}\left (x_{1n}\det X \right )^{j}(-1)^{n\frac{j(j-1)}2+(n-l-1)(n-k-1)}\\
&\times\det \nar\begin{bmatrix} \one^{[1,k]}& (U^{j})^{[1,1]} & {\cdot}s &(U^2)^{[1,1]} &  U^{[1,l]}  \end{bmatrix}\\
&= \left (x_{1n}\det X \right )^j (-1)^{n\frac{j(j-1)}2+kn} s_{kl}^{n-1} \fy_{kl}(U).
\end{align*}
Taking into account the definition of the signs $s_{kl}$ in~\eqref{signdef}, we get~\eqref{badphi}. \qed

\subsection{Proof of Theorem~\ref{expressions}(ii)}\label{badgproof}
 We start the proof with deriving expressions~\eqref{gdef} for functions $g_{ij}(U)$. 

We decompose a generic element $X \in GL_n$ as $X = X_+ X_{0,-}$, where $X_{0,-} \in \B_-$ and  $X_+ \in \N_+$. Following \cite[Section 7.1]{GV}, we define a sequence of rational maps $H_{k}: GL_n \to GL_n$ via
\begin{equation}
\label{Fmaps}
    H_0(X) = X, \quad  H_{k}(X) = X\bgamma(H_{k-1}(X)_+), \ k \geq 1, \quad  X \in GL_n.
\end{equation}
The rational map $H : GL_n \to GL_n$ is defined as the limit 
\begin{equation}\label{eq:fdefg}
H(X) = \lim_{k\rightarrow \infty}H_k(X), \quad X \in GL_n.
\end{equation}

The following lemma is a particular case of~\cite[Lemma 4.6]{GV}. Here we provide a proof geared towards needs of the current paper.

\begin{lemma}
\label{lemma:Fstable}
The sequence $Y_k=H_k(X)$ stabilizes at $k = n-2$. 
\end{lemma}

\begin{proof} We will show by induction that for $k\ge0$ there exists $N_k \in \gamma^{k}(\N_+)$ such that
\begin{equation}\label{Finduction}
(Y_k)_+= (Y_{k-1})_+ N_k,\quad Y_k= Y_{k-1}\bgamma(N_{k-1}).
\end{equation}
As a result, since $\bgamma^{n-1}(\N_+)=\{ \one_n\}$,
\[
Y=Y_{n-2}=Y_{n-1} = Y_{n}= {\cdot}s
\]
and the claim follows.

Indeed, for $k=1$, the second relation in~\eqref{Finduction} folows from the definition of $H_1$ with $N_0=X_+$. To get the
first one, we refactor $Y_1= X_+ X_{0,-} \bgamma(X_+)=X_+ N_1 B_1$
with $N_1 \in \bgamma(\N_+)$, $B_1\in\B_-$.  Then induction yields
\[
Y_{k} = X\bgamma((Y_{k-1})_+)= 
X\bgamma((Y_{k-2})_+)\bgamma(N_{k-1})=Y_{k-1} \bgamma(N_{k-1}),
\]
and $N_k\in \bgamma^{k}(\N_+)$ is then defined via the refactorization $(Y_{k-1})_{0,-} \bgamma(N_{k-1}) = N_{k} B_k$. This proves~\eqref{Finduction}.
\end{proof}

\begin{corollary}\label{cor:Fprincipal}
{\rm (i)}  $Y=H(X)$ satisfies the relation
\begin{equation*}%\label{eq:fgrel}
Y = X\bgamma(Y_+),
\end{equation*}
and  so, $H$ is a birational map.

{\rm(ii)} For $1\leq \alpha\leq \beta\leq n$ and any $I\subset [2,n]$ of size $\beta - \alpha + 1$,
\begin{equation}
\label{eq:fstability}
\det Y_{I}^{[\alpha,\beta]} = \det (Y_k)_{I}^{[\alpha,\beta]}\quad\text{for $k\geq \alpha-2$}.
\end{equation}
\end{corollary}

\begin{proof} (i) Follows immediately from Lemma~\ref{lemma:Fstable}.

(ii) Follows from the second relation in~\eqref{Finduction} since multiplying on the right by an element in 
$\bgamma^k(\N_+)$ does not change column-solid minors with the first column $k+1$ onwards.
\end{proof}

\begin{lemma}
\label{lem:Fminors}
For any $1< \alpha \le \beta \le n$ and $ I\subset [1,n]$, $|I|=\beta -\alpha +1$,
\begin{equation}\label{Fminors}
\begin{aligned}
&\det(Y_{\alpha-2})_I^{[\alpha,\beta]} \prod_{s=1}^{\alpha-2}\det(Y_{\alpha-2-s})_{[\alpha-s,n]}^{[\alpha-s,n]} \\
&\quad =  \sum_{I_1,\ldots, I_{\alpha-2}}
\det X^{I_1}_{I} \det X_{\gamma^*(I_1)\cup [\beta,n]}^{I_2}  \det X_{\gamma^*(I_2)\cup n}^{I_3} {\cdot}s  
\det X_{\gamma^*(I_{\alpha-2})\cup n}^{[2,n]},
\end{aligned}
\end{equation}
where $I_1$ and  $I_2, \ldots, I_{\alpha-2}$ are subsets of $[2,n]$ of sizes $\beta - \alpha +1$ and $n- \alpha +2, \ldots, n- 2$, respectively.
\end{lemma}

\begin{proof} From~\eqref{Fmaps} and the Binet--Cauchy formula,
\begin{align*}
&\det(Y_{\alpha-2})_I^{[\alpha,\beta]}  = \sum_{I_1\subset [2,n], |I_1|=|I|} \det X^{I_1}_{I} 
\det\bgamma((Y_{\alpha-3})_+)_{I_1}^{[\alpha,\beta]}\\
&\ \ = \sum_{I_1\subset [2,n], |I_1|=|I|} \det X^{I_1}_{I} \det((Y_{\alpha-3})_+)_{\gamma^*(I_1)}^{[\alpha-1,\beta-1]} \\ 
& \ \ = \sum_{I_1\subset [2,n], |I_1|=|I|} \det X^{I_1}_{I} 
\frac{ \det (Y_{\alpha-3})_{\gamma^*(I_1)\cup [\beta,n]}^{[\alpha-1,n]}}{ \det(Y_{\alpha-3})_{[\alpha-1,n]}^{[\alpha-1,n]}}.
\end{align*}
Now~\eqref{Fminors} follows by induction if one takes into an account Corollary~\ref{cor:Fprincipal}(ii).
\end{proof}

We are now ready to derive expression~\eqref{gdef} for the functions $g_{ij}(U)$. 
In~\cite[Section 7.2]{GV}, the functions  $g_{ij}(U)$, $1< j \leq i \leq n$, were defined using certain 
{\em flag minors\/} of $H(U)$ via the reciprocal relations
\begin{equation}\label{eq:fopflags}
\begin{aligned}
&g_{ij}(U)=\det [H(U)]_{[i,n]}^{[j,n-i+j]}\prod_{s=2}^{j-1}\det [H(U)]_{[s,n]}^{[s,n]}, \\ 
&\det (H(U))_{[i,n]}^{[j,n-i+j]} = \begin{cases}
    \dfrac{g_{ij}(U)}{g_{j-1,j-1}(U)} \quad &\text{if $j\geq 3$},\\
    g_{ij}(U) \quad &\text{otherwise}.
\end{cases}
\end{aligned}
\end{equation}

\begin{proposition}\label{gproof}
Functions $g_{ij}(U)$ are given by expression~\eqref{gdef}
\begin{equation}\label{gfuncorrect}
\begin{aligned}
g_{ij}(U)&=\det (H(U))_{[i,n]}^{[j,n-i+j]}\prod_{s=2}^{j-1}\det [H(U)]_{[s,n]}^{[s,n]}\\ 
&= \sum_{I_1,\ldots, I_{j-2}}
\det U^{I_1}_{[i,n]} \det U_{\gamma^*(I_1)\cup [n-i+j,n]}^{I_2}  \det U_{\gamma^*(I_2)\cup n}^{I_3} {\cdot}s  \det U_{\gamma^*(I_{j-2})\cup n}^{[2,n]} ,
\end{aligned}
\end{equation}
where $I_1$ and  $I_2, \ldots, I_{j-2}$ are subsets of $[2,n]$ of sizes $n-i +1$ and $n- j +2, \ldots, n- 2$, respectively.
\end{proposition}

\begin{proof} 
Follows immediately from~\eqref{eq:fopflags} and Lemma~\ref{lem:Fminors}.
\end{proof}

To prove Theorem~\ref{expressions}(ii) we need two auxiliary technical statements. 
For an index set $I$ of size $k$, denote $(-1)^I=(-1)^{k(k-1)/2}\prod_{i\in I} (-1)^i$, and write $A_{\hat I}$ for the submatrix of $A$ in rows not belonging to $I$.

\begin{lemma}
\label{lemma:U(X)minors}
For any two index sets $I,J\subset [1,n-1]$ of equal size $k$,
\begin{equation}
\label{U(X)minors}
\det \Psi'(X)_{I}^J = (-1)^{(n-1)k}\frac{(-1)^J}{x_{1n}\det X}
\det
\begin{bmatrix*}
X_{\hat J}\quad 0 \\
0\ X_{1\cup \gamma(I)}
\end{bmatrix*}.
\end{equation}
\end{lemma}

\begin{proof} Similarly to the proof in the previous Section, we use $x_{n1}$ in the lower left corner of the top block and $x_{1n}$ in the upper right  corner of the bottom block
to reduce the determinant in the right hand side of~\eqref{U(X)minors} via column operations to
\[
(-1)^{n-1} x_{n1}x_{1n} \det\begin{bmatrix} (X_R)_{\hat J} \\ (X_L)_I\end{bmatrix} = (-1)^{n-1} x_{n1}x_{1n}\det X_R \det\begin{bmatrix} \one_{\hat J} \\ \Psi'(X)_I\end{bmatrix},
\]
where $\one=\one_{n-1}$. Taking into account that the determinant of the matrix above equals 
$(-1)^{(n-1)k}(-1)^J\det \Psi'(X)_I^J$ and that $\det X_R=(-1)^{n-1}x_{n1}^{-1}\det X$, we finally get~\eqref{U(X)minors}.
\end{proof}

\begin{lemma}
\label{auxPlucker}
Let $m$, $q$, $p$ be natural numbers such that $m\geq q+p$ and $q>p$, $A$, $B$, $C$ be three subsets of $[1,m]$ of sizes $p$, $q$, and $q-p$, respectively, and $L$ be an $m\times q$ matrix.  Then
\begin{equation}
\label{eq:auxPlucker}
\sum_{B'\subset B, |B'|=p} (-1)^{B'} \det L_{A\cup (B\setminus B')} \det L_{B'\cup C} = \det L_B  \det L_{A\cup C} .
\end{equation}
\end{lemma}

\begin{proof} Consider a $2q\times 2q$ matrix 
\[
\tilde{L} = \begin{bmatrix} L_A & 0\\ L_B & L_B\\ 0 & L_C\end{bmatrix}.
\]
By the block Laplace expansion formula, $\det \tilde{L}$ is equal, up to a sign $(-1)^{p(q+p)}$, to the left hand side of~\eqref{eq:auxPlucker}. On the other hand, subtracting the first block column of $\tilde{L}$ from the second and permuting the rows of the result, one sees that $\det \tilde{L}$ is equal to
\[
(-1)^{pq+p}\det \begin{bmatrix} L_B & 0\\ 0 & L_{A\cup C}\end{bmatrix}= (-1)^{pq+p}\det L_B  \det L_{A\cup C}.
\]
\end{proof}

Further, we introduce some additional notation. Let $I_1,\ldots, I_k$ be subsets of $[1,n]$. We denote by $\NG(I_1,\ldots, I_k)$ a rectangular matrix of size $\sum_l|I_l|\times(n+k-1)$ built out of submatrices 
$X_{I_1},\ldots, X_{I_k}$ that are stacked one below the other as follows: $X_{I_1}$ occupies the upper left corner of $\NG(I_1,\ldots, I_k)$, and each block $X_{I_{l+1}}$ is shifted by one column to the right with respect to the block 
$X_{I_{l}}$ above it. Thus, 
$\NG(I_1,\ldots, I_k)$ has $k$ block rows of sizes $|I_1|\times (n+k-1), \ldots, |I_k|\times (n+k-1)$, and the $l$-th block has  $l-1$ initial zero columns, followed by $X_{I_{l}}$, followed by $k-l$ zero columns. 
For row index sets of $\NG(I_1,\ldots, I_k)$, we will use notation $(J_1,\ldots , J_k)$, where $J_1 \subset I_1, \ldots, J_k \subset I_k$.
Of a special interest is the situation when $I_1=I_2={\cdots}=I_{k-2}=I$; the corresponding matrix is denoted $\NG(I^{k-2},I_{k-1},I_k)$. In 
particular, we define 
$\NG_{ij}(X)$ for $i\in[2,n-1]$ and $j\in [2,i]$ via 
\[
\NG_{ij}(X)=\NG(\{1,n\}^{j-2},1\cup[n+j-i,n],1\cup[i+1,n]).
\] 
Note that in this case $\sum_l|I_l|=n+j-1$, so that $\NG_{ij}(X)$ are
square matrices of size $(n+j-1)\times(n+j-1)$. For example, the matrix $\NG_{53}(X)$ for 
$n=6$ is given by
\[
%G_{53}(X) = 
\nar\begin{bmatrix}
x_{11} & x_{12} & x_{13} & x_{14} & x_{15} & x_{16} & 0 & 0\\
x_{61} & x_{62} & x_{63} & x_{64} & x_{65} & x_{66} & 0 & 0\\
0 & x_{11} & x_{12} & x_{13} & x_{14} & x_{15} & x_{16} & 0 \\
0 & x_{41} & x_{42} & x_{43} & x_{44} & x_{45} & x_{46} & 0 \\
0 & x_{51} & x_{52} & x_{53} & x_{54} & x_{55} & x_{56} & 0 \\
0 & x_{61} & x_{62} & x_{63} & x_{64} & x_{65} & x_{66} & 0 \\
0 & 0 & x_{11} & x_{12} & x_{13} & x_{14} & x_{15} & x_{16}  \\
0 & 0 & x_{61} & x_{62} & x_{63} & x_{64} & x_{65} & x_{66} \\
\end{bmatrix}.
\]

\begin{lemma}\label{lem:gfunXgen}
Let $U=\Psi'(X)$, $2\leq \alpha \leq \beta \leq n-1$, and $I$ be a subset of $[2,n-1]$ of size $\beta -\alpha +1$. Then
\begin{multline}
\label{eq:gfunXgen}
\det [H(U)]_{I}^{[\alpha,\beta]}\prod_{s=2}^{\alpha-1}\det [H(U)]_{[s,n-1]}^{[s,n-1]} \\= 
\frac{\sigma(n,\alpha,\beta)}{x_{1n}^{\alpha-1}\det X}\det\NG(\{1,n\}^{\alpha-2},1\cup[\beta +1, n],1\cup\gamma(I)), 
%\langle \underbrace{1,n |{\cdot}s |1,n}_{(\alpha-2)\ \mbox{times}} | 1, [\beta +1, n] | 1, \gamma(I)   \rangle\ ,
\end{multline}
where the sign $\sigma(n,\alpha,\beta)$ does not depend on $I$ and is given by
\begin{equation}\label{gsigngen}
\sigma(n,\alpha,\beta)=(-1)^{(n-1)(\beta-\alpha-1)+\alpha\beta}.
\end{equation}
\end{lemma}

\begin{proof} We proceed by induction on $\alpha$. For $\alpha=2$,~\eqref{eq:gfunXgen} follows from~\eqref{eq:fstability} and Lemma~\ref{lemma:U(X)minors}, taking into account that the matrix featuring in the right hand side of equation~\eqref{U(X)minors} is $\NG(\hat J,1\cup\gamma(I))$. In this case, 
\begin{equation}\label{sigmainit}
\sigma(n,2,\beta)=(-1)^{(n-1)(\beta -1)} (-1)^{[2,\beta]} =(-1)^{(n-1)(\beta-1)}. 
\end{equation}

Next, by~\eqref{eq:fstability} and~\eqref{Fmaps},
\[ 
\begin{aligned}
\det [H(U)]_{I}^{[\alpha+1,\beta]}&=\det [H_{\alpha-1}(U)]_{I}^{[\alpha+1,\beta]}
=\det \left [ U\bgamma(H_{\alpha-2}(U)_+)\right ]_{I}^{[\alpha+1,\beta]}\\ 
&= \sum_{\substack{K \subset [2,\beta]\\ |K| = \beta -\alpha}}\det U_I^K  \det \left [H_{\alpha-2}(U)_+\right ]_{\gamma^*(K)}^{[\alpha,\beta-1]}\\
&= \sum_{\substack{K \subset [2,\beta]\\ |K| = \beta -\alpha}}\det U_I^K  \frac{\det \left [H_{\alpha-2}(U)\right ]_{\gamma^*(K)\cup [\beta,n-1]}^{[\alpha,n-1]}}
{\det \left [H_{\alpha-2}(U)\right ]_{[\alpha,n-1]}^{[\alpha,n-1]}}\\
&= \sum_{\substack{K \subset [2,\beta]\\ |K| = \beta -\alpha}}\det U_I^K  \frac{\det \left [H(U)\right ]_{\gamma^*(K)\cup [\beta,n-1]}^{[\alpha,n-1]}}
{\det \left [H(U)\right ]_{[\alpha,n-1]}^{[\alpha,n-1]}}.
\end{aligned}
\]
Using the inductive assumption and Lemma~\ref{lemma:U(X)minors}, we obtain for a subset $I$ of $[2,n-1]$ of size $\beta-\alpha$
\begin{multline}\label{eq:gfunXderive}
\det [H(U)]_{I}^{[\alpha+1,\beta]}\prod_{s=2}^{\alpha}\det [H(U)]_{[s,n-1]}^{[s,n-1]} \\ 
=\sum_{\substack{K \subset [2,\beta]\\ |K| = \beta -\alpha}}\det U_I^K \det \left [H(U)\right ]_{\gamma^*(K)\cup [\beta,n-1]}^{[\alpha,n-1]}\prod_{s=2}^{\alpha-1}\det [H(U)]_{[s,n-1]}^{[s,n-1]}\\
= \sum_{\substack{K \subset [2,\beta]\\ |K| = \beta -\alpha}}\det U_I^K \frac{\sigma(n,\alpha,n-1)}{x_{1n}^{\alpha-1}\det X}
\det\NG(\{1,n\}^{\alpha-2},\{1,n\},1\cup K\cup[\beta +1, n])\\
= \frac{(-1)^{(n-1)(\beta-\alpha)}\sigma(n,\alpha,n-1)}{x_{1n}^{\alpha}(\det X)^2}  
\sum_{\substack{K \subset [2,\beta]\\ |K| = \beta -\alpha}} (-1)^K
\det\NG( \hat K,1\cup \gamma(I))\times\\
\det\NG(\{1,n\}^{\alpha-2},\{1,n\},1\cup K\cup[\beta +1, n]).
\end{multline}
Note that for $J=[2,\beta]\setminus K$ one has $(-1)^J=(-1)^K(-1)^{\beta(\alpha+1)}$, so
after multiplication by $(-1)^{\beta(\alpha+1)}x_{n1}^{\alpha-1} x_{1n}$ and passing to the summation over
$J$, the sum above can be rewritten as
\begin{equation}
\label{aux1}
\sum_{J \subset [2,\beta], |J| = \alpha-1} (-1)^J
\det\NG(n^{\alpha-1}, 1\cup J\cup [\beta+1,n],  1\cup \gamma(I))\det\NG(\{1,n\}^{\alpha-1},\hat J, 1), 
\end{equation}
and we can apply Lemma~\ref{auxPlucker} with $L= \NG(\{1,n\}^{\alpha-1}, [1, n], 1\cup \gamma(I) )$,
$m = n+\alpha +\beta - 2$, $q = n+\alpha$, $p=\alpha -1$, $A=(1,\dots,1,\pu,\dots\pu)$ with $\alpha-1$ nonempty subsets, 
$B=(n,\dots,n, [1, n], 1)$, $C= (n,\dots,n, 1\cup [\beta+1,n], 1\cup \gamma(I) )$. Then the sum in~\eqref{aux1} is equal to
\[
 \det L_B  \det L_{A\cup C}  =  \left (x_{n1}^{\alpha-1} x_{1n} \det X\right ) 
\det\NG(\{1,n\}^{\alpha-1}, 1\cup [\beta+1,n], 1\cup  \gamma(I)).
\]
Thus, we obtain~\eqref{eq:gfunXgen}  from~\eqref{eq:gfunXderive} with $\sigma(n,\alpha,\beta)$ satisfying the recurrent
relation
\begin{equation*}
\sigma(n, \alpha +1,\beta)= (-1)^{(n-1)(\beta-\alpha)+\beta(\alpha+1)}\sigma(n,\alpha,n-1).
\end{equation*}
It is easy to check that the solution of this recurrent relation with the initial conditions~\eqref{sigmainit} is given 
by~\eqref{gsigngen}
\end{proof}

The statement of Theorem~\ref{expressions}(ii) follows immediately from~\eqref{eq:fopflags} and Lemma~\ref{lem:gfunXgen}
for $I=[i,n-1]$, $\alpha=j$, $\beta=n-1-i+j$. \qed

\subsection{Proof of Theorem~\ref{expressions}(iii)} Recall that functions $\ec_r(X)$ are defined via 
$\det(\lambda\NA(X)+\NB(X))=\sum_{r=0}^{n-1}(-1)^{rn}\lambda^r\ec_r(X)$, where $\NA(X)$ and $\NB(X)$ are the blocks of the periodic staircase structure defined at the beginning of Section~\ref{twomapsiproof}. Relations~\eqref{badc}
follow immediately from~\eqref{ABvsU} and the definition~\eqref{cdef} of the functions $c_r(U)$ for $GL_{n-1}$. Note that $\ec_0(X)=x_{1n}\det X$ is not included in the initial cluster $\FF_n$. \qed

\subsection{Proof of Theorem~\ref{expressions}(iv)} Let $\bar q_i$ and $p_i$, $i\in[0,n-1]$, be the coefficients of
a rational function $\bar M(\lambda)\in\RR_{n-1}$ as defined in~\eqref{qandp}; recall that $p_{n-1}=1$. Additionally, define $\bar q_i=p_i=0$ for $i<0$ and $i\ge n$. The Laurent expansion~\eqref{moments} yields
\begin{equation}\label{barqviaph}
\bar q_j=\sum_{i=0}^\infty p_{i+j}\bar h_i\qquad\text{for $j\in\Z$}.
\end{equation}
%note that the sum above is finite since  
Define a $(2n-1)\times(2n-1)$ matrix $F=F(\bar M)$ via 
\[
F_{2i,j}=p_{j-i-1},\qquad F_{2i+1,j}=\bar q_{j-i-1}
\]
and denote $F_k= F_{[2n-k,2n-1]}^{[2n-k,2n-1]}$ for $k\in[1,2n-1]$. To compute $\det F_k$ we subtract from each $\bar q$-row in $F_k$ the linear combination of all $p$-rows in $F_k$ above it with the coefficients $\bar h_0$ for the $p$-row immediately above it, $\bar h_1$ for the previous $p$-row, and so on. As a result, the last $\lfloor k/2\rfloor$ entries in each $\bar q$-row in $F_k$ vanish by~\eqref{barqviaph}. For $k=2m-1$, 
the $t$-th entry of the $s$-th $\bar q$-row equals  
\[
\sum_{i=s-1}^\infty p_{n-m+t-1+i}\bar h_i.
\]
Consequently, 
\[
\det F_{2m-1}=(-1)^{m(m-1)/2}\det (h_{\alpha+\beta-2})_{\alpha,\beta=1}^{m} 
\det(p_{\alpha+\beta+n-m-2})_{\alpha,\beta=1}^{m}= \bar t_{m}^-,
\]
since the second determinant above equals $(-1)^{m(m-1)/2}$. 
For $k=2m$, 
the $t$-th entry of the $s$-th $\bar q$-row equals  
\[
\sum_{i=s}^\infty p_{n-m+t-2+i}\bar h_i.
\]
Consequently, 
\[
\det F_{2m}=(-1)^{m(m+1)/2}\det (h_{\alpha+\beta-1})_{\alpha,\beta=1}^{m} 
\det(p_{\alpha+\beta+n-m-2})_{\alpha,\beta=1}^{m}= \bar t_{m}^+
\]
for the same reason as above.

To complete the proof, define a $(2n-1)\times(2n-1)$ matrix $\NF(X)=x_{1n}F(\Psi''(X))$. The first row of $\NF(X)$ coinsides with the last row of $X$ appended with $n-1$ zeros, the second and the third rows are the first and the last rows of $X$ shifted right by one and apended with $n-2$ zeros, and so on. For example, the matrix $\NF(X)$ for $n=4$ equals
\[
\begin{bmatrix}
x_{41} & x_{42} & x_{43} & x_{44} & 0 & 0 & 0\\
0 & x_{11} & x_{12} & x_{13} & x_{14} & 0 & 0\\
0 & x_{41} & x_{42} & x_{43} & x_{44} & 0 & 0\\
0 & 0 & x_{11} & x_{12} & x_{13} & x_{14}  & 0 \\
0 & 0 & x_{41} & x_{42} & x_{43} & x_{44} & 0 \\
0 & 0 & 0 & x_{11} & x_{12} & x_{13} & x_{14}   \\
0 & 0 & 0 & x_{41} & x_{42} & x_{43} & x_{44} 
\end{bmatrix}.
\]
Finally, set $\ef_k(x)=
\det\NF(X)_{[2n-k,2n-1]}^{[2n-k,2n-1]}=x_{1n}^k\det F_k(\Psi''(X))$. Relations~\eqref{badbart} follow immediately from the above discussion. Relation $\bar t_n^-(\Psi''(X))/\bar t_{n-1}^+(\Psi''(X))=x_{n1}/x_{1n}$ is now evident. \qed

\section{Proof of Propositions~\ref{readytoglue}--\ref{Dexchange}}
\subsection{Proof of Proposition~\ref{readytoglue}}  (i) By Theorem~\ref{expressions}, the set of irreducible factors of denominators consists of
$z_1=x_{1n}$ and $z_2=\det X$. 
The situation at all vertices of $Q_{n-1}^\dag$ except for $(1,1)$  is covered by constructions 
of~\cite[Section~4.1 and Remark~4.8]{GSVpullback}. A straightforward computation with expressions~\eqref{badphi},~\eqref{badg}, 
and~\eqref{badbart} shows that the discrepancy $\delta_{1v}$ computed via~\eqref{discrep} vanishes at all vertices $v\ne(1,1)$ of $Q_{n-1}^\dag$ except for $(n-2,1)$, at which $\delta_{1v}=1$.
Consequently, a new vertex $A$ corresponding to $x_{1n}$ and an arrow $(n-2,1)\to A$ are added in $\Psi'^*Q_{n-1}^\dag$. 
Similarly, the discrepancy $\delta_{2v}$ vanishes at all vertices except for $(n-2,1)$ and $\langle n-1,2\rangle$, for which
$\delta_{2v}=1$ and $\delta_{2v}=-1$, respectively. Consequently, a new vertex $B$ corresponding to $\det X$ is added in 
$\Psi'^*Q_{n-1}^\dag$ together with arrows $(n-2,1)\to B$ and 
$B\to \langle n-1,2\rangle$. 
Vertex $(1,1)$ is handled as explained in~\cite[Section~4.4]{GSVpullback}: the function $\fy_1(r)$ computed  
via the first relation in~\cite[Eq.~(4.5)]{GSVpullback} equals $1-r/(n-1)$ for $r\in[1,n-1]$ and vanishes for $r=0$, and 
$\tau_1$ computed via the second relation in~\cite[Eq.~(4.5)]{GSVpullback} equals $1/(n-1)$. Therefore, we are exactly within
conditions of~\cite[Example~4.7(i)]{GSVpullback}, so an arrow from $A$ to $(1,1)$ is added in $\Psi'^*Q_{n-1}^\dag$ and it is changed under mutations in a usual way.

(ii) By Theorem~\ref{expressions}, the only irreducible factor in the denominators is $z_1=x_{1n}$. 
Since $\bar Q_{n-1}^T$ does not have generalized vertices, constructions of~\cite[Section~4.1]{GSVpullback} apply. A straightforward computation similar to the one in part (i) shows that the discrepancy $\delta_{1v}$ vanishes at all vertices 
$v$ of $\bar Q_{n-1}^T$ except for $(1,-)$, at which $\delta_{1v}=-1$. Consequently, a new vertex $A$ corresponding to $x_{1n}$ and an arrow $A\to (1,-)$ are added in 
$\bar \Psi''^*Q_{n-1}^T$. \qed

\subsection{Proof of Proposition~\ref{frozenchar} and applications}\label{prooffrozenchar} (i) Assume that $f$ is frozen, then by compatibility, $\{f,x_i(t)\}_{\G\CC}=c_ifx_i(t)$ for any cluster variable in any cluster $t$, and moreover, $\{f,L(t)\}_{\G\CC}=c_LfL(t)$ for any Laurent monomial $L(t)$ in cluster variables in $t$. Let $g\in\C[V]$. Since $\G\CC$ is complete, for any $t$ one can write $g$ as a sum of Laurent monomials
in cluster variables in $t$: $g=\sum_k L_k(t)$. Consequently, $\{f,g\}_{\G\CC}=\sum_k c_{L_k}fL_k(t)=f\tilde g$
for $\tilde g=\sum_k c_{L_k}L_k(t)$. So, for any fixed $t$, $\tilde g$ has a Laurent representation in cluster
variables in $t$, and hence belongs to the upper cluster algebra $\UC(\G\CC)$. By regularity of $\G\CC$, this means
that $\tilde g$ is regular.  

(ii) Note first that for any abstract generalized cluster structure, a Laurent monomial that contains a cluster variable in the denominator does not belong to the upper cluster algebra. Indeed, let $L=M'/M''$ be a Laurent monomial and let $M''=f^d\bar M''$ for a monomial 
$\bar M''$, a cluster variable $f$ in some cluster $t$ and $d>0$. In the adjacent cluster $t'$ obtained via mutation at $f$, we can write 
$L=M'(f')^d/P^d\bar M''$, where $P$ is the right hand side of the exchange relation $ff'=P$. Assuming that $L$ belongs to the ring of Laurent polynomials in the cluster $t'$ we get $M'(f')^d/P^d\bar M''=Q/\tilde M$ for some polynomial $Q$ and monomial $\tilde M$, which
leads to $P^dQ\bar M''=M'\tilde M(f')^d$, a contradiction, since the right hand side is a monomial, and the left hand side is not.  
Consequently, for a regular complete generalized structure on $V$ such a monomial is not a regular function. 

Assume now that $f$ is ordinary mutable, and that $ff'=M_1+M_2$ is the exchange relation for $f$. We claim that
$\{f,f'\}_{\G\CC}/f$ is not regular. Indeed, since all cluster variables in $M_1$ and $M_2$ lie in the same cluster 
as $f$ does, the compatibility yields $\{f,M_1\}_{\G\CC}=c_1fM_1$, $\{f,M_2\}_{\G\CC}=c_2fM_2$. On the other 
hand,~\cite[Theorem 1.4]{GSVb} yields  $\{f,M_1/M_2\}_{\G\CC}=(c_1-c_2)fM_1/M_2$ with $c_1-c_2\ne 0$. Consequently,
$\{f,f'\}_{\G\CC}/f=(c_1M_1+c_2M_2)/f$ for $c_1\ne c_2$. Note that the regularity of the latter expression together
with regularity of $f'=(M_1+M_2)/f$ would imply that both $M_1/f$ and $M_2/f$ are regular, a contradiction. 

Let now $f$ be generalized mutable of order $d$, and let the exchange relation for $f$ be 
$ff'=\sum_{i=0}^d M_i$. Similarly to above, we get $\{f,M_i\}_{\G\CC}=c_ifM_i$. Assuming that the right hand side of the
exchange relation for $f$ is written in the form prescribed by~\cite[Eq.~(2.4)]{GSVdouble}, we infer 
from~\cite[Proposition 2.5]{GSVdouble} that $c_0\ne c_d$ and that $c_i=c_0+i(c_d-c_0)/d$ for $i\in [0,d]$. Define 
a sequence of functions 
$$
f_j=\dfrac1f\sum_{i=j}^d\dfrac{i!}{(i-j)!}M_i
$$ 
for $j\in[0,d]$, so that $f_0=f'$, then $\{f,f_j\}_{\G\CC}/f=c_jf_j+\dfrac{c_d-c_0}{d}f_{j+1}$ for $j\in[0,d-1]$.
Consequently, regularity of $f_j$ and of the left hand side implies regularity of $f_{j+1}$. Note that $f_0$ is regular, so assuming that $\{f,f_j\}_{\G\CC}/f$ are regular for $j\in [0,d-1]$ we get that $f_d=d!M_d/f$ is regular,  a contradiction. This contradicition shows that the set $\{f_0,\dots,f_{d-1}\}$ contains at least one regular function $f_j$ for which  $\{f,f_j\}_{\G\CC}/f$ is not regular.
\qed

As an application, we prove that $x_{1n}$ and $x_{n1}$ should be frozen in a regular complete generalized cluster structure compatible with $\Poi$. Indeed, it follows immediately from~\eqref{gradientsforx} that
\[
(\nabla x_{1n}X)_>=(X\nabla x_{1n})_>=0,\quad \gamma(\nabla x_{1n}X)_<=0,
\quad \gamma^*(X\nabla x_{1n})_<=0,
\]
so the second and the third terms in expressions~\eqref{Rplusgamma} for $\bar R_+$ and $R_+$ vanish.
Further, by~\eqref{barR0} and~\eqref{R0}, 
\[
\bar R_0(\nabla x_{1n}X)=-\frac{n+1}{2n}x_{1n}\one_n+\frac1n x_{1n}\bar D_n,\quad
R_0(X\nabla x_{1n})=-\frac{n+1}{2n}x_{1n}\one_n+\frac1n x_{1n}D_n,
\]
and hence for any function $g$ one has 
\begin{multline}\label{x1nsigizmund}
\{x_{1n},g\}=\frac{x_{1n}}n\left(\left\langle \bar D_n-\frac{n+1}2\one_n, (\nabla g X)_0\right\rangle-
\left\langle D_n-\frac{n+1}2\one_n, (X\nabla g)_0 \right\rangle\right)\\
=\frac{x_{1n}}n\left(\left\langle \bar D_n, (\nabla g X)_0\right\rangle-\left\langle D_n,(X\nabla g)_0 \right\rangle\right)
\end{multline}
since $\langle \one_n, (\nabla g X)_0\rangle=\langle \one_n, (X\nabla g)_0\rangle$.
Clearly, the expression in brackets above is a regular function if $g$ is regular, so Proposition~\ref{frozenchar} applies.

Similarly, it follows from~\eqref{gradientsforx} that
\[
(\nabla x_{n1}X)_<=(X\nabla x_{n1})_<=0,\quad \gamma(\nabla x_{n1}X)_>=0,
\quad \gamma^*(X\nabla x_{n1})_>=0,
\]
so the third terms in expressions~\eqref{Rplusgamma} for $\bar R_+$ and $R_+$ vanish, and the second terms
become $(e_1X_{[n,n]})_>$ and $(X^{[1,1]}e_n^T)_>$, respectively. Further, by~\eqref{barR0} and~\eqref{R0}, 
\[
\bar R_0(\nabla x_{n1}X)=\frac{n-3}{2n}x_{n1}\one_n+\frac1n x_{n1}\bar D_n,\quad
R_0(X\nabla x_{n1})=\frac{n-3}{2n}x_{n1}\one_n+\frac1n x_{n1}D_n,
\]
and hence for any function $g$ one has 
\begin{multline*}
\{x_{n1},g\}=\frac{x_{n1}}n\left(\left\langle \bar D_n-\frac{n+3}2\one_n, (\nabla g X)_0\right\rangle-
\left\langle D_n-\frac{n+3}2\one_n, (X\nabla g)_0 \right\rangle\right)\\+
\left\langle e_1X_{[n,n]},\nabla g X\right\rangle-\left\langle X^{[1,1]}e_n^T,X\nabla g\right\rangle.
\end{multline*}
Note that $Xe_1X_{[n,n]}=X^{[1,1]}e_n^TX$, so that the last two terms cancel each other, which yields
\begin{equation}\label{xn1sigizmund}
\{x_{n1},g\}=\frac{x_{n1}}n\left(\left\langle \bar D_n, (\nabla g X)_0\right\rangle-\left\langle D_n,(X\nabla g)_0 \right\rangle\right),
\end{equation}
 and hence Proposition~\ref{frozenchar} applies for the same reason as above. 

As an immediate corollary of this computation we get the following fact.

\begin{lemma}\label{xn1/x1n}
 $x_{1n}/x_{n1}$ is a Casimir of the bracket $\Poi$.
\end{lemma}

\subsection{Proof of Proposition~\ref{Dexchange}} \label{dexchangeproof}
Recall that the four-term Pl\"ucker relation for a $(k+3)\times k$ matrix $A$ and six distinguished row indices
$\alpha< \beta< \gamma< \delta< \xi< \zeta$ reads
\[
\widehat{\alpha,\beta,\gamma}\cdot \widehat{\delta,\xi,\zeta}-\widehat{\alpha,\beta,\delta}\cdot \widehat{\gamma,\xi,\zeta}+
\widehat{\alpha,\gamma,\delta}\cdot \widehat{\beta,\xi,\zeta}-\widehat{\beta,\gamma,\delta}\cdot \widehat{\alpha,\xi,\zeta}=0
\]
where $\widehat{i_1,i_2,i_3}$ is the $k\times k$ minor obtained by deleting rows $i_1,i_2,i_3$ from $A$.
Define a $(2n+1)\times (2n-2)$ matrix $\NG(X)$ 
obtained from $\NG(\{1,n\}^{n-3},\{1,n-1,n\},\{1,n\})$ by prepending $[1,0,\ldots,0]$ as the topmost row followed by 
$[0,\ldots,0, 1]$, and consider the four-term Pl\"ucker relation for $\NG(X)$ and $\alpha=1$, $\beta=2$, $\gamma=3$, 
$\delta=2n-2$, $\xi=2n$, $\zeta=2n+1$.   It is easy to see that
\begin{align*}
& \widehat{\alpha, \beta,\gamma }=x_{n1} \eg_{n-1,n-2}(X), \quad \widehat{\delta,\xi,\zeta}=\ef_{2n-4}(X),\quad 
\widehat{\alpha,\beta, \delta}=\eg_{n-1,n-1}(X), \\  
&  \widehat{\beta,\xi, \zeta} =0, \quad\widehat{\beta,\gamma,\delta}=\ef_{2n-3}(X),
\quad \widehat{\alpha, \xi,\zeta}=-\eg_{n-2,n-2}(X),
\end{align*}
which verifies~\eqref{dexchange} with 
${\eg}(X) = \widehat{\gamma,\xi,\zeta}$. Clearly, $\eg(X)$ defined by this relation equals 
$\det\NG_{n-1,n-2}(X)_{[2,2n-3]}^{[2,2n-3]}$. \qed

%\begin{remark}\label{4termpluck}
%In what follows we refer to relations of the type used above as four-term Pl\"ucker relations for the matrix 
%$\NG(I_1,\ldots,I_p)$; 
%\end{remark} 

\section{Compatibility and toric action}\label{compatibility}

\subsection{Preliminary results}
The main technical tool in this Section is the following statement.

\begin{lemma}\label{x1nbra}
 If $g$ satisfies~\eqref{eq:T-hom} then $\{x_{1n},g\}=\frac1n(\bar\xi_g-\xi_g)x_{1n}g$ and 
$\{x_{n1},g\}=\frac1n(\bar\xi_g-\xi_g)x_{n1}g$.
\end{lemma}

\begin{proof}
Follows immediately from~\eqref{x1nsigizmund} and~\eqref{xn1sigizmund}. 
\end{proof}

 In view of the above lemma, we start with proving that all cluster variables in $\FF_n$ are homogeneous with respect to the action 
$X\mapsto t^{D_n}Xs^{\bar D_n}$ and computing the corresponding weights $\xi$ and
$\bar\xi$.

\begin{lemma}\label{weights}
All cluster variables in $\FF_n$ satisfy~\eqref{eq:T-hom}. 
The corresponding weights are given by
\begin{equation}\label{explweights}
\begin{aligned}
\xi_{x_{1n}}&=\bar\xi_{x_{1n}}=n,\qquad\xi_{x_{n1}}=\bar\xi_{x_{n1}}=1,\qquad 
\xi_{\det X}=\bar\xi_{\det X}=\frac{n(n+1)}2,\\
\xi_{\ec_r}&=\bar\xi_{\ec_r}=\frac{n(n+3)}2-r,\\
\xi_{\eg_{ij}}&=\frac{(i-j)(i-j+1)}2+\frac{(n-i)(n-i-1)}2+(j+1)n-1,\\
\bar\xi_{\eg_{ij}}&=\frac{n(n+1)}2+\frac{(j+1)(j-2)}2+i,\\
\xi_{\ef_m}&=\begin{cases} \dfrac{m(n+1)}2\qquad\quad&\text{for $m$ even},\\
                               \dfrac{(m-1)(n+1)}2+1\quad&\text{for $m$ odd},\end{cases}\\
\bar\xi_{\ef_m}&=\begin{cases} nm-\dfrac{m^2}4\qquad\quad&\text{for $m$ even},\\
                               nm-\dfrac{m^2-1}4\quad&\text{for $m$ odd},\end{cases}\\	
\xi_{\phi_{kl}}&=\frac{n(n-k-l)(k+l+4)}2-\frac{n(n-1)}2+\frac{l(l-1)}2+\frac{k(k+1)}2,\\
\bar\xi_{\phi_{kl}}&=\frac{(n+1)(n+2)}2(n-k-l)-\frac{(n-k-l-1)(n-k-l+2)}2-n+k.												
\end{aligned}
\end{equation}
\end{lemma}

\begin{proof}
Homogeneity for $x_{1n}$, $x_{n1}$, and $\det X$ is evident, and the corresponding expressions in~\eqref{explweights} follow immediately from the definition~\eqref{eq:T-hom}. To treat functions $\ec_r$ note that by~\eqref{badc}, 
$\ec_r(X)=x_{1n}\det X c_r(\Psi'(X))$. Observe that
\begin{align*}
\left ( t^{D_n} X s^{\bar D_n} \right )_L &= t^{D_{n-1}} X_L s^{\bar D_{n-1}}, \\
\left ( t^{D_n} X s^{\bar D_n} \right )_R &= t^{D_{n-1}+\one_{n-1}} X_L s^{\bar D_{n-1}+\one_{n-1}},
\end{align*}
and so
\[
\Psi' \left ( t^{D_n} X s^{\bar D_n} \right ) = \frac{1}{t s} \Psi'(X).
\]
Since $c_r(\Psi'(X))$ is, up to a sign, the $r$th coefficient of the characteristic polynomial of $\Psi'(X)$, we conclude that $c_r(\Psi'(X))$ satisfies~\eqref{eq:T-hom} with $\xi_{c_r}=\bar\xi_{c_r}=-r$. Consequently,
$\xi_{\ec_r}=\xi_{x_{1n}}+\xi_{\det X}+\xi_{c_r}=n(n+3)/2-r$, and the same holds true for $\bar\xi_{\ec_r}$.

To treat the $t$ part of the action $X\mapsto t^{D_n}Xs^{\bar D_n}$ for functions $\eg_{ij}(X)$ note that each row of 
$\NG_{ij}(X)$ equals to some row of $X$ prepended and appended by a number of zeros. It follows immediately that $\eg_{ij}$ are homogeneous with respect to the $t$ part of the action and that
\begin{multline*}
\xi_{\eg_{ij}}=(j-2)(n+1)+n+(1+2+\dots+i-j+1)+n+(1+2+\dots+n-i)\\
=\frac{(i-j)(i-j+1)}2+\frac{(n-i)(n-i-1)}2+(j+1)n-1.
\end{multline*}
To treat the $s$ part of the action write
\begin{multline*}
\NG_{ij}(Xs^{\bar D_n})=\diag(\one_2, s^{-1}\one_2,\dots,s^{-j+3}\one_2,s^{-j+2}\one_{i-j+2},s^{-j+1}\one_{n-i+1}
\NG_{ij}(X)\\
\times\diag(s,s^2,\dots,s^{n+j-1}).
\end{multline*}
It follows that $\eg_{ij}$ are homogeneous with respect to the $s$ part of the action and that
\begin{multline*}
\bar\xi_{\eg_{ij}}=-(j-2)(j-3)-(j-2)(i-j+2)-(j-1)(n-i+1)+(1+2+\dots+n+j-1)\\
=\frac{n(n+1)}2+\frac{(j+1)(j-2)}2+i.
\end{multline*}

To treat the $t$ part of the action for functions $\ef_m(X)$ recall that they are defined as trailing minors of the
matrix $\NG_{n-1,n-1}(X)$, and that by~\eqref{explweights}, $\xi_{\eg_{n-1,n-1}}=n^2-1$. Consequently, for $m$ even
\[
\xi_{\ef_m}=\xi_{\eg_{n-1,n-1}}-(n+1)\left(n-1-\frac m2\right)=\frac{m(n+1)}2,
\]
whereas for $m$ odd
\[
\xi_{\ef_m}=\xi_{\eg_{n-1,n-1}}-(n+1)\left(n-1-\frac {m-1}2\right)+1=\frac{(m-1)(n+1)}2+1.
\]
The $s$ part of the action for $\ef_m(X)$ is treated similarly $\eg_{ij}(X)$: for $m$ even the corresponding matrix is multiplied by $\diag(s^{m/2-1},s^{m/2-1},\dots,s,s,1,1)$ on the left and by $\diag(s^{n-m+1},\dots,s^n)$ on the right, which gives
\[
\bar\xi_{\ef_m}=mn-(1+\dots+m-1)+2\left(1+\dots+\frac m2-1\right)=mn-\frac{m^2}4.
\]
For $m$ odd the corresponding matrix is multiplied by $\diag(s^{(m-1)/2},s^{(m-1)/2-1},\allowbreak s^{(m-1)/2-1},\dots,s,s,1,1)$ on the left and by $\diag(s^{n-m+1},\dots,s^n)$ on the right, which gives
\[
\bar\xi_{\ef_m}=mn-(1+\dots+m-1)+2\left(1+\dots+\frac {m-1}2-1\right)+\frac {m-1}2=mn-\frac{m^2-1}4.
\]

To treat the $t$ part of the action for functions $\phi_{kl}(X)$ write
\begin{multline*}
\Nphi_{kl}(t^{D_n}X)\\=\diag(t^{D'_n+(n-k-l-1)\one_{n-1}},t^{2n-k-l-1},\dots,t^{D'_n+\one_{n-1}},t^{n+1},t^{D'_{n-l+1}},
t^n,t^{D'_{n-k}})\\
\times\Nphi_{kl}(X)\diag(t^{-n+k+l+1}\one_{n+1},\dots,t^{-1}\one_{n+1},\one_{n+1})
\end{multline*} 
with $D'_m=\diag(m-1,\dots,1)$, which implies
\begin{multline*}
\xi_{\phi_{kl}}=(1+\dots+n-1)(n-k-l-1)+(1+\dots+n-k-l-1)n+n(n-k-l)+(1+\dots+n-l)\\
+(1+\dots+n-k-1)-(1+\dots+n-k-l-1)(n+1)\\
=\frac{n(n-k-l)(k+l+4)}2-\frac{n(n-1)}2+\frac{l(l-1)}2+\frac{k(k+1)}2.
\end{multline*}
Finally, the $s$ part of the action for functions $\phi_{kl}(X)$ is described via
\begin{multline*}
\Nphi_{kl}(Xs^{\bar D_n})\\=\diag(\one_{n-1},s^{-1}\one_n,\dots,s^{-n+k+l+2}\one_n,s^{-n+k+l+1}\one_{n-l+1},
s^{-n+k+l}\one_{n-k})\\
\times\Nphi_{kl}(X)\diag(\bar D_{n+1},\bar D_{n+1}+\one_{n+1},\dots, \bar D_{n+1}+(n-k-l-1)\one_{n+1}).
\end{multline*} 
Consequently,
\begin{multline*}
\bar\xi_{\phi_{kl}}=\frac{(n+1)(n+2)}2(n-k-l)+(1+\dots+n-k-l-1)(n+1)-(1+\dots+n-k-l-2)n\\
-(n-k-l-1)(n-l+1)-(n-k-l)(n-k)\\
=\frac{(n+1)(n+2)}2(n-k-l)-\frac{(n-k-l-1)(n-k-l+2)}2-n+k.
\end{multline*}
\end{proof} 

\subsection{Compatibility}
The proof of the compatibility statement in Theorem~\ref{initialseed} relies on~\cite[Proposition~2.1]{GSVnewdouble}
which is reproduced below for readers’ convenience. 

\begin{proposition}\label{compatchar}
Assume that all Laurent monomials 
\[
\hat p_{kr}=\frac{\left(p_{kr}v_{k;>}^{[r]}v_{k;<}^{[d_k-r]}\right)^{d_k}}
{\left(v_{k;>}^{[d_k]}\right)^r \left(v_{k;<}^{[d_k]}\right)^{d_k-r}} 
\]
are Casimirs of the bracket $\Poi$ and for any mutable vertex $\alpha\in Q$
\begin{equation}\label{xybracket}
\{x_\beta,y_\alpha\}=\lambda d_\alpha\delta_{\alpha\beta}x_\beta y_\alpha\quad\text{for any $\beta\in Q$,}
\end{equation}
where $d_\alpha$ is the multiplicity at $\alpha$, $y_\alpha$ is the $y$-variable at $\alpha$ given
 by~{\rm\eqref{yvar}}, $\lambda$ is a rational number not depending on $\alpha$, and $\delta_{\alpha\beta}$ is the Kronecker symbol.
 Then the bracket $\Poi$ is compatible with $\GCC(\Sigma)$.
\end{proposition}

We start with checking the condition on monomials $\hat p_{kr}$ in Proposition~\ref{xybracket}. In our case $k=1$, and the only generalized exchange relation is given by~\eqref{ger}:
\[
\phi_{11}\phi'_{11}=x_{1n}\phi_{12}^{n-1}+\sum_{r=1}^{n-2} \frac{\ec_r}{\det X}\phi_{12}^{n-1-r}\phi_{21}^r+
x_{n1}\phi_{21}^{n-1}.
\]
Therefore, $d_1=n-1$, $v_{1;>}^{[d_1]}=x_{n1}$, $v_{1;<}^{[d_1]}=x_{1n}$, $v_{1;>}^{[r]}=v_{1;<}^{[d_1-r]}=1$ for 
$r\in[1,n-2]$, $p_{1r}=\ec_r/\det X$, so that
\begin{equation}\label{casimir}
\hat p_{1r}=\left(\frac{\ec_r}{x_{1n}}\right)^{n-1}\left(\frac{x_{1n}}{x_{n1}}\right)^r\frac1{\det X^{n-1}}.
\end{equation}

Note that $\det X$ is a Casimir by construction (see the paragraph iimediately preceding Theorem~\ref{thm:main}), and 
$x_{1n}/x_{n1}$ is a Casimir by Lemma~\ref{xn1/x1n}. Further, by~\eqref{badc}, $\ec_r/x_{1n}=c_r(\Psi'(X))\det X$, so it remains to prove that $c_r(\Psi'(X))$ is a Casimir of the bracket $\Poi$. Recall that $c_r(U)$ is a Casimir of the
bracket $\Poi^\dag$, as explained in~\cite[Section 3.2]{GSVnewdouble}, consequently, by Theorem~\ref{twomaps}(i), 
$\{c_r(\Psi'(X)),f(\Psi'(X))\}$ vanishes for any function $f(U)$. Further, by Theorem~\ref{twomaps}(iii), 
$\{c_r(\Psi'(X)),x_{1m}/x_{1n}\}=\{c_r(\Psi'(X)),x_{nm}/x_{1n}\}=0$ for all $m\in[1,n]$. To prove that $\{c_r(\Psi'(X)),x_{1n}\}=0$ we use Lemma~\ref{x1nbra} and the fact that $\xi_{c_r}=\bar\xi_{c_r}=-r$ (see the proof of Lemma~\ref{weights} above).

Lemma~\ref{weights} and the description of the quiver $Q_n$ in Section~\ref{outline} imply the following statement.

\begin{lemma}\label{yweights}
For any mutable vertex $\alpha\in Q_n$ the variable $y_\alpha$ satisfies~\eqref{eq:T-hom} and $\xi_{y_\alpha}=\bar\xi_{y_\alpha}=0$.
\end{lemma}

\begin{proof} A straightforward though tedious check based on~\eqref{explweights}. For example, for $\alpha=D$, by~\eqref{dexchange},
\[
y_D=\frac{\eg_{n-1,n-2}(X)\ef_{2n-4}(X)x_{n1}}{\eg_{n-2,n-2}(X)\ef_{2n-3}(X)},
\]
so that 
\begin{multline*}
\xi_{y_D}=\left(\frac{1{\cdot} 2}2+\frac{1{\cdot} 0}2+(n-1)n-1+(n-2)(n+1)+1\right)\\
-\left(\frac{0{\cdot} 1}2+\frac{2{\cdot} 1}2+(n-1)n-1+(n-2)(n+1)+1\right)=0
\end{multline*}
and
\begin{multline*}
\bar\xi_{y_D}=\left(\frac{n(n+1)}2+\frac{(n-1)(n-4)}2+n-1+n(2n-4)-(n-2)^2+1\right)\\
-\left(\frac{n(n+1)}2+\frac{(n-1)(n-4)}2+n-2+n(2n-3)-\frac{(2n-3)^2-1}4\right)=0.
\end{multline*}
\end{proof} 

In order to prove~\eqref{xybracket}, we distinguish the following three cases.

{\it Case 1:\/} $\alpha\in\Psi'^*Q_{n-1}^\dag$, $\alpha\ne D$.  Assume first that $\beta=(k,l)$, then
\[
\{x_\beta,y_\alpha\}=\{\phi_{kl}(X),y_\alpha(X)\}=\es_{kl}\left\{(x_{1n}\det X)^{n-k-l}\fy_{kl}(\Psi'(X)),y_\alpha(X)\right\}
\]
by~\eqref{badphi}. The latter bracket has three ingredients: 
\begin{align*}
\{&x_{1n}^{n-k-l},y_\alpha(X)\}=(n-k-l)\{x_{1n},y_\alpha(X)\},\\
\{&\det X^{n-k-l},y_\alpha(X)\}=(n-k-l)\{\det X,y_\alpha(X)\},
\end{align*}
and $\{\fy_{kl} (\Psi'(X)),y_\alpha(X)\}$. The first expression above vanishes by Lemma~\ref{x1nbra} and Lemma~\ref{yweights}. The second expression vanishes since $\det X$ is a Casimir of the bracket $\Poi$. Further,
\[
\{\fy_{kl} (\Psi'(X)),y_\alpha(X)\}=\{\fy_{kl} (\Psi'(X)),y_\alpha(\Psi'(X))\}=-\{\fy_{kl} (U),y_\alpha(U)\}^\dag
\]
by~\cite[Proposition 4.1 and Remark 4.8]{GSVpullback} and Theorem~\ref{twomaps}(i). It is proved in~\cite[Section 6.2]{GSVdouble} that the latter expression
vanishes for $(k,l)\ne\alpha$ and equals $\fy_{kl} (U)y_\alpha(U)$ for $(k,l)=\alpha$. Consequently,~\eqref{xybracket} holds with $\lambda=1$.
The case $\beta=\langle i,j\rangle$ is treated in the same way. Cases $\beta=A, B,C$ follow from Lemma~\ref{x1nbra} and Lemma~\ref{yweights} (for $A$ and $C$) and the Casimir property of $\det X$ (for $B$). Finally, cases $\beta=(m,+)$ and $\beta=(m,-)$ are treated similarly
based on Theorem~\ref{twomaps}(iii).

{\it Case 2:\/} $\alpha\in\Psi''^*\bar Q_{n-1}^T$, $\alpha\ne D$. This case is treated similarly to the previous one. The nontrivial ingredient of the bracket~\eqref{xybracket} is
\[
\{\bar t_m^\pm (\Psi''(X)),y_\alpha(X)\}=\{\bar t_m^\pm (\Psi''(X)),y_\alpha(\Psi''(X))\}=\{\bar t_m^\pm(\bar M),y_\alpha(\bar M)\}^T
\]
by Theorem~\ref{twomaps}(ii). It is proved in~\cite[Lemma 5.1]{GSV_Acta} that the latter expression
vanishes for $(m,\pm)\ne\alpha$ and equals $\bar t_m^\pm (\bar M)y_\alpha(\bar M)$ for $(m,\pm)=\alpha$. Consequently,~\eqref{xybracket} holds with $\lambda=1$.

{\it Case 3:\/} $\alpha=D$. Write $y_D(X)=\eg(X)\ef(X)$ with
\begin{align*}
\eg(X)&=\frac{\eg_{n-1,n-2}(X)}{\eg_{n-2,n-2}(X)}=\pm\frac{g_{n-1,n-2}(\Psi'(X))}{g_{n-2,n-2}(\Psi'(X))},\\
\ef(X)&=\frac{\ef_{2n-4}(X)x_{1n}}{\ef_{2n-3}(X)}=\pm\frac{\bar t^+_{n-2}(\Psi''(X))}{\bar t^-_{n-1}(\Psi''(X))}
\end{align*}
(exact expressions for signs are inessential since they do not enter into the final formula). 
Consider fist the case $\beta=D$. 
The bracket $\{x_D,y_D(X)\}$ has two ingredients: $\{x_D,\eg(X)\}$ and $\{x_D,\ef(X)\}$. To treat the first one we write
$x_D=\ef_{2n-2}(X)=x_{1n}^{2n-2}\bar t^+_{n-1}(\Psi''(X))$, so that
\begin{multline*}
\{x_D,\eg(X)\}=x_{1n}^{2n-2}\left\{\bar t^+_{n-1}(\Psi''(X)), \frac{g_{n-1,n-2}(\Psi'(X))}{g_{n-2,n-2}(\Psi'(X))}\right\}
\\+(2n-2)x_{1n}^{2n-3}\bar t^+_{n-1}(\Psi''(X))\{x_{1n},\eg(X)\}.
\end{multline*}
The first bracket above vanishes due to Theorem~\ref{twomaps}(iii). The second bracket equals $\frac{x_{1n}\eg(X)}n
(\bar\xi_\eg-\xi_\eg)=\frac{x_{1n}\eg(X)}n$ by Lemma~\ref{x1nbra}. 

To treat the second ingredient of $\{x_D,y_D(X)\}$ we write
$x_D=\eg_{n-1,n-1}(X)=\pm x_{1n}^{n-2}\det Xg_{n-1,n-1}(\Psi'(X))$, so that up to trivial terms
\begin{multline*}
\{x_D,\ef(X)\}=x_{1n}^{n-2}\det X\left\{g_{n-1,n-1}(\Psi'(X)),\frac{\bar t^+_{n-2}(\Psi''(X))}
{\bar t^-_{n-1}(\Psi''(X))}\right\}\\
+(n-2)x_{1n}^{n-3}\det X\{x_{1n},\ef(X)\}.
\end{multline*}
The first bracket above vanishes for the same reason as before. The second bracket equals $\frac{x_{1n}\ef(X)}n
(\bar\xi_\ef-\xi_\ef)=-\frac{x_{1n}\ef(X)}n$ by Lemma~\ref{x1nbra}. Summing up, we get
\[
\{x_D,y_D(X)\}=\left(\frac{2n-2}n-\frac{n-2}n\right)x_Dy_D(X)=x_Dy_D(X),
\]
hence~\eqref{xybracket} in this case is satisfied with $\lambda=1$.

Assume now that $\beta=\langle i,j\rangle\ne D$. Similarly to above, the bracket $\{\eg_{ij}(X),y_D(X)\}$ has two components: $\{\eg_{ij}(X),\ef(X)\}$ and $\{\eg_{ij}(X),\eg(X)\}$. For the first one we use~\eqref{badg}, 
Lemma~\ref{x1nbra}, computations for the weights $\xi_\ef$ and $\bar\xi_\ef$, and Theorem~\ref{twomaps}(iii) to get
\begin{multline*}
\{\eg_{ij}(X),\ef(X)\}=\{x_{1n}^{j-1}\det X g_{ij}(\Psi'(X)),\ef(X)\}\\
=\frac{j-1}nx_{1n}^{j-2}\det X g_{ij}(\Psi'(X))
\{x_{1n},\ef(X)\}=-\frac{j-1}n\eg_{ij}(X)\ef(X)
\end{multline*}
(all the other components of the bracket vanish). For the second ingredient we use~\eqref{badg}, 
Lemma~\ref{x1nbra}, computations for the weights $\xi_\eg$ and $\bar\xi_\eg$, Theorem~\ref{twomaps}(i), 
and~\cite[Proposition 4.1]{GSVpullback} to get
\begin{multline*}
\{\eg_{ij}(X),\eg(X)\}=\{x_{1n}^{j-1}\det X g_{ij}(\Psi'(X)),\eg(X)\}\\
=\frac{j-1}nx_{1n}^{j-2}\det X g_{ij}(\Psi'(X))
\{x_{1n},\eg(X)\}\\
+x_{1n}^{j-1}\det X\left\{g_{ij}(\Psi'(X)),\frac{g_{n-1,n-2}(\Psi'(X))}{g_{n-2,n-2}(\Psi'(X))}\right\}\\
=\frac{j-1}n\eg_{ij}(X)\eg(X)-x_{1n}^{j-1}\det X\left\{g_{ij}(U),\frac{g_{n-1,n-2}(U)}{g_{n-2,n-2}(U)}\right\}^\dag
\end{multline*}
with $U=\Psi'(X)\in GL_{n-1}$.
The contribution of the first term above cancels the contribution of the first ingredient. To compute the second term, introduce variables $z_j(U)=g_{n-1,j}(U)/g_{jj}(U)$ for $j\in[2,n-1]$, so that $z_{n-1}(U)=1$. Additionally, define
$z_1(U)=\fy_{n-2,1}(U)/\det U$. It is easy to derive from the description of the quiver $Q_{n-1}^\dag$ that
\begin{equation}\label{yprod}
y_j(U)=\prod_{i=j}^{n-1}y_{ij}(U)=\frac{z_{j-1}(U)z_{j+1}(U)}{z_j^2(U)}, \qquad j\in[2,n-2].
\end{equation}
By the compatibility of the bracket $\Poi^\dag$ and the generalized cluster structure $\G\CC_{n-1}^\dag$ (see 
Theorem~\ref{dualgroupcs}), we have $\{g_{ij}(U),y_m(U)\}^\dag=-\delta_{jm}g_{ij}(U)y_m(U)$ for $i\in[j,n-1]$.
Introducing constants $\omega_m^{ij}$ via $\{g_{ij}(U),z_m(U)\}^\dag=\omega_m^{ij}g_{ij}(U)z_m(U)$, we rewrite the above relation using~\eqref{yprod} as a linear system
\[
\omega_{m-1}^{ij}+\omega_{m+1}^{ij}-2\omega_m^{ij}=-\delta_{mj},\quad m\in[2,n-2],\qquad \omega_{n-1}^{ij}=0,
\]
which yields $\omega_1^{ij}=-(j-1)+(n-2)\omega_{n-2}^{ij}$.

To proceed further we need the following statement.

\begin{lemma}\label{omega1}
For all $j\in[2,n-2]$ and $i\in[j,n-1]$, $\omega_1^{ij}=-(j-1)$.
\end{lemma}

\begin{proof}
Since $\fy_{n-2,1}(U)=u_{n-1,1}$, we have to prove that 
\begin{equation}\label{eq:omega1}
\{u_{n-1,1},g_{ij}(U)\}^\dag=(j-1)u_{n-1,1}g_{ij}(U). 
\end{equation}
In Section~\ref{badgproof} we defined a map $U\mapsto H(U)$ such that functions $g_{ij}(U)$ can be written via flag minors of $H(U)$, see~\eqref{eq:fopflags} where $n$ should be replaced by $n-1$.  By Corollary~\ref{cor:Fprincipal}(i), 
$H(U)$ satisfies the relation $U=H(U)\bgamma(H(U)_+)^{-1}$. We refactor the right hand side and define $\widehat{H}(U)=
\bgamma(H(U)_+)^{-1}H(U)$. Note that $\det [\widehat{H}(U)]_{[i,n-1]}^{[j,n+j-i-1]}=\det[H(U)]_{[i,n-1]}^{[j,n+j-i-1]}$ 
for $j\le i$. Following~\cite{GSVuni}, define the bracket $\Poi_{\varnothing}^\dag$ as the 
bracket~\eqref{dualbra} with $R_+=R_0+\pi_>$ where $R_0$ is given by~\eqref{R0}. By~\cite[Proposition 3.15]{GV}, the map
$(GL_{n-1},\Poi_{\varnothing}^\dag)\to(GL_{n-1},\Poi^\dag)$ given by $\widehat{H}(U)\mapsto U$ is Poisson. Consequently,
by~\eqref{eq:fopflags},
\begin{multline}\label{briq}
\{u_{n-1,1},g_{ij}(U)\}^\dag=\left(\left\{\log v_{n-1,1}, \log\det[H(V)]_{[i,n-1]}^{[j,n+j-i-1]}\right\}_\varnothing^\dag\right.\\
+\left.\sum_{s=2}^{j-1}\left\{\log v_{n-1,1}, \log\det[H(V)]_{[s,n-1]}^{[s,n-1]}\right\}_\varnothing^\dag\right)
u_{n-1,1}g_{ij}(U).
\end{multline}

To proceed further, denote $N=N(V)=V_{0,-}$ and $B=B(V)=V_+$, then for any function $f$ on $\B_-$ we have
\[
(\nabla f\circ N){\cdot} V=\Ad_{N^{-1}}(N\nabla f)_\ge,\qquad
V{\cdot}(\nabla f\circ N)=\Ad_B(N\nabla f)_\ge\in\b_+.
\]
Therefore, for any two functions $f_1$, $f_2$ on $\B_-$ we get
\begin{multline}\label{interimbra}
\{f_1\circ N,f_2\circ N\}_\varnothing^\dag\\=
\left\langle(R_0+\pi_>)\left(\Ad_{N^{-1}}(N\nabla_1)_\ge-\Ad_B(N\nabla_1)_\ge\right),
\Ad_{N^{-1}}(N\nabla_2)_\ge-\Ad_B(N\nabla_2)_\ge\right\rangle\\
-\left\langle\Ad_{N^{-1}}(N\nabla_1)_\ge-\Ad_B(N\nabla_1)_\ge,\Ad_{N^{-1}}(N\nabla_2)_\ge\right\rangle
\end{multline}
with $\nabla_1=\nabla f_1$, $\nabla_2=\nabla f_2$. The second term above can be rewritten as
\[
\langle (N\nabla_1)_\ge,(N\nabla_2)_\ge\rangle-\left\langle\Ad_B(N\nabla_1)_\ge,
\Ad_{N^{-1}}(N\nabla_2)_\ge\right\rangle,
\]
which is equal to
\begin{multline*}
%\langle (N\nabla_1)_\ge,(N\nabla_2)_\ge\rangle-\left\langle\Ad_B(N\nabla_1)_\ge,
%\Ad_{N^{-1}}(N\nabla_2)_\ge\right\rangle\\=
\hskip-8pt\langle (N\nabla_1)_0,(N\nabla_2)_0\rangle-\left\langle(N\nabla_1)_0,\Ad_{N^{-1}}(N\nabla_2)_\ge\right\rangle
-\left\langle(\Ad_B(N\nabla_1)_\ge)_>,\Ad_{N^{-1}}(N\nabla_2)_\ge\right\rangle\\
=\langle (N\nabla_1)_0,(N\nabla_2)_0\rangle-\langle (N\nabla_1)_0,(\nabla_2{\cdot} N)_0\rangle
-\left\langle(\Ad_B(N\nabla_1)_\ge)_>,\Ad_{N^{-1}}(N\nabla_2)_\ge\right\rangle.
\end{multline*}
To treat the first term, note that
\begin{align*}
(R_0&+\pi_>)\left(\Ad_{N^{-1}}(N\nabla_1)_\ge-\Ad_B(N\nabla_1)_\ge\right)\\
&=(R_0+\pi_>)(\nabla_1{\cdot} N)-R_0(N\nabla_1)_0-(\Ad_B(N\nabla_1)_\ge)_>\\
&=R_0([\nabla_1,N]_0)+(\nabla_1{\cdot} N)_>-(\Ad_B(N\nabla_1)_\ge)_>.
\end{align*}
Note that $\left\langle \left(\nabla_1{\cdot} N-\Ad_B(N\nabla_1)_\ge\right)_>,\Ad_B(N\nabla_2)_\ge\right\rangle$ vanishes
since it is the pairing of an element in $\n_+$ and an element in $\b_+$. Further, 
\begin{multline*}
\left\langle (\nabla_1{\cdot} N)_>,\Ad_{N^{-1}}(N\nabla_2)_\ge\right\rangle=
\left\langle (\nabla_1{\cdot} N)_>,\nabla_2{\cdot} N -\Ad_{N^{-1}}(N\nabla_2)_<\right\rangle\\
=\left\langle (\nabla_1{\cdot} N)_>,\nabla_2{\cdot} N\right\rangle-\left\langle N\nabla_1, (N\nabla_2)_<\right\rangle
=\left\langle (\nabla_1{\cdot} N)_>,\nabla_2{\cdot} N\right\rangle-\left\langle (N\nabla_1)_>, N\nabla_2\right\rangle,
\end{multline*}
so that the first term in~\eqref{interimbra} equals
\begin{multline*}
\left\langle R_0([\nabla_1,N]_0)+(\nabla_1{\cdot} N)_>-(\Ad_B(N\nabla_1)_\ge)_>,
\Ad_{N^{-1}}(N\nabla_2)_\ge-\Ad_B(N\nabla_2)_\ge\right\rangle\\
=\left\langle R_0([\nabla_1,N]_0),[\nabla_2,N]_0\right\rangle+
\left\langle (\nabla_1{\cdot} N)_>,\nabla_2{\cdot} N\right\rangle-\left\langle (N\nabla_1)_>, N\nabla_2\right\rangle\\
-\left\langle (\Ad_B(N\nabla_1)_\ge)_>,\Ad_{N^{-1}}(N\nabla_2)_\ge\right\rangle.
\end{multline*}
Therefore,
\begin{multline*}
\{f_1\circ N,f_2\circ N\}_\varnothing^\dag=\left\langle R_0([\nabla_1,N]_0),[\nabla_2,N]_0\right\rangle
-\langle (N\nabla_1)_0,(N\nabla_2)_0\rangle\\+\langle (N\nabla_1)_0,(\nabla_2{\cdot} N)_0\rangle
+\left\langle (\nabla_1{\cdot} N)_>,\nabla_2{\cdot} N\right\rangle-\left\langle (N\nabla_1)_>, N\nabla_2\right\rangle.
\end{multline*}
Recall that the standard Sklyanin bracket is given by
\begin{multline*}
\{f_1\circ N,f_2\circ N\}_\varnothing=\frac12\langle(\nabla_1{\cdot} N)_0,(\nabla_2{\cdot} N)_0\rangle
+\left\langle (\nabla_1{\cdot} N)_>,\nabla_2{\cdot} N\right\rangle\\-\frac12\langle (N\nabla_1)_0,(N\nabla_2)_0\rangle
-\left\langle (N\nabla_1)_>, N\nabla_2\right\rangle,
\end{multline*}
see, e.g., \cite[p.~9]{GSVMem}, so that finally
\begin{multline}\label{finalbra}
\{f_1\circ N,f_2\circ N\}_\varnothing^\dag=\{f_1\circ N,f_2\circ N\}_\varnothing+\left\langle R_0([\nabla_1,N]_0),[\nabla_2,N]_0\right\rangle\\
+\langle (N\nabla_1)_0,(\nabla_2{\cdot} N)_0\rangle-\frac12\left(\langle (N\nabla_1)_0,(N\nabla_2)_0\rangle
+\langle(\nabla_1{\cdot} N)_0,(\nabla_2{\cdot} N)_0\rangle\right).
\end{multline}

Note that the functions $\log v_{n-1,n}$ and $\log\det[H(V)]_{[i,n-1]}^{[j,n+j-i-1]}$ that feature in the right hand side of~\eqref{briq} are, in fact, functions on $\B_-$, so to complete the computation we can use~\eqref{finalbra} for $$
f_1=\log n_{n-1,1}\qquad \text{and}\qquad f_2=\log\det[H(N)]_{[i,n-1]}^{[j,n+j-i-1]}, i\ge j\ge 2.
$$
 Recall that the standard Sklyanin bracket for the entries of $N$ is given by
\[
\{n_{ij},n_{kl}\}_\varnothing=\frac12(\sign(k-i)+\sign(l-j))n_{il}n_{kj},
\]
see, e.g.,~\cite[p.~12]{GSVb}. In particular,
\[
\{n_{n-1,1},n_{kl}\}_\varnothing=\begin{cases} 0\qquad &\text{for $k<n-1$,}\\
\frac12n_{n-1,1}n_{kl}\qquad &\text{for $k=n-1$, $l>1$.} \end{cases}
\]
Using the Laplace expansion of $\det[H(N)]_{[i,n-1]}^{[j,n+j-i-1]}$ with respect to the last row, we then obtain
\[
\{n_{n-1,1},\det[H(N)]_{[i,n-1]}^{[j,n+j-i-1]}\}_\varnothing=\frac12n_{n-1,n}\det[H(N)]_{[i,n-1]}^{[j,n+j-i-1]},
\]
so for $f_1$ and $f_2$ as above we get $\{f_1\circ N,f_2\circ N\}_\varnothing=1/2$. Further
\begin{gather*}
(\nabla_1{\cdot}N)_0=e_{11}, \qquad (\nabla_2{\cdot}N)_0=e_{jj}+\dots+e_{n+j-i-1,n+j-i-1},\\
(N\nabla_1)_0=e_{n-1,n-1}, \qquad (N\nabla_2)_0=e_{ii}+\dots+e_{n-1,n-1}.
\end{gather*}
As we have seen in Section~\ref{twomapsiproof} above, $R_0$ acts as $1/(1-\gamma)$ on $\sl_{n-1}$, hence
\[
R_0([\nabla_1,N]_0)=\frac1{(1-\gamma)}(e_{11}-e_{n-1,n-1})=\one_{n-1}-e_{n-1,n-1},
\]
and so $\left\langle R_0([\nabla_1,N]_0),[\nabla_2,N]_0\right\rangle=1-\delta_{ij}$. Next,
\[
\langle (N\nabla_1)_0,(\nabla_2{\cdot} N)_0\rangle=\delta_{ij},\quad
\langle (N\nabla_1)_0,(N\nabla_2)_0\rangle=1,\quad
\langle(\nabla_1{\cdot} N)_0,(\nabla_2{\cdot} N)_0\rangle=0,
\]
so that finally~\eqref{finalbra} gives $\{f_1\circ N,f_2\circ N\}_\varnothing^\dag=1$. Plugging into~\eqref{briq} 
yields~\eqref{eq:omega1} as required.
\end{proof}

It follows immediately that $\omega_{n-2}^{ij}=0$, so that $\left\{g_{ij}(U),\frac{g_{n-1,n-2}(U)}{g_{n-2,n-2}(U)}\right\}^\dag$ vanishes, and hence~\eqref{xybracket} holds for $\beta=\langle i,j\rangle\ne D$.

The case $\beta=(k,l)$ is treated in a similar way leading to
\[
\{\phi_{kl}(X),\ef(X)\}=-\frac{n-k-l}n\phi_{kl}(X)\ef(X)
\]
and
\begin{multline*}
\{\phi_{kl}(X),\eg(X)\}=\frac{n-k-l}n\phi_{kl}(X)\eg(X)\\-
(x_{1n}\det X)^{n-k-l}\left\{\fy_{kl}(U),\frac{g_{n-1,n-2}(U)}{g_{n-2,n-2}(U)}\right\}^\dag.
\end{multline*}
To treat the second term in the latter expression we introduce constants $\bar\omega_m^{kl}$ via $\{\fy_{kl}(U),z_m(U)\}^\dag=\bar\omega_m^{kl}\fy_{kl}(U)z_m(U)$. By the compatibility of the bracket $\Poi^\dag$ and the generalized cluster structure $\G\CC_{n-1}^\dag$, we have $\{\fy_{kl}(U),y_m(U)\}^\dag=0$, which can be rewritten using~\eqref{yprod} as a linear system
\[
\bar\omega_{m-1}^{kl}+\bar\omega_{m+1}^{kl}-2\bar\omega_m^{kl}=0,\quad m\in[2,n-2],\qquad \bar\omega_{n-1}^{kl}=0.
\]
Therefore, $\bar\omega_1^{kl}=(n-2)\bar\omega_{n-2}^{kl}$.

Similarly to the previous case, we need the following statement.

\begin{lemma}\label{baromega1}
For all $k$, $l$ such that $1\le k,l\le n-2$, $k+l\le n-1$, $\bar\omega_1^{kl}=0$.
\end{lemma}

\begin{proof}
According to the definition of $\bar\omega_1^{kl}$, we have to prove that 
\[
\{\fy_{kl}(U),\fy_{n-2,1}(U)\}^\dag=0
\]
with the bracket $\Poi^\dag$ given by~\eqref{dualbra} and $R_+$ given by~\eqref{Rplusgamma}. Note that every $\fy_{kl}(U)$ is invariant under the conjugation by elements of $\N_+$, while $\fy_{n-2,1}(U)$ is additionally both left and right $\N_+$-invariant. Consequently, $[\nabla\fy_{kl}(u),U]\in\n_+$ and $\nabla\fy_{n-2,1}(U){\cdot} U\in\n_+$, hence
\begin{multline*}
\{\fy_{kl}(U),\fy_{n-2,1}(U)\}^\dag\\=\langle R_0([\nabla \fy_{kl}(U),U]_0),[\nabla\fy_{n-2,1}(U),U]_0\rangle
-\langle[\nabla \fy_{kl}(U),U]_0,(\nabla\fy_{n-2,1}(U){\cdot} U)_0\rangle
\end{multline*}
with $R_0$ given by~\eqref{R0}. To proceed further we compute the weights $w_{kl}=(w_{kl}^1,\dots,\allowbreak 
w_{kl}^{n-1})$ of the conjugation action by diagonal matrices $D=\diag(d_1,\dots,d_{n-1})$ given by
\[
\fy_{kl}(DUD^{-1})=\prod_{i=1}^{n-1}d_i^{w_{kl}^i}\fy_{kl}(U).
\]
Clearly, $[\nabla \fy_{kl}(U),U]_0=-w_{kl}\fy_{kl}(U)$.

We have
\begin{multline*}\nar
\fy_{kl}(DUD^{-1})=\pm\det\begin{bmatrix}\one_{n-1}^{[1,k]}& (DUD^{-1})^{[1,l]}& (DUD^{-1})^2e_1& \dots&  (DUD^{-1})^{n-k-l}e_1\end{bmatrix}\\
=\pm d_1{\cdot} d_{n-1}\det\nar\begin{bmatrix}(D^{-1})^{[1,k]}& (UD^{-1})^{[1,l]}& U^2D^{-1}e_1& \dots & U^{n-k-l}D^{-1}e_1\end{bmatrix}\\
=\frac{d_{k+1}{\cdot} d_{n-1}}{d_1{\cdot} d_l}\fy_{kl}(U)d_1^{-(n-k-l-1)}=\frac{d_{k+1}{\cdot} d_{n-1}}
{d_1^{n-k-l}d_2{\cdot} d_l}\fy_{kl}(U),
\end{multline*}
so that
$w_{kl}=(0,\dots,0,1,\dots,1)-(n-k-l,1,\dots,1,0,\dots,0)$,
where the first vector has $k$ zeros, and the second one has $n-l-1$ zeros. In particular, $w_{n-2,1}=(-1,0,\dots,0,1)$.

Further, the first two terms in~\eqref{R0} vanish, and the the third term is proportional to $\one_{n-1}$, so its contribution vanishes. The contribution of the remaining term is 
\[
-\left\langle\frac{\gamma^*}{1-\gamma^*}w_{kl},w_{n-2,1}\right\rangle\fy_{kl}(U)\fy_{n-2,1}(U)=(n-k-l)\fy_{kl}(U)\fy_{n-2,1}(U).
\]
 On the other hand,  
\begin{multline*}
\left\langle[\nabla \fy_{kl}(U),U]_0,(\nabla\fy_{n-2,1}(U){\cdot} U)_0\right\rangle\\=
-\langle w_{kl},(1,0,\dots,0)\rangle\fy_{kl}(U)\fy_{n-2,1}(U)=(n-k-l)\fy_{kl}(U)\fy_{n-2,1}(U),
\end{multline*}
and hence the bracket $\{\fy_{kl}(U),\fy_{n-2,1}(U)\}^\dag$ vanishes.
\end{proof}

It follows immediately that $\bar\omega_{n-2}^{kl}=0$, so that $\left\{\fy_{kl}(U),\frac{g_{n-1,n-2}(U)}{g_{n-2,n-2}(U)}\right\}^\dag$ vanishes, and hence~\eqref{xybracket} holds for $\beta=(k,l)$.

Assume now that $\beta=(m,+)\ne D$. Similarly to the two previous cases, we have
\[
\{\ef_{2m}(X),\eg(X)\}=\frac {2m}n \ef_{2m}(X)\eg(X)
\]
and
\[
\{\ef_{2m}(X),\ef(X)\}=-\frac {2m}n \ef_{2m}(X)\ef(X)+x_{1n}^{2m}\left\{\bar t_m^+(\bar M),\frac{\bar t^+_{n-2}(\bar M)}
{\bar t^-_{n-1}(\bar M)}\right\}^T.
\]
The contributions of the former expression and of the first term in the latter cancel each other. To treat the remaining term we introduce $\zeta_j(\bar M)=\bar t_j^+(\bar M)/\bar t_{j+1}^-(\bar M)$ for $j\in[1,n-2]$ with
$\zeta_{n-1}(\bar M)=0$ 
and $\zeta_0=1/\bar h_0$. It is easy to derive from the description of the quiver $\bar Q_{n-1}^T$ that
\begin{equation}\label{baryprod}
\bar y_j(\bar M)=y_{j,+}(\bar M)y_{j+1,-}(\bar M)=\frac{\zeta_{j-1}(\bar M)\zeta_{j+1}(\bar M)}{\zeta_j^2(\bar M)}, \qquad j\in[1,n-2],
\end{equation}
similarly to~\eqref{yprod}. By the compatibility of the bracket $\Poi^T$ and the cluster structure $\bar\CC_{n-1}^T$ (see 
Theorem~\ref{todacs}), we have $\{\bar t_m^+(\bar M),\bar y_j(\bar M)\}^T=\delta_{jm}\bar t_m^+(\bar M)\bar y_j(\bar M)$.
Introducing constants $\omega_j^{m}$ via $\{\bar t_m^+(\bar M),\zeta_j(\bar M)\}^T=\omega_j^{m}\bar t_m^+(\bar M)\zeta_j(\bar M)$, we rewrite the above relation using~\eqref{baryprod} as a linear system
\[
\omega_{j-1}^{m}+\omega_{j+1}^{m}-2\omega_j^{m}=\delta_{mj},\quad j\in[1,n-3],
\]
which yields $\omega_0^{m}=m+(n-1)\omega_{n-2}^{m}$. 

To proceed further we need the following statement.

\begin{lemma}\label{omega0}
$\omega_0^m=m$.
\end{lemma}

\begin{proof} Recall that $\zeta_0(\bar M)=1/\bar h_0$, so we have to evaluate $\omega_0^m$ from 
$\{\bar h_0,\bar t_m^+(\bar M)\}^T=\omega_0^m\bar h_0\bar t_m^+(\bar M)$. Note that $\omega_0^m$ can be evaluated by looking at any term in the product $\bar h_0\bar t_m^+(\bar M)$ and finding with which coefficient it enters the left hand side. Let us consider the term $\bar h_0\bar h_m^m$ that represents the antidiagonal of the Hankel matrix $H_m^+$. By~\eqref{bartodabramoment}, the only source of such a term in the left hand side is the bracket 
$\{\bar h_0,\bar h_m^m\}^T=mh_m^{m-1}\{\bar h_0,\bar h_m\}^T=m\sum_{k=0}^{m-1}\bar h_{k+1}\bar h_{m-k-1}$, which
contains the term $m\bar h_0\bar h_m^m$, so that $\omega_0^m=m$.
\end{proof}

As an immediate corollary we get $\omega_{n-2}^{m}=0$, wchih means that the bracket~\eqref{xybracket} vanishes for 
$\beta=(m,+)\ne D$.

The case $\beta=(m,-)$ is treated similarly. The corresponding equations are  $\{\bar t_m^-(\bar M),\bar y_j(\bar M)\}^T=\delta_{j-1,m}\bar t_m^-(\bar M)\bar y_j(\bar M)$, and they are translated to
\[
\bar\omega_{j-1}^{m}+\bar\omega_{j+1}^{m}-2\bar\omega_j^{m}=\delta_{m,j-1},\quad j\in[1,n-3]
\]
for constants $\bar\omega_j^{m}$ defined via $\{\bar t_m^-(\bar M),\zeta_j(\bar M)\}^T=\bar\omega_j^{m}\bar t_m^-(\bar M)\zeta_j(\bar M)$. Consequently, $\bar\omega_0^{m}=m-1+(n-1)\bar\omega_{n-2}^{m}$.

The analog of Lemma~\ref{omega0} looks as follows.

\begin{lemma}\label{baromega0}
$\bar\omega_0^m=m-1$.
\end{lemma}

\begin{proof} Similarly to the proof of Lemma~\ref{omega0}, we evaluate $\bar\omega_0^m$ via looking at the term
$\bar h_0\bar h_{m-1}^m$ in the product $\bar h_0\bar t_m^-(\bar M)$. By~\eqref{bartodabramoment}, there are two sources of such a term in the bracket $\{\bar h_0,\bar t_m^-(\bar M)\}^T$: 
$\{\bar h_0,\bar h_{m-1}^m\}^T=mh_{m-1}^{m-2}\{\bar h_0,\bar h_{m-1}\}^T$, which
produces the term $m\bar h_0\bar h_{m-1}^m$, and $-\{\bar h_0,\bar h_0\bar h_{m-1}^{m-2}\bar h_{2m-2}\}^T$, which 
produces the term $-\bar h_0\bar h_{m-1}^m$
so that $\bar\omega_0^m=m-1$.
\end{proof}
As an immediate corollary we get $\bar\omega_{n-2}^{m}=0$, wchih means that the bracket~\eqref{xybracket} vanishes for 
$\beta=(m,-)$. Finally, the check of~\eqref{xybracket} for the case when $\beta$ is a frozen vertex, that is, $\beta=A, B, C$, is trivial, which completes the proof of the compatibility condition in Theorem~\ref{initialseed}. \qed

\subsection{Toric action} 
We can now prove Proposition~\ref{prop:torus}. Indeed, all cluster variables in $\FF_n$ are homogeneous polynomials. Further, a straightforward check shows that all right hand sides of exchange relations are homogeneous as well.
Next, by Lemma~\ref{weights}, all cluster variables in $\FF_n$ are homogeneous with respect to the action 
$X\mapsto t^{D_n}Xs^{\bar D_n}$, and by Lemma~\ref{yweights}, all right hand sides of exchange relations are homogeneous with respect to this action as well; for the generalized exchange relation this requires, additionally,
to check relations
\[
\xi_{\ec_r}-\xi_{\det X}+r(\xi_{\phi_{21}}-\xi_{\phi_{12}})=\xi_{x_{1n}},\qquad
\bar\xi_{\ec_r}-\bar\xi_{\det X}+r(\bar\xi_{\phi_{21}}-\bar\xi_{\phi_{12}})=\bar\xi_{x_{1n}},
\]
which follow immediately from~\eqref{explweights}. Finally, the Casimirs $\hat p_{1r}$ given by~\eqref{casimir} are homogeneous functions of weight~0 and
\[
\xi_{\hat p_{1r}}=\bar\xi_{\hat p_{1r}}=(n-1)\left(\frac{n(n+3)}2-r-n\right)+r(n-1)-(n-1)\frac{n(n+1)}2=0.
\]
Proposition~\ref{prop:torus} now follows from the standard condition of extendability of a local toric action for generalized cluster structures, see~\cite[Proposition 2.6]{GSVdouble}. \qed

\section{Regularity and Completeness}\label{regcomplet}

\subsection{Proof of Theorem~\ref{thm:InitLaurent}}\label{proofInitLaurent}
Our first goal is to invert the combined map $\Psi : (GL_n,\Poi)\to \widetilde{GL}_{n-1}^{\dag}\times_\C \bar\RR_{n-1}^T$ featuring in Theorem~\ref{twomaps}(iii). Denote
\begin{equation}
\label{eq:tildeXdef}
\widetilde{X} = X \begin{bmatrix} \one_{n-1} & 0 \\ -x_{1n}^{-1}X_{[1,1]}^{[1,n-1]} & 1\end{bmatrix} =   
\begin{bmatrix} 0 & x_{1n} \\ X_L & X^{[n,n]}_{[2,n-1]}\end{bmatrix},
\end{equation}
then the defining equation $X_L = U X_R$ for $U=U(X)$ can be rewritten as
\begin{equation}
\label{eq:tildeX}
\begin{bmatrix}{\widetilde{X}}^{[1,n-1]}_{[2,n-1]}\\ \widetilde{X}_{[n,n]}^{[1,n-1]}\end{bmatrix}  = 
U  \begin{bmatrix} 0 & x_{1n}\\ {\widetilde{X}}^{[1,n-1]}_{[2,n-1]}  & X^{[n,n]}_{[2,n-1]}\end{bmatrix} 
\begin{bmatrix} M \\ \mu \end{bmatrix},
\end{equation}
where
\begin{align*}
 M &= \begin{bmatrix}  -{x_{n1}}^{-1}X_{[n,n]}^{[2,n-1]} & - {x_{n1}}^{-1}{x_{nn}}\\ \one_{n-2} & 0 \end{bmatrix},\\
\mu&=-(x_{1n} x_{n1})^{-1}\begin{bmatrix}\det X_{1n}^{12} & \cdots & \det X_{1n}^{1n}\end{bmatrix}.
\end{align*}
We view~\eqref{eq:tildeX} as a linear equation on $\wx_{[2,n-1]}$.  To treat this equation, it will be convenient to 
perform a change of variables. Using Lemma~\ref{lemma:Fstable} as a guidance, denote by $V=V_+V_{0,-}$ 
%($V_{\leq 0}\in \B_-,  V_+\in \N_+$)  
the element of $GL_{n-1}$ related to $U$ via $U = V \gamma(V_+)^{-1}$. Set
\begin{equation}
\label{eq:YZdef}
Y= (V_+^{-1})_{[1,n-2]}^{[1,n-2]}{\wx}^{[1,n-1]}_{[2,n-1]},\qquad y=(V_+^{-1})_{[1,n-2]}^{[1,n-2]} X^{[n,n]}_{[2,n-1]},
\end{equation}
then~\eqref{eq:tildeX} becomes
\begin{equation}
\label{eq:YZ}
\begin{bmatrix}Y \\ 0 \end{bmatrix}  + 
v \ m=  V_{0,-}   \begin{bmatrix} 0 & x_{1n}\\ Y & y \end{bmatrix} \begin{bmatrix} M \\ \mu \end{bmatrix},
\end{equation}
where we denote 
\[
m =\wx_{[n,n]}^{[1,n-1]} = - {x_{1n}}^{-1} \begin{bmatrix} \det X_{1n}^{1n} & \dots & \det X_{1n}^{n-1,n}\end{bmatrix},
\qquad v=(V_+^{-1})^{[n-1,n-1]}.
\]

\begin{lemma}
\label{eq:vtomu}
Let
\[
C= \begin{bmatrix}\mu\\ \mu M \\ \vdots \\ \mu M^{n-2}\end{bmatrix},
\]
then $m=\begin{bmatrix}x_{n,n-1} &\dots & x_{n1}\end{bmatrix}C$.
\end{lemma}

\begin{proof} 
First, observe that by Cayley--Hamilton theorem,
\[
C M = \begin{bmatrix}  0 & \one_{n-2}\\- \frac{x_{nn}}{x_{n1}} & \cdots   - \frac{x_{n2}}{x_{n1}}  \end{bmatrix} C.
\]
Next, the identity
\[
x_{n, j+1} \det X_{1n}^{1n}  - x_{n1}\det X_{1n}^{j+1,n}  =  x_{nn}\det X_{1n}^{1, j+1}.
\]
implies $mM = -x_{nn} \mu$. Consequently,
\begin{align*}
 mC^{-1} &=  -x_{nn} \mu (CM)^{-1} =  -x_{nn}  \mu C^{-1}  \begin{bmatrix}  0 & \one_{n-2}\\  
 - \frac{x_{nn}}{x_{n1}} & \cdots   - \frac{x_{n2}}{x_{n1}}  \end{bmatrix} ^{-1}\\
 &  = -x_{nn}  e_1^T  \begin{bmatrix}  0 & \one_{n-2}\\   - \frac{x_{nn}}{x_{n1}} & \cdots   - \frac{x_{n2}}{x_{n1}}  \end{bmatrix} ^{-1} = \begin{bmatrix}x_{n,n-1} &\ldots & x_{n1}\end{bmatrix}.
\end{align*}
\end{proof}

Using Lemma~\ref{eq:vtomu}, we rewrite the system~\eqref{eq:YZ} once more by introducing 
\begin{equation}
\label{eq:subst1}
\wY = Y + v_{[1,n-2]} m, \qquad \tilde y = y + x_{nn} v_{[1,n-2]},
\end{equation}
so that~\eqref{eq:YZ} becomes
\begin{equation}
\label{eq:tildeYZ}
\begin{bmatrix} \wY \\ m \end{bmatrix}   =  V_{0,-}   \begin{bmatrix}  0 & x_{1n}\\ \wY & \tilde y \end{bmatrix} 
\begin{bmatrix} M \\ \mu \end{bmatrix}.
\end{equation}

To perform the last transformation of our linear system, consider $N\in \B_-$ such that $V_{0,-} = \gamma^*(N)^{-1}N$. It is easy to check that such $N$ is given by
\[
N= \cdots (\gamma^*)^2(V_{0,-}) \gamma^*(V_{0,-}) V_{0,-}.
\]
Note that the upper left entry of $N$ is equal to the product of all diagonal entries of  $V_{0,-}$, therefore
\begin{equation}
\label{eq:fibercond}
n_{11}=\det V_{0,-} = \det V = \det U = \frac{x_{n1}}{x_{1n}}.
\end{equation}
Introducing
\begin{equation}
\label{eq:subst2}
\Upsilon = N_{[2,n-1]}^{[2,n-1]} \wY,\qquad \smup =  N_{[2,n-1]}^{[2,n-1]} \tilde y + x_{1n} N_{[2,n-1]}^{[1,1]}, 
\end{equation}
we can use~\eqref{eq:fibercond} to further simplify~\eqref{eq:tildeYZ}:
\begin{equation}
\label{eq:greekYZ}
\begin{bmatrix} \Upsilon \\ m \end{bmatrix}   =   \begin{bmatrix} 0 &  x_{n1} \\ \Upsilon & \smup \end{bmatrix}
\begin{bmatrix} M \\ \mu \end{bmatrix}.
\end{equation}
 
\begin{lemma}
\label{greekYZsols}
 The solution of~\eqref{eq:greekYZ} is given by
\[
\smup_j= x_{n, j+1}, \quad j\in[1, n-2],\qquad  \Upsilon = \begin{bmatrix} x_{n1}& 0&\dots & 0 \\ x_{n2} & x_{n1}  &\dots &0\\ \vdots & \ddots & \ddots &0 \\  
x_{n,n-2}  &  \dots &  x_{n2}  & x_{n1}\end{bmatrix}C_{[1,n-2]}. 
\]
\end{lemma}

\begin{proof} Re-writing~\eqref{eq:greekYZ} row by row:
\begin{align*}
\Upsilon_{[1,1]} &= x_{n1} \mu,\\
\Upsilon_{[i+1,i+1]} &= \Upsilon_{[i,i]} M + \smup_i \mu, \quad i\in[1, n-3],\\
m &= \Upsilon_{[n-2,n-2]} M + \smup_{n-2} \mu,
\end{align*}
one sees that~\eqref{eq:greekYZ} is equivalent to
\begin{equation}
\label{eq:greekYp}
\begin{bmatrix} \Upsilon \\ m \end{bmatrix}   =   
\begin{bmatrix} x_{n1}&0 & \dots& 0 \\ \smup_1 & x_{n1}  &\dots &0\\ \vdots & \ddots & \ddots & 0\\
 \smup_{n-2} &  \dots & \smup_1 & x_{n1}\end{bmatrix} C.
\end{equation}
Comparing the last row above  with Lemma~\ref{eq:vtomu}, we obtain the desired result for $\smup_j$, and the formula for $\Upsilon$ follows.
\end{proof}

Now we can use \eqref{eq:subst2},\eqref{eq:subst1},\eqref{eq:YZdef},\eqref{eq:tildeXdef} to restore $X_{[2,n-1]}$ from $U, X_{\{1,n\}}$. First, define
\begin{gather*}
\Theta = \begin{bmatrix}   1 & 0\\ 0 & V_+  \end{bmatrix} \begin{bmatrix}   N^{-1} & 0\\ 0 & 1  \end{bmatrix},\\
L= \begin{bmatrix} x_{n1}& 0&\dots & 0 \\ x_{n2} & x_{n1}  &\dots &0\\ \vdots & \ddots & \ddots &0 \\  
x_{n,n-1}  &  \dots &  x_{n2}  & x_{n1}\end{bmatrix}C.
\end{gather*}

\begin{proposition} 
\label{prop:theta}
\begin{equation}
\label{eq:ThetaL}
X = \Theta \cdot \left ( \begin{bmatrix}0 & 0\\ L &0\end{bmatrix} + X_{[n,n]}^T x_{1n}^{-1} X_{[1,1]} \right ).
\end{equation}
\end{proposition}

\begin{proof}
From~\eqref{eq:subst2} and Lemma~\ref{greekYZsols},
\begin{align*}
\tilde y&= \left ( N_{[2,n-1]}^{[2,n-1]}\right ) ^{-1} (X_{[n,n]}^{[2,n-1]})^T -  
\left ( N_{[2,n-1]}^{[2,n-1]}\right ) ^{-1} N^{[1,1]}_{[2,n-1]} x_{1n} \\
&= \left ( N_{[2,n-1]}^{[2,n-1]}\right ) ^{-1} (X_{[n,n]}^{[2,n-1]})^T - 
e_1^T\frac{x_{n1}}{n_{11}} \\
&= \left ( N^{-1}\right )_{[2,n-1]}(X_{[n,n]}^{[1,n-1]})^T ,
\end{align*}
thus
\begin{align*}
y &= \tilde y - x_{nn} v_{[1,n-2]} = \left [ \left ( N^{-1}\right )_{[2,n-1]} \ (- v_{[1,n-2]})\right ] X_{[n,n]}^T,\\
X^{[n,n]}_{[2,n-1]}&=(V_+)_{[1,n-2]}^{[1,n-2]}\left[\left( N^{-1}\right )_{[2,n-1]}\  (- v_{[1,n-2]})\right]X_{[n,n]}^T.
\end{align*}
Note that
\[
-(V_+)_{[1,n-2]}^{[1,n-2]} v_{[1,n-2]} = (V_+)_{[1,n-2]}^{[n-1,n-1]}, 
\]
so that  we get
\[
X^{[n,n]}_{[2,n-1]}= (V_+)_{[1,n-2]} \begin{bmatrix}  \left ( N^{-1}\right )_{[2,n-1]}  & 0\\ 0 & 1  \end{bmatrix} 
X_{[n,n]}^T= \Theta_{[2,n-1]} X_{[n,n]}^T.
\]

Similarly,
\[
\wY= \left ( N_{[2,n-1]}^{[2,n-1]}\right ) ^{-1} \Upsilon,\qquad  
Y =  \left [  \left ( N_{[2,n-1]}^{[2,n-1]}\right )^{-1}  \  (- v_{[1,n-2]}) \right ] 
\begin{bmatrix}\Upsilon \\ m\end{bmatrix},
\]
and
\begin{align*}
\tilde X_{[2,n-1]}^{[1,n-1]}& = (V_+)_{[1,n-2]}^{[1,n-2]} Y = (V_+)_{[1,n-2]}^{[1,n-2]}  \left [   \left ( N_{[2,n-1]}^{[2,n-1]}\right )^{-1} \  (- v_{[1,n-2]}) \right ] 
\begin{bmatrix}\Upsilon \\ m\end{bmatrix}\\& =\Theta_{[2,n-1]}^{[2,n]}  L.
\end{align*}
Finally,
\begin{align*}
X_{[2,n-1]}^{[1,n-1]}& = \tilde X_{[2,n-1]}^{[1,n-1]} +  X^{[n,n]}_{[2,n-1]} x_{1n}^{-1}X_{[1,1]}^{[1,n-1]} \\
&=\Theta_{[2,n-1]} \left ( \begin{bmatrix}0\\ L\end{bmatrix} + X_{[n,n]}^T x_{1n}^{-1}X_{[1,1]}^{[1,n-1]}   \right ).
\end{align*}
Taking into account that $m=L_{[n-1,n-1]}$ by Lemma~\ref{eq:vtomu} and using~\eqref{eq:fibercond} we obtain the desired formula for $X$.
\end{proof}

To finish the proof of Theorem~\ref{thm:InitLaurent} we need the following auxiliary statement. 
Let $\overline{\UC}(\GCC_{n-1}^\dag)$ be the generalized upper cluster algebra defined by the generalized cluster structure $\GCC_{n-1}^\dag$ and $\overline{\UC}(\bar\CC_{n-1}^T)$ be the  upper cluster algebra defined by the cluster structure 
$\bar\CC_{n-1}^T$, where in both cases the ground ring is the corresponding polynomial ring in frozen variables. 
By Theorem~\ref{dualgroupcs} and Remark~\ref{grfordual}, $\UC(\GCC_{n-1}^\dag)$ is obtained from 
$\overline{\UC}(\GCC_{n-1}^\dag)$ by localizing at $\det U$. Similarly, by Theorem~\ref{todacs} and Remark~\ref{grfortoda}, $\UC(\bar\CC_{n-1}^T)$ is obtained from $\overline{\UC}(\bar\CC_{n-1}^T)$ by localizing at $t_{n-1}^+$. 

\begin{lemma}
\label{lem:upperalgebras}
Let $\psi'$ be an element of $\overline{\UC}(\GCC_{n-1}^\dag)$ and $\psi''$ be an element of $\overline{\UC}(\bar\CC_{n-1}^T)$. Then $\psi'(\Psi'(X))$ and $\psi''(\Psi''(X))$ are elements of the generalized upper cluster algebra $\UC(\GCC_n)$ divided by monomials in $x_{1n}$, $x_{n1}$, and $\det X$.
\end{lemma}

\begin{proof} We prove the claim for $\psi'$, a similar argument applies to $\psi''$ as well. In view of the general construction described in Section~\ref{outline} and of Proposition~\ref{readytoglue}, standard quasi-isomorphism  considerations imply that any cluster variable in $\GCC_{n-1}^\dag$ 
%=(\GCC(\DD_{n-1},Q_{n-1}^\dag, \P_{n-1})$ 
evaluated at the image of $\Psi'(X)$ is equal to the corresponding cluster variable in $\GCC_n$ times a Laurent monomial in frozen variables $x_{1n}$, $x_{n1}$, $\det X$. Since $\psi'\in \overline{\UC}(\GCC_{n-1}^\dag)$, it has a Laurent polynomial expression in all seeds adjacent to $(\DD_{n-1},Q_{n-1}^\dag, \P_{n-1})$. Therefore, $\psi'(\Psi'(X))$ has a Laurent polynomial expression in all seeds adjacent to the initial seed in $\GCC_n$ in all directions $(k,l)$ and $\langle k,l\rangle$. Direction $\langle n-1,n-1\rangle$ needs to be considered separately since the corresponding vertex is frozen in 
$\GCC_{n-1}^\dag$, but the variable $\eg_{n-1,n-1}(X)$ enters $\psi'(\Psi'(X))$ polynomially, hence the claim is still true. Finally, mutations in directions $(i,-)$, $i\in[1,n-1]$, $(i,+)$, $i\in[1, n-2]$, do not affect the Laurent polynomial expression of $\psi'(\Psi'(X))$. Thus $\psi'(\Psi'(X))$ has Laurent polynomial expression in the star of the initial seed in $\GCC_n$ with denominators possibly containing $x_{1n}$, $x_{n1}$, $\det X$, and the assertion of the lemma follows.
\end{proof}

We are now ready to complete the proof of Theorem~\ref{thm:InitLaurent}. By Proposition~\ref{prop:theta}, $X$ can be expressed as a product of two matrices. The entries of the second matrix are, up to a Laurent monomial in $x_{1n}$ 
and $x_{n1}$, polynomials in the entries of the first and the last row of $X$. Then the definition of the map $\Psi''$ implies that each of this entries is a polynomial $P$ in coefficients of $\bar q(\lambda)$, $p(\lambda)$  defined in~\eqref{qandp} times a Laurent monomial in $x_{1n}$ and $x_{n1}$. As explained above, $P$ belongs to the upper cluster algebra $\overline{\UC}(\bar\CC_{n-1}^T)$ localized at $t^+_{n-1}$. Using Lemma~\ref{lem:upperalgebras} 
and~\eqref{badbart}, we conclude that every matrix entry of the second factor in~\eqref{eq:ThetaL} is an element of 
$\UC(\GCC_n)$ divided by a monomial in $\ef_{2n-2}=\eg_{n-1,n-1}$ and $x_{1n}, x_{n1}$.

On the other hand, the entries of  $\Theta$ in~\eqref{eq:ThetaL} are rational functions in matrix entries of 
$U=\Psi'(X)$. It follows from the definitions of $V$, $N$, and $\Theta$ that denominators of these functions written in terms of matrix entries of $V$ are monomials in terms of  trailing principal minors of $V$. 
By Corollary~\ref{cor:Fprincipal} and~\eqref{eq:fopflags}, 
\[
\det V_{[j,n-1]}^{[j,n-1]} = \begin{cases}
    \dfrac{g_{jj}(U)}{g_{j-1,j-1}(U)} \quad &\text{if $j\geq 3$},\\
    g_{jj}(U) \quad &\text{otherwise}.
\end{cases}
\]
Moreover,~\eqref{Fminors} together with~\eqref{eq:fopflags} show that every matrix entry  $v_{ij}$ for $j>2$ is a polynomial in $U$ (and therefore an element of $\overline{\UC}(\GCC_{n-1}^\dag)$) divided by $g_{j-1,j-1}(U)$, while 
$v_{i1}=u_{i1}$, $v_{i2}=u_{i2}$, $i\in[1, n-1]$. 
Using once again Lemma~\ref{lem:upperalgebras} together with~\eqref{badg}, we see that every matrix entry of $\Theta$ is an element of $\UC(\G\CC_n)$ divided by a monomial that may contain only $\eg_{22},\ldots,\eg_{n-1,n-1}$ and frozen variables $x_{1n}, x_{n1}, \det X$. Consequently, the same is true for matrix entries of $X$. \qed

\begin{remark} 
One can show that the matrix $L$ featured in Proposition~\ref{prop:theta} is, in fact, symmetric with  entries given by
\[
L_{ij} = -\frac{1}{x_{1n}} \sum_{s=0}^{\min (i-1,n-j-1)} \det X_{1n}^{i-s,j+1+s}
\]
for $i\leq j$. However, this explicit formula is not needed in the proof of Theorem~\ref{thm:InitLaurent}, and we omit its derivation.
\end{remark}

\subsection{Mutation sequence $\Ws_n$: proof of Theorem~\ref{fish_to_cuttlefish}(i)} 
\subsubsection{Quiver $Q_n^0$ and universal notation} We start with defining matrices 
$\NK_{ij}(X)$ of size $(n-j+2)\times(n-j+2)$ for $1\le j<i\le n$, $(i,j)\ne(n,1)$. These are {\it maximal trailing submatrices\/} of matrices of type $\NG(I_1,I_2)$, that is, square submatrices that occupy all rows and the last 
$|I_1|+|I_2|$ columns (see Section~\ref{badgproof} immediately after the proof of Lemma~\ref{auxPlucker} for the definition of such matrices).  The subsets $I_1$ and $I_2$ that correspond to $\NK_{ij}$ are $I_1=[i,n]$ and $I_2=[n+j-i,n]$. Recall that $\ek_{ij}(X)=\det\NK_{ij}(X)$ for $1\le j<i\le n$, $(i,j)\ne(n,1)$, and $\ek_{ii}(X)=\det X_{[i,n]}^{[i,n]}$.

Since functions $\phi_{kl}(X)$, $k,l\ge1$, $k+l\le n-1$, and $\ek_{nn}(X)=f_1(X)$ are already present in the initial cluster $\FF_n$, we invoke \cite[Th.~3.6]{CL} that states that if two seeds are mutationally equivalent and share a set of common cluster variables, there exists a sequence of mutations that connects these
seeds and does not involve the common cluster variables. Consequently, we will have to mutate the subquiver $Q_n^0$ of the initial quiver $Q_n$ that does not contain vertices $(k,l)$ for $k+l<n-1$. The vertices $(k,n-k-1)$ for $1\le k\le n-2$ enter $Q_n^0$ as frozen, same as vertices $B$ and $C$; vertex $(1,-)$ is frozen as well, so that vertex $A$ is not contained in $Q_n^0$. For convenience of describing the mutation sequence, we introduce a universal numbering of the vertices of quivers $Q_n^0$, so that the same vertex gets the same number in all $Q_n^0$ for different $n$. Vertex $D$ gets number~0, mutable vertices in the part of $Q_n^0$ that corresponds to $Q_{n-1}^\dagger$ are numbered by positive integers from~1 to $n(n-3)/2$ layer by layer starting from the layer closest to $D$, bottom up in each layer. 
Frozen vertices in this part of $Q_n^0$ are numbered by positive integers from $n(n-3)/2+1$ to $n(n-1)/2$ as follows: first $B$, then $C$, and then the vertices $(k,n-k-1)$ bottom up. Mutable vertices in the part of $Q_n^0$ that corresponds to $Q_n^T$ are numbered by negative integers from $-1$ to $-2n+4$ 
layer by layer starting from the layer closest to $D$, bottom up in each layer. The remaining frozen vertex is numbered $-2n+3$. The quiver $Q_6^0$ with the universal numbering of its vertices is shown in Fig.~\ref{Q60}. 

\begin{figure}[ht]
\begin{center}
\includegraphics[width=12cm]{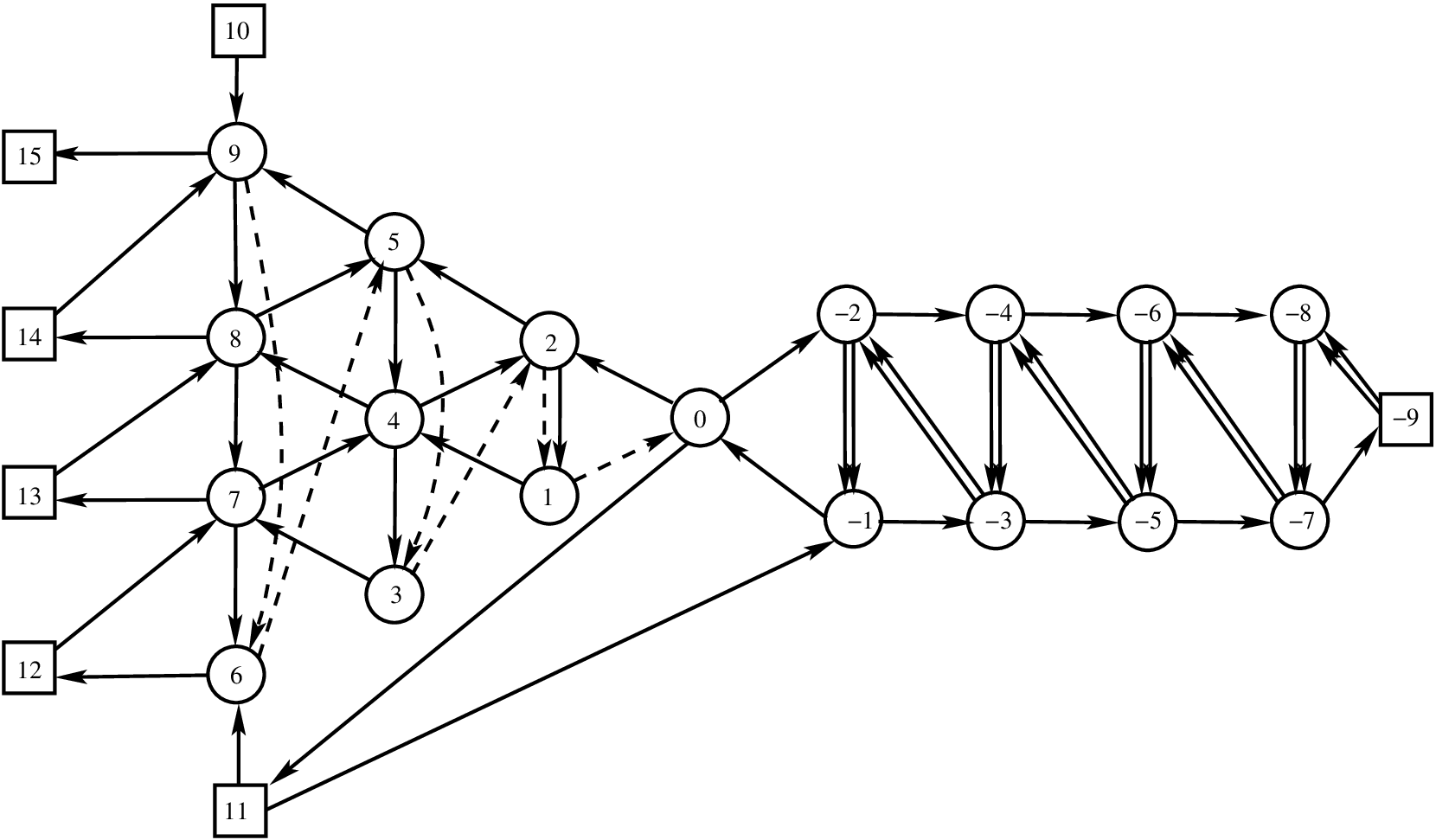}
\caption{Quiver $Q_6^0$}
\label{Q60}
\end{center}
\end{figure}

Our goal is to find a mutation sequence that takes cluster variables that correspond to mutable vertices of $Q_n^0$ to
the functions $\ek_{ij}(X)$ for $1\le j\le i\le n$, $(i,j)\ne (1,1), (n,1), (n,n)$.   
We will see that all cluster variables that arise along the mutation path (excluding $\ek_{ij}$ themselves) are determinants of maximal trailing submatrices of certain matrices $\NG(I_1,\dots,I_p)$ of the folowing two types: 

(i) $I_j$ for $j\in[1,p]$ consists of the first row and $|I_j|-1>0$ last rows, or 

(ii) $I_1$ consists of the $|I_1|>0$ last rows, and $I_j$ for $j\in[2,p]$ consists of the first row and $|I_j|-1>0$ last rows.

To facilitate notation, the functions of the first type are denoted $[k_1\, k_2\ldots k_p]$,  and the functions of the second type, $[\bar k_1\, k_2\ldots k_p]$, where $k_j=|I_j|$. Note that the symbol $[k_1\, k_2\ldots k_p]$ defines a function whenever $n\ge k_i$, $i\in[1,p]$, and $n+p-1\ge\sum_{i=1}^pk_i$.
If $r$ consecutive values of $k_j$ are equal to the same value $k$, we occasionally write $k^r$ instead of 
$k\ldots k$. The corresponding matrices are denoted $\NG(k_1,\ldots,k_p)$ and $\NG(\bar k_1,\ldots,k_p)$, respectively.
Adopting this notation, we write the cluster variable $\eg_{n-1,n-1}(X)$ at the vertex~0 of $Q_n^0$ as $[2^{n-1}]$. Further, cluster variables at the vertices of the $k$th layer to the left of~0 bottom up are $[2^{n-k-3}\ 2\ k+2]$, $[2^{n-k-3}\ 3\ k+1], \ldots, [2^{n-k-3}\ k+2\ 2]$. The cluster variable at $B$ is 
$[n]$, the cluster variables at the frozen layer bottom up are $[\bar 2\ n-1]$, $[\bar 3\ n-2],\ldots, [\overline{n-1}\ 2]$. The variable at vertex $-2k+1$ is $[\bar 1\ 2^{n-k-1}]$, the variable at vertex $-2k$ is $[2^{n-k-1}]$. 

We define the {\it promotion\/} 
of the functions described above as follows: the promotion of  $[k_1\ldots k_p]$ is $[2\, k_1\ldots k_p]$, 
the promotion of  $[\bar1\, k_2\ldots k_p]$ is $[\bar1\,2\, k_2\ldots k_p]$, 
the promotion of $[\bar k_1\, k_2\ldots k_p]$ for $k_1>1$ is $[\bar 1\, k_1+1\,k_2\ldots k_p]$.  
The {\it demotion\/} of a promoted function is defined as follows: the demotion of $[2\, k_1\ldots k_p]$ is
$[\bar1\, k_1\ldots k_p]$, and the demotion of $[\bar1\, k_1\ldots k_p]$ is  $[k_1\ldots k_p]$. Note that promotion and demotion are not inverse of each other; they commute if both are applicable.
Promotions and demotions of matrices $\NG$ are defined in a similar way.
Note that for any mutable vertex of $Q_{n-1}^0$ with the universal number $i$, the variable attached to $i$ in
$Q_n^0$ is the promotion of the variable  attached to $i$ in $Q_{n-1}^0$. We call this the {\it promotion property\/}
for  $Q_{n-1}^0$ and $Q_n^0$. Assume that $Q_{n-1}'$ and $Q_n'$ are obtained from $Q_{n-1}^0$ and $Q_n^0$, respectively, via the same sequence of mutations. We say that the promotion property is valid for $Q_{n-1}'$ and $Q_n'$ at $i$ if the variable attached to $i$ in $Q_n'$ is the promotion of the variable attached to $i$ in $Q_{n-1}'$. 

Using notation introduced above, exchange relation~\eqref{dexchange} at vertex $D$ reads
\begin{equation}\label{0exchange}
[2^{n-1}]\cdot[\bar 1\ 2^{n-4}\ 3]=
[\bar 1\ 2^{n-2}]\cdot [2^{n-3}\ 3]+[2^{n-2}]\cdot[\bar 1\ 2^{n-4}\ 3\ 2].
\end{equation}
 Note that the last term in the right hand side corresponds to the product of $x_{n1}$ and $\eg_{n-1,n-2}(X)$; in all mutations along the sequence %$\Ws_n$ 
we are looking for that involve vertex $C$, $x_{n1}$ enters exchange relations in a similar way.
Recall that relation~\eqref{0exchange} was obtained in Section~\ref{dexchangeproof} as the four-term Pl\"ucker relation for the matrix obtained by prepending rows $[0,\dots,0,1]$ and $[1, 0,\dots,0]$ to $\NG(2^{n-3},3,2)$ and choosing 
the indices as follows: $\alpha$ and $\beta$ correspond to the prepended rows, $\gamma$ to the first row of the first block row, $\delta$ to the second row of the penultimate block row, $\xi$ and $\zeta$ to the rows of the last block row. Slightly abusing terminology, we call it the four-term Pl\"ucker relation for $\NG(2^{n-3},3,2)$.
More generally, the four-term Pl\"ucker relation for $\NG(k_1,\dots,k_{p-1},2)$ is obtained by prepending rows 
$[0,\dots,0,1]$ and $[1, 0,\dots,0]$ to its last $\sum_{i=1}^p k_i-1$ columns and defining $\alpha$, $\beta$, $\gamma$, 
$\delta$, $\xi$, and $\zeta$ in the same way as above. Note that one of the terms in such relations vanishes identically, so, in fact, they have only three terms. In particular, the four-term Pl\"ucker relation for 
$\NG(k,2)$ with $k>2$ reads
\begin{equation}\label{4pluck}
[k-1\,2]\cdot[\overline{k-1}]=[\overline{k-2}\,2]\cdot[k]+[\overline{k-1}\,2]\cdot[k-1].
\end{equation}

The following statement is obtained straightforwardly from the definitions.

\begin{proposition}\label{promo1}
{\rm (i)} The four-term Pl\"ucker relation for the promotion of a matrix $\NG(k_1,\dots,k_{p-1},2)$ is obtained by promoting every function involved in the four-term Pl\"ucker relation for $\NG(k_1,\dots,k_{p-1},2)$.

{\rm(ii)} The four-term Pl\"ucker relation for the demotion of a promoted matrix is obtained by demoting every function involved in the four-term Pl\"ucker relation for the promoted matrix.
\end{proposition}

Consequently, the four-term Pl\"ucker relations for $\NG(2^q,k,2)$ and its demotion $\NG(\bar1, 2^{q-1},k,2)$ read
\begin{equation}\label{promo4pluck}
\begin{aligned}
&[2^q\,k-1\,2]\cdot[\bar1\, 2^{q-1}\,k]=[\bar1\,2^{q-1}\,k-1\,2]\cdot[2^q\,k]+[\bar1\,2^{q-1}\,k\,2]\cdot[2^q\,k-1],\\
&[\bar1\,2^{q-1}\,k-1\,2]\cdot[2^{q-1}\,k]=[2^{q-1}\,k-1\,2]\cdot[\bar1\,2^{q-1}\,k]+[2^{q-1}\,k\,2]\cdot[\bar1\,2^{q-1}\,k-1].
\end{aligned}
\end{equation}

We will also use the Desnanot--Jacobi identity for an $(r+1)\times r$ matrix $A$ and three rows $\alpha$, $\beta$, and $\gamma$ (see \cite[Eq.~(6.2)]{GSVMem}) that reads
\begin{equation}\label{genDJ}
\widehat{\beta}\cdot\widehat{\alpha,\gamma}=\widehat{\gamma}\cdot\widehat{\alpha,\beta}+\widehat{\alpha}\cdot
\widehat{\beta,\gamma},
\end{equation}
where $\widehat{i}$ denotes the $r\times r$ minor of $A$ obtained by deleting row $i$, and $\widehat{i,j}$ denotes
the $(r-1)\times(r-1)$ minor of $A$ obtained by deleting rows $i$ and $j$ and column~1. The Desnanot--Jacobi identity
for $\NG(k_1,\dots,k_{p})$ with $k_p>2$ is obtained from the above relation when $A$ is the submatrix in the last 
$\sum_{i=1}^p k_i-1$ columns of $\NG(k_1,\dots,k_{p})$,
$\alpha$ is the first row of the first block row, $\beta$ is the second row of the penultimate block row, and $\gamma$ is the second row of the last block row of $\NG(k_1,\dots,k_{p})$. In particular, the Desnanot--Jacobi identity for $\NG(k_1,k_2)$ with $k_1,k_2>2$ reads
\begin{equation}\label{DJ}
[k_1-1\,k_2]\cdot[\overline{k_1-1}\,k_2-1]=[k_1\,k_2-1]\cdot[\overline{k_1-2}\,k_2]+[\overline{k_1-1}\,k_2]
\cdot[k_1-1\,k_2-1].
\end{equation}

The following analog of Proposition~\ref{promo1} is obtained straightforwardly from the definitions.
 
\begin{proposition}\label{promo2}
{\rm (i)} The Desnanot--Jacobi identity for the promotion of a matrix $\NG(k_1,\dots,k_{p})$ is obtained by promoting every function involved in the Desnanot--Jacobi identity for $\NG(k_1,\dots,k_p)$.

{\rm (ii)} The Desnanot--Jacobi identity for the demotion of a promoted matrix is obtained by demoting every function involved in the Desnanot--Jacobi identity for the promoted matrix.
\end{proposition}

Consequently, the Desnanot--Jacobi identities for $\NG(2^q,k_1,k_2)$ with $k_1,k_2>2$ and its demotion
$\NG(\bar1, 2^{q-1},k_1,k_2)$ read
\begin{equation}\label{promoDJ}
\begin{aligned}
[2^q\,k_1-1\,k_2]\cdot[\bar1\, 2^{q-1}\,k_1\,k_2-1]&=[2^q\,k_1\,k_2-1]\cdot[\bar1\,2^{q-1}\,k_1-1\,k_2]\\
&+[\bar1\,2^{q-1}\,k_1\,k_2]\cdot[2^q\,k_1-1\,k_2-1],\\
[\bar1\,2^{q-1}\,k_1-1\,k_2]\cdot[2^{q-1}\,k_1\,k_2-1]&=[\bar1\,2^{q-1}\,k_1\,k_2-1]\cdot[2^{q-1}\,k_1-1\,k_2]\\
&+[2^{q-1}\,k_1\,k_2]\cdot[\bar1\,2^{q-1}\,k_1-1\,k_2-1].
\end{aligned}
\end{equation}

The sequence $\Ws_n$ is defined as the concatenation of two sequences: the {\it head\/} $\Hs_n$ and the {\it tail\/} 
$\Ts_n$, which will be treated separately.

\subsubsection{Sequence $\Hs_n$} The sequence $\Hs_n$, in turn, is the concatenation of four subsequences: the {\it prefix\/} $\Ps_n$, the {\it infix\/} $\Is_n$, the {\it root\/} $\Rs_n$, and the {\it suffix\/} $\S_n$. The root $\Rs_n$ consists of the vertices in the $(n-3)$rd layer to the left of~0 bottom up. It is the leftmost mutable layer in
$Q_n^0$, so that $\Rs_4=1\, 2$, $\Rs_5=3\, 4\, 5$, etc.
The other three subsequences are defined recursively via
\[
\Ps_n=\Hs_{n-1},\qquad \Is_n=\Rs_{n-1}^-\S_{n-1}^-,\qquad \S_n=\Is_n^-,
\] 
where $\mathbf A^-$ means subtracting~1 from every index in $\mathbf A$. The recursion starts at $n=4$ with $\Ps_4=0$, 
$\Is_4=-1$, and $\S_4=-2$, so that $\Hs_4=0\, {-}1\, 1\, 2\, {-}2$. Further, $\Ps_5=0\, {-}1\, 1\, 2\, {-}2$, $\Is_5=0\, 1\, {-}3$, $\S_5={-}1 \, 0\, {-}4$, so that $\Hs_5=0\, {-}1\, 1\, 2\, {-}2\, 0\, 1\, {-}3\, 3\, 4\, 5\, {-}1 \, 0\, {-}4$, etc.

Consider the quiver $Q_n^n$ that has the same set of vertices as $Q_n^0$ (both frozen and mutable) arranged in a different way. Vertices of $Q_n^n$ are arranged into $n-1$ boomerang shapes of sizes $n+1, n,\ldots,3$. Each boomerang shape contains a subset of vertices with consecutive numbers in descending order. The shapes are placed on the 
$(n+1)\times(n+1)$ greed as shown in Fig.~\ref{boomer} for the case $n=6$. The vertex placed at the point with coordinates $i$ and $j$ is denoted $|i,j|$.

\begin{figure}[ht]
\begin{center}
\includegraphics[width=8cm]{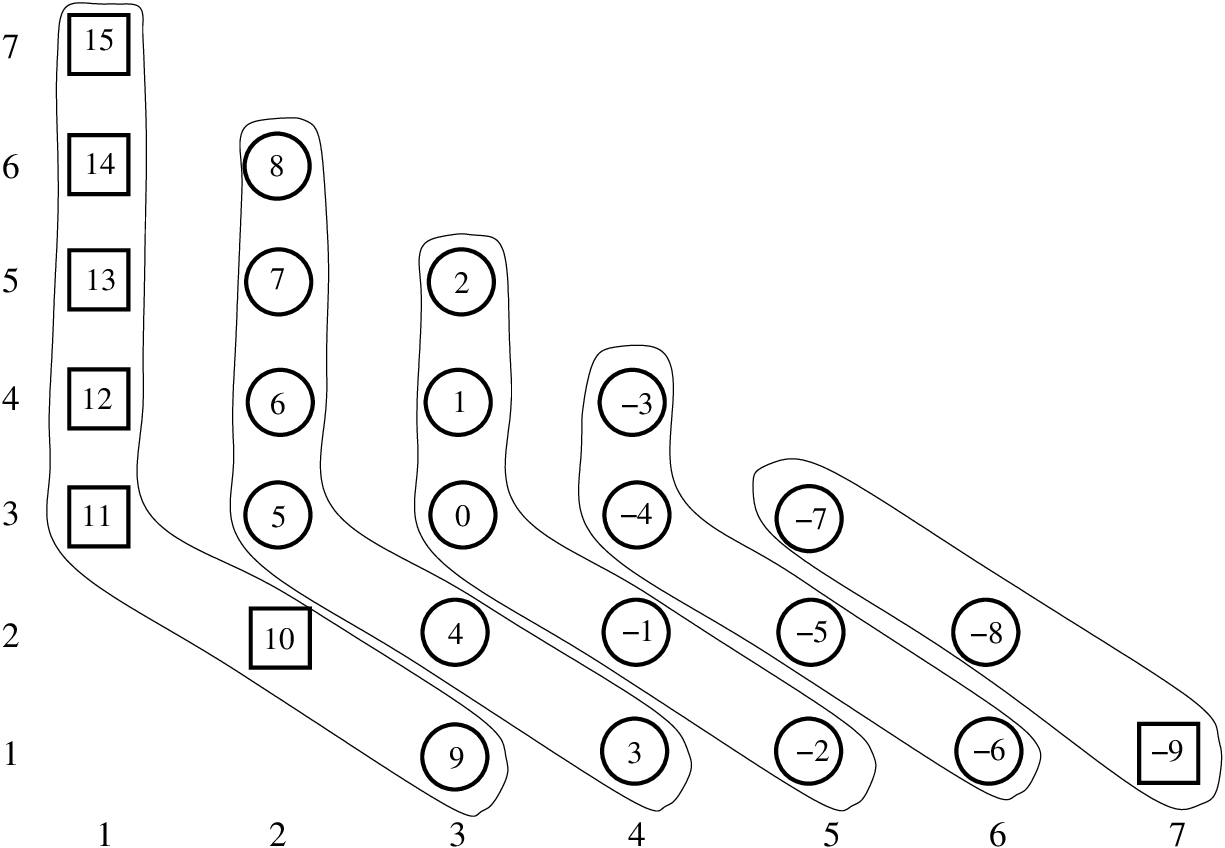}
\caption{Vertex arrangement for the quiver $Q_6^6$}
\label{boomer}
\end{center}
\end{figure}

The arrows of $Q_n^n$ are $|i,j|\to|i-1,j+1|$ for $i\in [2,n+1]$, $j\in[1,n-i+2]$, $|i,j|\ne|2,2|$, $|i,j|\to|i,j-1|$ for $i\in[2,n]$, $j\in[2,n-i+2]$, $|i,j|\ne|2,3|$, $|i,j|\to |i+1,j|$ for $i\in[1,n-1]$, $j\in [2,n-i+1]$. Additionally, there is an oriented 
path $|1,n+1|\to|3,1|\to|2,n|\to\cdots\to |n-k-1,k+3|\to |n-k+1,1|\to |n-k,k+2|\to\cdots\to|n-1,3|\to|n+1,1|\to|1,n+1|$.
The quiver $Q_6^6$ is shown in Fig.~\ref{Q66}; the arrows in the additional path are shown by dashed lines.

\begin{figure}[ht]
\begin{center}
\includegraphics[width=8cm]{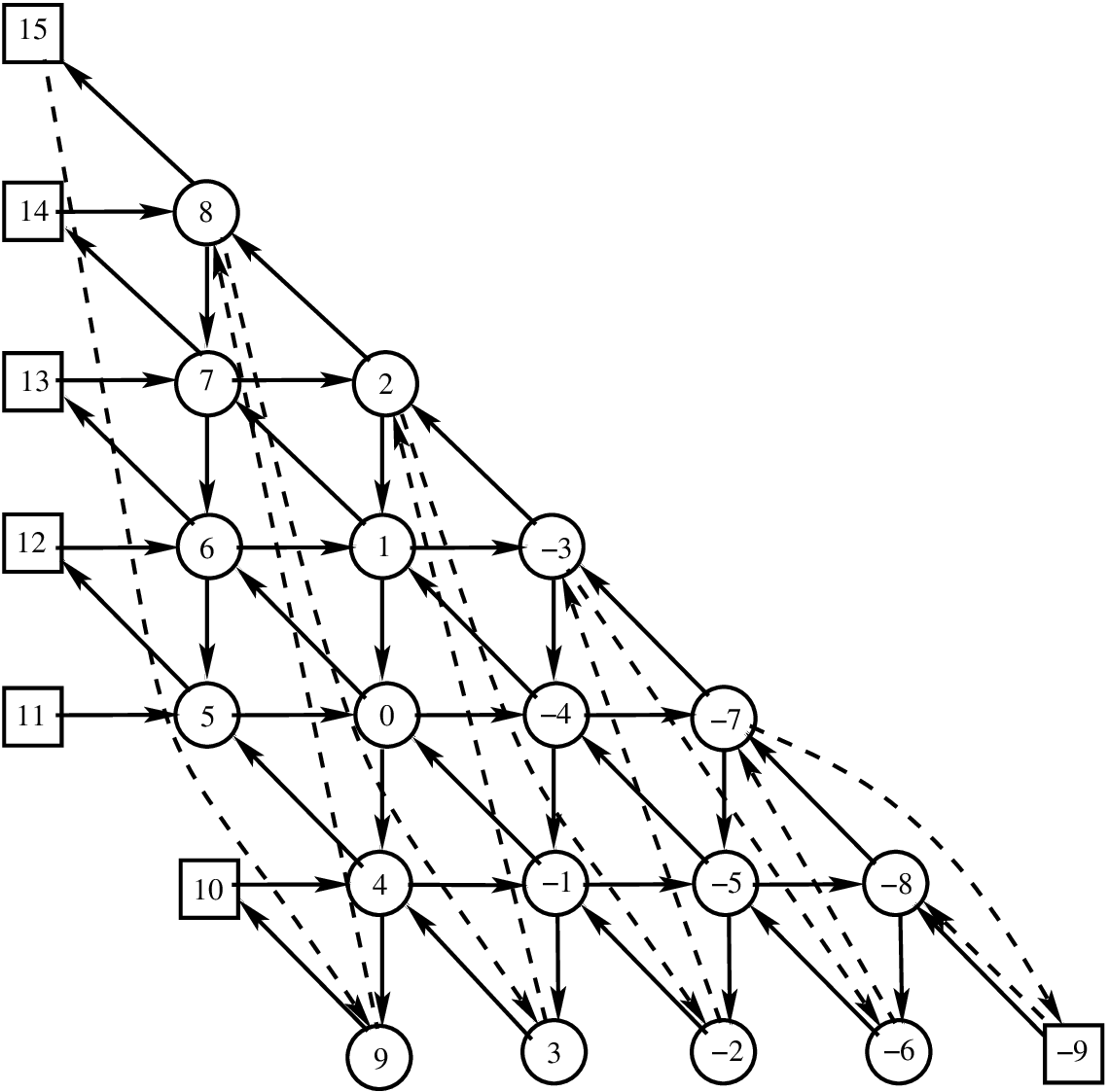}
\caption{Quiver $Q_6^6$}
\label{Q66}
\end{center}
\end{figure}

\begin{lemma} Mutation sequence $\Hs_n$ applied to the quiver $Q_n^0$ results in $Q_n^n$. The variable attached to vertex $|i,j|$ in $Q_n^n$ is $[\overline{j-2}\ n-i-j+4]$ for $j\ge 3$ except for vertex $C=|1,3|$, $[n-i+2]$ for 
$j=2$, and $[\overline{n-i+2}]$ for $j=1$. %the function attached to $C$ is $x_{n1}$.
\end{lemma}

\begin{proof} The proof is by induction on $n$. For $n=4$, the exchange relations along the mutation path $\Hs_4$ are
\begin{equation}\label{exchange4}
\begin{aligned}
0&:\quad [2^3]\cdot[\bar 1\,3]=[\bar 1\,2^2]\cdot[2\,3]+[2^2]\cdot[\bar 1\,3\,2], \\
-1&: \quad [\bar 1\,2^2]\cdot[3]=[\bar 1\,3]\cdot[2^2]+[3\,2]\cdot[\bar 1\,2],\\
1&: \quad [2\,3]\cdot[\bar 2\,2]=[\bar 1\,3]\cdot[3\,2]+[2^2]\cdot[\bar 2\,3], \\
2&: \quad [3\,2]\cdot[\bar 3]=[\bar 2\,2]\cdot[4]+[3]\cdot[\bar 3\,2],\\
-2&: \quad [2^2]\cdot[\bar 2]=[3]\cdot[\bar 1\,2]+[\bar 2\,2]\cdot[2].
\end{aligned}
\end{equation}
A direct check shows that the assertions of the lemma hold true. Additionaly, one can check that for $n>4$ the above mutation sequence adds the arrow from $3$ to $C$ and deletes the arrow from $5$ to $3$.

Note that the first two exchange relations above are 
particular cases of~\eqref{promo4pluck} for $k=3$ and $q=2$,   
the third relation is~\eqref{DJ} for $k_1=k_2=3$, while the last two are~\eqref{4pluck} for $k=4$ and $k=3$, respectively. We thus see that to handle the case $n=4$ we have to consider matrices $\NG(3,2)$ (more exactly, its promotion $\NG(2,3,2)$), $\NG(3,3)$, and $\NG(4,2)$.

Assume that the assertion of the lemma holds true for $\Hs_{n-1}$, and consider the evolution of $Q_n^0$ and the associated cluster variables along the mutation sequence $\Hs_n$ divided into seven consecutive segments: $\Hs_n=\Ps_{n-1}\Is_{n-1}\Rs_{n-1}\S_{n-1}\Is_n\Rs_n\S_n$. Recall that the promotion property for $Q_{n-1}^0$ and $Q_n^0$ is valid at every mutable vertex of $Q_{n-1}^0$. Moreover, mutations in the first two segments of $\Hs_n$ are identical to mutations in the first two segments of $\Hs_{n-1}=\Ps_{n-1}\Is_{n-1}\Rs_{n-1}\S_{n-1}$, that is, involve vertices with the same universal numbers connected with the same arrows. Therefore, by Propositions~\ref{promo1} and~\ref{promo2},
the promotion property remains valid for all mutated vertices along the sequence.

Further, compare mutations in the third segment of $\Hs_{n-1}$ and the third segment of $\Hs_{n}$. Recall that these
are mutations in the first from the left mutable layer of $Q_{n-1}^0$ and the second from the left mutable layer of $Q_n^0$. The bottom vertex of this layer has universal number $i=(n-3)(n-4)/2$.
The arrows inside this layer and to the right of it are identical in both mutated quivers, as explained above. 
An additional inductive assumption for the quiver obtained from $Q_n^0$ claims the following: 

(i) the arrow from the top vertex to $i$ is deleted; 

(ii) all other arrows inside this layer are the same as in $Q_{n}^0$;

(iii) the arrow from $i$ to $i-1$ is reversed; 

(iv) there are arrows  from $i$ to $C$ and from $C$ to $i-1$. 

\noindent The arrows on the left of this layer are nor involved in the previous mutations, so they remain exactly as in  $Q_{n-1}^0$ and $Q_n^0$, respectively. This assumption can be verified by direct observation for $n=5$.

We start the third segment with the mutation at $i$. The left neighborhood of $i$ in the quiver obtained from $Q_{n-1}^0$ before and after the mutation is shown in Fig.~\ref{ileft}a). The variable attached to the frozen vertex 
$j=(n-2)(n-3)/2+2$ is $[\bar2\,n-2]$. A similar neighborhood in the quiver obtained from $Q_{n}^0$ is shown in 
Fig.~\ref{ileft}b). The variable attached to the mutable vertex $j-1$ is $[3\,n-2]$. Additionally, in the second case, there is $x_{1n}$ attached to vertex $C$. Note that the product of variables attached to $j-1$ and $C$ is $[\bar1\,3\,n-2]$, that is, the promotion of  $[\bar2\,n-2]$. Consequently, Proposition~\ref{promo2} guarantees that the promotion property at $i$ remains valid after this mutation. Moreover, the situation before the mutation at $i+1$ is exactly the same as before, up to shifting all numbers up by one. Consequently, the promotion property at all mutable vertices of 
$Q_{n-1}^0$ remains valid after the third segment is applied. Note that in the last mutation in this segment, the vertices that correspond to $i+1$ and $j-1$ in Fig.~\ref{ileft}b) are the bottom and the top vertices of the next layer, respectively, so that deletion of the arrow from $j-1$ to $i+1$ proves inductive assumption (i), reversal 
of the arrows from $i+1$ to $i$ and from $i$ to $C$ and addition of the arrow from $i+1$ to $C$ proves inductive assumptions (iii) and (iv). All other mutations of this segment do not affect arrows in the leftmost mutable layer of 
$Q_n^0$, which proves inductive assumption (ii). 

\begin{figure}[ht]
\begin{center}
\includegraphics[width=8cm]{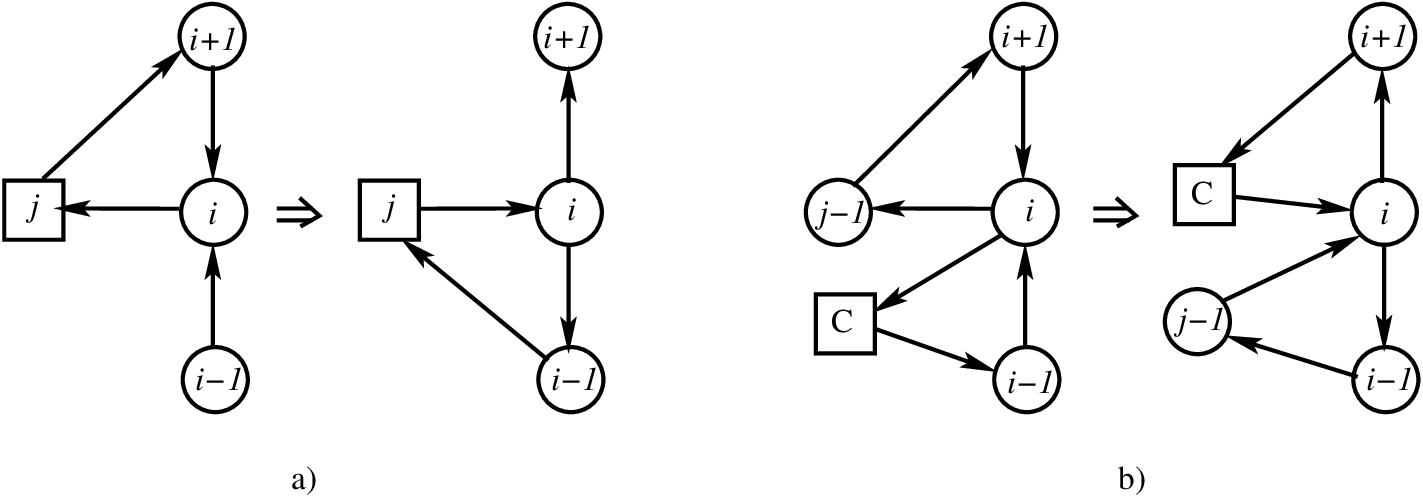}
\caption{Left neighborhood of $i$ before and after the first mutation in $\Rs_{n-1}$: a) for $\Hs_{n-1}$; 
b) for $\Hs_n$}
\label{ileft}
\end{center}
\end{figure}

Finally, mutations in the fourth segment in $\Hs_{n-1}$ and $\Hs_n$ again involve only vertices with the same numbers connected with the same arrows, and hence by the end of the first four segments in $\Hs_n$ the promotion property remains valid for all mutable vertices of $Q_{n-1}^0$. Taking into account the induction assumption concerning the 
structure of $Q_{n-1}^{n-1}$, we can describe the quiver $Q_n^{n-1}$ obtained from $Q_n^0$ via the mutation sequence 
$\Hs_{n-1}$ and the variables attached to its vertices, as compared to $Q_{n-1}^{n-1}$. 

Quiver $Q_n^{n-1}$ is placed on the grid, same as $Q_{n-1}^{n-1}$. The restriction of $Q_n^{n-1}$ to positions $|i,j|$, 
$i\in[3,n+1]$, $j\in[1,n-i+2]$ coincides with restriction of $Q_{n-1}^{n-1}$ to positions $|i,j|$, $i\in[2,n]$, $j\in[1,n-i+1]$, so that the universal number for the vertex at position $|i,j|$ in $Q_n^{n-1}$ is the same as for $|i-1,j|$ in $Q_{n-1}^{n-1}$. Consequently, by the promotion property, the variable attached to position $|i,j|$ for $i\in[3,n+1]$, 
$j\in[1,n-i+2]$ is $[\bar1\, j-1\, n-i-j+5]$ for $j\ge 3$, $[2\, n-i+3]$ for $j=2$, and $[\bar1\, n-i+2]$ 
for $j=1$. The only difference between the two restrictions is that  in $Q_{n-1}^{n-1}$ vertices at positions $|2,2|$ and $|n,1|$ are frozen, while in $Q_n^{n-1}$ the corresponding vertices at $|3,2|$ and $|n+1,1|$ are mutable. Frozen vertices from the leftmost layer of $Q_n^0$ are placed at positions $|1,j|$ for $j\in[3,n]$ in the same vertical order. Vertices of the next layer are placed at positions $|2,j|$ for $j\in[2,n]$ in the same vertical order, except for the bottom mutable vertex $(n-2)(n-3)/2$ that has been already placed at $|3,2|$. The arrows between the vertices in these two layers are not affected by the mutations in $\Hs_{n-1}$, except for the cases described in induction assumptions 
(i)--(iv) that are already proved for this layer. Evolution of arrows between this layer and the layer to the right of it can be seen in Fig.~\ref{ileft}b). Finally, vertices $-2n+4$ and $-2n+5$ are placed at positions $|n+1,3|$ and $|n+2,2|$, respectively. Arrows between them and the rest of the vertices are not affected by mutations in $\Hs_{n-1}$. The quiver $Q_6^5$ is shown in Fig.~\ref{Q65}a). The part of $Q_6^5$ that coincides with the described above part of $Q_5^5$ is contained inside the triangle shown by a thin line.

\begin{figure}[ht]
\begin{center}
\includegraphics[width=12cm]{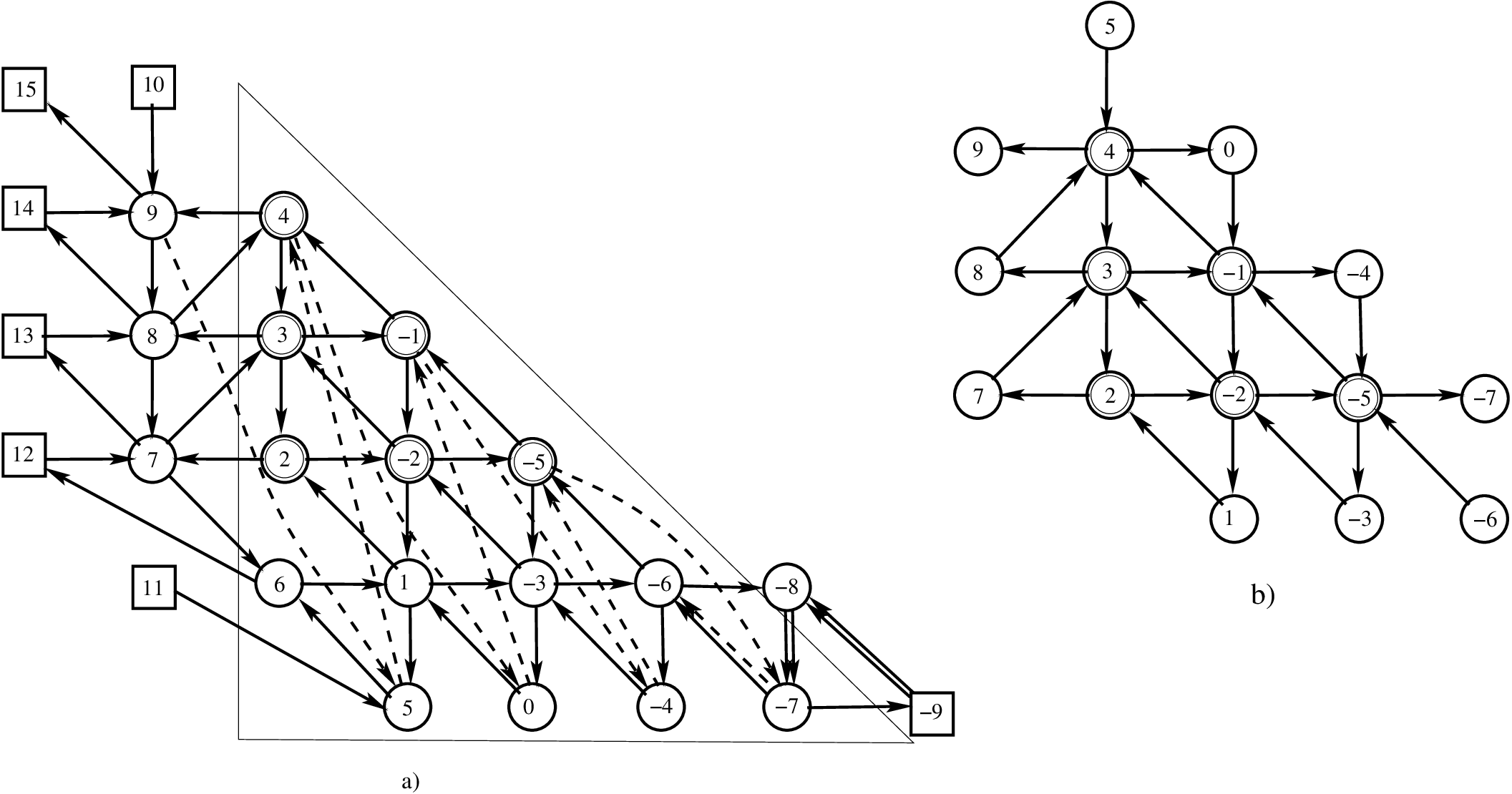}
\caption{Quiver $Q_6^5$: a) the quiver; b) the part of $Q_6^5$ relevant for the sequence $\Is_6$}
\label{Q65}
\end{center}
\end{figure}

By the promotion property, the variables inside this triangle are promotions of the 
corresponding variables in $Q_5^5$, so that by the inductive assumption, the variable attached to $|i,j|$ is
$[\bar1\, j-1\,n-i-j+3]$. The variables outside the triangle retain their initial values. It is easy to see 
that vertices that are mutated in the sequence $\Is_n$ fill the triangle $i\in[3,n-1]$, $j\in[3,n-i+2]$ and that order is by columns from left to right, bottom up within each column. The corresponding vertices for the quiver $Q_6^5$ are shown by double circles in Fig.~\ref{Q65}.  Consequently, the relevant vertices outside the triangle 
are $|2,j|$ with $j\in[3,n-1]$, and the attached variables are $[j\,n-j+1]$. The part of $Q_6^5$ relevant for the sequence $\Is_6$ is shown in Fig.~\ref{Q65}b).

Let us rearrange the vertices of $Q_n^{n-1}$ 
after the sequence $\Is_n$ is applied in the following way: the vertices in rows $j=1,2$ are shifted left by one, and the vertices $|i,j|$ on the upper diagonal $i+j=n+2$ are shifted down by $j$ and right by one, so that they now form the lower row of the quiver. The result of this rearrangement for the quiver $Q_6^5$ is shown in Fig.~\ref{Q65I}.

\begin{figure}[ht]
\begin{center}
\includegraphics[width=8cm]{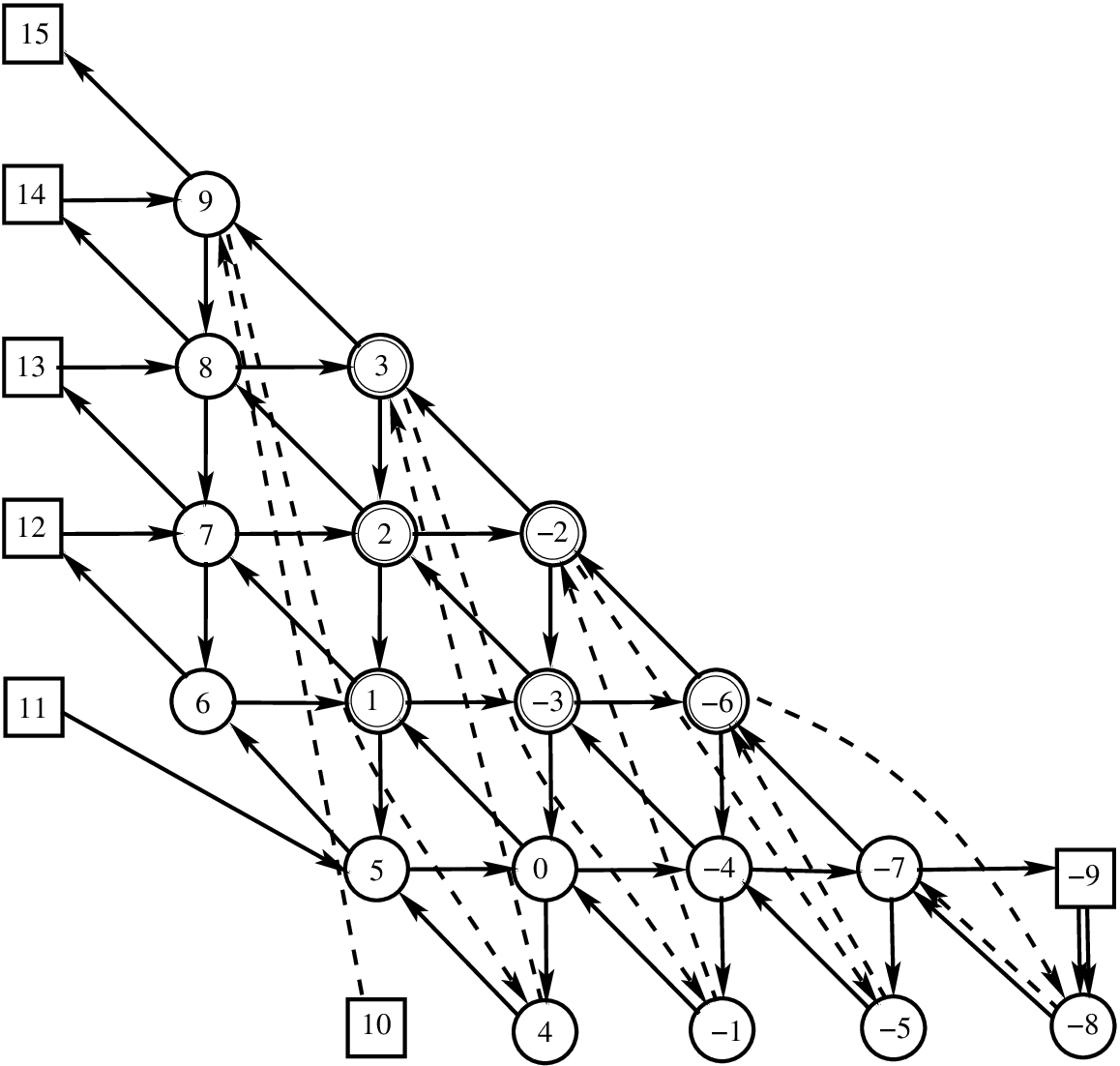}
\caption{Quiver $Q_6^5$ after mutation sequence $\Is_6$ and rearrangement}
\label{Q65I}
\end{center}
\end{figure}
 
Comparing quivers $Q_n^{n-1}$ (before the application of $\Is_n$) and $Q_{n-1}^{n-2}$ after the application of $\Is_{n-1}$ we observe that the subquiver of the latter quiver that is relevant for the sequence $\Rs_{n-1}\S_{n-1}$ is isomorphic to 
the subquiver of the former one that is relevant for the sequence  $\Is_n$. Moreover, the variable attached in the latter 
quiver to the vertex that correspond to the vertex $|i,j|$ of $Q_n^{n-1}$ is $[j-1\, n-i-j+3]$. It is easy to see that
the variable $[\bar1\, j-1\,n-i-j+3]$ attached to $|i,j|$ in $Q_n^{n-1}$ is obtained from $[j-1\, n-i-j+3]$ by consecutive promotion and demotion, and the same holds true for all other vertices in the two isomorphic subquivers. 
So, by the promotion and demotion properties, we conclude that all exchange relations along the sequence $\Is_n$ are
obtained by consecutive promotion and demotion of the exchange relations along the sequence $\Rs_{n-1}\S_{n-1}$. Therefore, the variable attached to the vertex $|i,j|$ of the rearranged quiver is $[j-1\,n-i-j+4]$ for $i\in[2,n-1]$, 
$j\in [3,n-i+2]$, $[\bar1\, n-i+2]$ for $j=2$, $i\in[3,n]$,  $[n-i+3]$ for $i\in[3,n+1]$, and 
$[\overline{j-2}\, n-j+3]$ for $i=1$, $j\in[4,n+1]$.

Further, the sequence $\Rs_n$ is applied. It is easy to see that the vertices that participate in this sequence are the vertices in the second column of the current quiver bottom up (vertices $6, 7, 8, 9$ in Fig.~\ref{Q65I} for $n=6$). The corresponding exchange relations are~\eqref{DJ} for
$k_1=j$ and $k_2=n-j+2$ with $j\in[3,n-1]$, and~\eqref{4pluck} for $k=n$. The resulting variables are $[\overline{j-1}\,
n-j+1]$ for $j\in[3,n-1]$ and $[\overline{n-1}]$ for $j=n$.

Finally, the sequence $\S_n$ is applied. It is easy to see that the vertices that participate in this sequence fill the triangle $i\in[3,n-1]$, $j\in[3,n-i+2]$ (it is shown by double circles in Fig.~\ref{Q65I}). The order of mutations is 
by columns from left to right, bottom up in each column. Direct inspection shows that the relevant part of the quiver is exactly the same as for sequences $\Rs_{n-1}\S_{n-1}$ and $\Is_n$, see Fig.~\ref{Q65}b) for the case $n=6$. Moreover, 
the functions attached to the vertices involved in the mutations are exactly the same as in the case of $\Rs_{n-1}\S_{n-1}$, and hence the variables attached to the vertices after the mutation coincide with those obtained at the end of $\Hs_{n-1}$. It remains to rearrange the resulting quiver similarly to what was done before: the vertices in the rows $j=1,2$ 
are shifted left by one, the vertices of the first column are shifted down by one,  and the vertices $|i,j|$ on the upper diagonal $i+j=n+2$ are shifted down by $j$ and right by one, so that they now form the lower row of the quiver. The result of this rearrangement coincides with the quiver $Q_n^n$.
\end{proof}

\subsubsection{Sequence $\Ts_n$} Note that the variables attached to the first row of the quiver $Q_n^n$ (see 
Fig.~\ref{Q66} for $n=6$) are exactly 
$\ek_{ii}(X)$, so we invoke once again~\cite[Th.~3.6]{CL} and freeze them. 
The remaining mutable vertices are mutated by columns from right to left (in a contrast with mutations in $\Hs_n$). As we shall see, mutations in the same column commute. 

The exchange relations
are Desnanot--Jacobi identities for matrices $\NG(\bar k_1,k_2)$ with $k_2>2$ that are obtained from~\eqref{genDJ} 
when $A$ is the submatrix in the last $k_1+k_2-1$ columns of $\NG(\bar k_1, k_{2})$,
$\alpha$ is the first row of the first block row, $\beta$ is the first row of the second block row, and $\gamma$ is the second row of the second block row of $\NG(\bar k_1,k_{2})$. Consequently, the corresponding Desnanot--Jacobi identity for $k_1>1$ reads
\begin{equation}\label{barDJ}
[\overline{k_1-1}\,k_2-1]\cdot[\bar k_1\,\overline{k_2-1}]=
[\bar k_1\,k_2-1]\cdot[\overline{k_1-1}\,\overline{k_2-1}]+[\overline{k_1-1}\,k_2]
\cdot[\bar k_1\,\overline{k_2-2}], 
\end{equation}
and for $k_1=1$,
\begin{equation}\label{bar1DJ}
[k_2-1]\cdot[\bar 1\,\overline{k_2-1}]=[\bar 1\,k_2-1]\cdot[\overline{k_2-1}]+[k_2]\cdot[\bar1\,\overline{k_2-2}]. 
\end{equation}

The first mutation in $\Ts_n$ is in column $n$ that contains the only mutable vertex $|n,2|$. The exchange relation at $|n,2|$ is~\eqref{bar1DJ} for $k_2=3$, which produces $\ek_{n,n-1}(X)$. As a result of the mutation at $|n,2|$, the arrow between the vertices in column $n-1$ disappears, so mutations at $|n-1,2|$ and $|n-1,3|$ commute, as mentioned above. 
The corresponding exchange relations are~\eqref{bar1DJ} for $k_2=4$ at $|n-1,2|$, which produces $\ek_{n,n-2}(X)$, 
and~\eqref{barDJ} for $k_1=2$, $k_2=3$ at $|n-1,3|$, which produces $\ek_{n-1,n-2}(X)$. As a result of these two mutations the arrows between the vertices in column $n-2$ disappears, so mutations at the vertices in this column commute. 
Following this pattern, we establish that the exchange relation at the vertex $|i,j|$, $j\in[2,n-i+2]$, of column $i$ 
is~\eqref{bar1DJ} for $k_2=n-i+4$ if $j=2$, which produces $\ek_{n,i-1}$, and~\eqref{barDJ} with $k_1=j-1$, $k_2=n-i-j+4$ if $j>2$, which produces $\ek_{n-j+2,i-1}(X)$. Note that for $i=2$, $j=3$, that is, for $k_1=2$, $k_2=n$, 
the function $[\bar1\, n]$ in the right hand side of~\eqref{barDJ} is the product $x_{n1}\det X$.

This concludes the proof of Theorem~\ref{fish_to_cuttlefish}(i). \qed

\subsection{Proof of Theorem~\ref{fish_to_cuttlefish}(ii)}
Recall that mutable and frozen variables in the cluster built in the previous section are 
 \begin{equation}\label{cuttlemut}
\{\phi_{kl}(X): k,l \geq 1, \ k+l \leq n-1 \}\cup \{\ek_{ij}(X) : 1\le j \leq i \leq n, (i,j)\ne (1,1), (n,1) \}.
\end{equation}
and
\begin{equation}\label{cuttlefroz}
\{x_{1n}, x_{n,1}, \ek_{11}(X)=\det X\}\cup \{\ec_r/\det X: 1\le r\le n-1\}, 
\end{equation}
respectively. 
Slightly abusing the language, we will say that a funcion of $X$ is {\it Laurent\/} if it admits a Laurent polynomial expression in terms of the variables in~\eqref{cuttlemut},~\eqref{cuttlefroz}. Further, a Laurent function is {\it good Laurent\/} if this Laurent polynomial expression does not contain frozen variables in the denominator. A matrix is said 
to be Laurent (or good Laurent) if every its entry possesses this property. In this terms, our goal is to prove that
$X$ is good Laurent.

Consider the Gauss factorization $X=X_+ X_{0,-}$ of $X$. We first treat the second factor.

\begin{lemma}
\label{lem:X<Laurent}
$X_{0,-}$ is good Laurent depending only on $\ek_{ij}(X)$ and $x_{n1}$.
\end{lemma}

\begin{proof}
Clearly 
\[
(X_{0,-})_{ii} =(X_{0})_{ii} = \frac{\det X_{[i,n]}^{[i,n]}}{\det X_{[i+1,n]}^{[i+1,n]}} =\frac{\ek_{ii}(X)}{\ek_{i+1,i+1}(X)},\quad 1\le i\le n.
%(X_{0,-})_{11} = \frac{\det X}{\ek_{22}(X)}.
\]
For $X,Y\in \Mat_n$ and $i>j$ define
\begin{equation}
\label{hatekfunctions}
\hat\ek_{ij}(X,Y)= \det
\begin{bmatrix*}[l]
Y_{[i,n]}\ 0& \\
\ 0\ X_{[n+j-i,n]}\hspace{-1em}
\end{bmatrix*}^{[j,n+1]}, 
%1< j < i \leq n \ ,
%\nonumber
%&\ek_{ii}(X)= \det X_{[i,n]}^{[i,n]}, i=2,\ldots, n\ .
\end{equation}
so that $\ek_{ij}(X)=\hat\ek_{ij}(X,X)$. Then for any $\beta_1,\beta_2\in\N_+$ and $\nu\in\N_-$,
\[
\hat\ek_{ij}(X,Y)= \hat\ek_{ij}(\beta_1 X\nu,\beta_2Y\gamma(\nu)), %\ (1< j < i \leq n)
\]
and, in particular,
\begin{align*}
\ek_{ij}(X)=\hat\ek_{ij}(X,X)&= \hat\ek_{ij}(X_{0,-},X_{0,-}) =\hat\ek_{ij}\left(X_{0},X_{0} X_-\gamma(X_-^{-1})\right)\\
&=\det \left(X_-\gamma(X_-^{-1})\right)_{[i,n]}^{[j,n+j-i]} \det X_{[i,n]}^{[i,n]}\det X_{[n+j-i,n]}^{[n+j-i,n]}.
\end{align*}
Denote $N=X_-\gamma(X_-^{-1})$. We see that
\begin{equation}
\label{eq:Nminors}
\det N_{[i,n]}^{[j,n+j-i]}=\frac{\ek_{ij}(X)}{\ek_{ii}(X)\ek_{n+j-i,n+j-i}(X)}, \quad 1\leq j < i \leq n,\ (i,j)\neq (n,1).
\end{equation}
Viewed as functions of $N$, minors above together with $N_{n1}=x_{n1}$ form an initial cluster for the standard cluster structure on $n\times n$ lower unipotent matrices.
It is well-known (for details see~\cite[Remark~3.4]{GSVdouble}, which, in turn, relies on~\cite{BFZ}) that  this cluster structure is regular and complete and thus all matrix entries $N_{ij}$ have Laurent polynomial expressions in terms of the initial cluster. Being a frozen variable, $N_{n1}=x_{n1}$ appears only in the numerators of these expressions. Since 
\[
X_{0,-}=X_0 X_- = X_0N\gamma(N)\gamma^2(N)\dots
\] 
and
\begin{equation}\label{X0viak}
(X_0)_{ii}=\frac{\ek_{ii}(X)}{\ek_{i+1,i+1}(X)},\quad i\in[1,n-1],\qquad (X_0)_{nn}=\ek_{nn}(X),
\end{equation}
the proof is complete.
\end{proof}

We start the treatment of the first factor in the Gauss decomposition of $X$ from the following claim.

\begin{lemma}
\label{lem:X+1Laurent}
Matrix entries of the first row of $X_+$ are  good Laurent.
\end{lemma}

\begin{proof} Since $(X_+)_{11}=1$ and $(X_+)_{1n}={x_{1n}}/{x_{nn}}={x_{1n}}/{\ek_{nn}(X)}$, we only need to consider 
$(X_+)_{1j}$ for $j\in[2, n-1]$. Recall the determinantal description for $\phi_{k,n-1-k}$ given 
at the beginning of Section~\ref{explicitphi}:
\begin{align*}
&\phi_{k,n-1-k}(X) =\det 
\begin{bmatrix*}[l]
X_{[n-k,n]}\ 0& \\
\ 0\ X_{1\cup [k+2,n]}\hspace{-1em}
\end{bmatrix*}
= \det 
\begin{bmatrix*}[l]
(X_0 N)_{[n-k,n]}\ 0& \\
\ 0\ (X_+X_0)_{1\cup [k+2,n]} \hspace{-1em}
\end{bmatrix*}\\
&\ = {\det X_{[n-k,n]}^{[n-k,n]} }{\det X_{[k+2,n]}^{[k+2,n]}} 
\det\begin{bmatrix*}[l]
N_{[n-k,n]}^{[1,k+2]} \\
0\ (X_+X_0)_{1}^{[1,k+1]}
\end{bmatrix*}\\
&\ ={\ek_{n-k,n-k}}(X){\ek_{k+2,k+2}}(X)\\
& \quad\quad\times \left (  (X_+X_0)_{1,k+1} \det N_{[n-k,n]}^{[1,k+1]} + \sum_{j=1}^k (-1)^{j+k}(X_+X_0)_{1j} \det N_{[n-k,n]}^{[1,k+1]\setminus \{j+1\}}\right )\\
&\ =\ek_{n-k,1}(X)\cdot(X_+)_{1,k+1}\\ &\quad\quad+{\ek_{n-k,n-k}}(X){\ek_{k+2,k+2}}(X)
  \left (\sum_{j=1}^k (-1)^{j+k}(X_+X_0)_{1j} \det N_{[n-k,n]}^{[1,k+1]\setminus \{j+1\}}\right );
\end{align*}
in the last equality we used~\eqref{eq:Nminors} for $i=n-k$ and $j=1$ and~\eqref{X0viak} for $i=k+1$. 
 The resulting formulas for $\phi_{1,n-2}(X)$, $\phi_{2,n-3}(X),\dots,\phi_{n-1,1}(X)$ can be viewed as a triangular system of linear equations for $(X_+)_{12},\ldots, (X_+)_{1,n-1}$. Due to Lemma~\ref{lem:X<Laurent}, the coefficients of this system  are good Laurent,  and in solving the system recursively for $(X_+)_{12},\ldots, (X_+)_{1,n-1}$, one only needs to divide by $\ek_{n-k,1}(X),\ldots,\ek_{21}(X)$, which proves the claim.
\end{proof}

To obtain  required Laurent polynomial expressions for the remaining matrix entries of $X_+$, we first generalize the map 
$$
\Psi':GL_n \to GL_{n-1} ,\quad  \Psi': X\mapsto U(X)=X_L X^{-1}_R
$$ 
from Section~\ref{twomapsiproof} to a map 
$$
{\widehat\Psi}':GL_n\times GL_n\to GL_{n-1} , \quad  {\widehat\Psi}': (X,Y)\mapsto {\widehat U}(X,Y)=X_L Y^{-1}_R
$$
and observe that it satisfies the following easily verifiable invariance property:
\[
{\widehat U}(\gamma(\beta)X\nu,\beta Y\gamma(\nu))=\beta_{[1,n-1]}^{[1,n-1]}{\widehat U}(X,Y)\left (\beta_{[1,n-1]}^{[1,n-1]}\right )^{-1}
\]
for any $\beta\in\N_+$ and $\nu\in\N_-$. Therefore,
\begin{align}
\label{eq:Ured}
U(X)&={\widehat U}\left (X_+X_{0,-},X_+X_{0,-}\right )\\
\nonumber
&=(X^{-1}_+)_{[1,n-1]}^{[1,n-1]}{\widehat U}\left (\gamma(X^{-1}_+)X_+X_0,X_0X_-\gamma(X_-^{-1}) \right ) (X_+)_{[1,n-1]}^{[1,n-1]}.
\end{align}

Denote 
\begin{equation}
\label{eq:Vfrom7.2}
V=\gamma(X^{-1}_+)X_+ .
\end{equation}
Since the functions~\eqref{phidef} are invariant under conjugation by upper unipotent elements of $GL_{n-1}$, the formula above implies
\begin{equation}
\label{eq:UhatU}
\varphi_{kl}(U(X)) = \varphi_{kl}\left ( {\widehat U}\left (V X_0,X_0 N) \right )\right ).
\end{equation}
To better understand the structure of  $\widehat\UU={\widehat U}\left (V X_0,X_0 N\right )$, observe that for an upper triangular matrix $B$ and a lower triangular matrix $C$, 
\begin{align*}
B_L &=B_{[2,n]}^{[2,n]}
\begin{bmatrix} 
0 &  \one_{n-2}\\
-\frac{b_{11}}{b_{1n}} &\cdots  -\frac{b_{1,n-1}}{b_{1n}} 
\end{bmatrix}, \\
C^{-1}_R &=
\begin{bmatrix} 
-\frac{c_{n2}}{c_{n1}} &\cdots  -\frac{c_{nn}}{c_{n1}} \\
 \one_{n-2}& 0\\
\end{bmatrix}^{-1} \tilde C^{-1}
%\left(C_{[1,n-1]}^{[1,n-1]}\right)^{-1} 
= \begin{bmatrix} 
0 &  \one_{n-2}\\
-\frac{c_{n1}}{c_{nn}} &\cdots  -\frac{c_{n,n-1}}{c_{nn}} 
\end{bmatrix}\tilde C^{-1}
%\left(C_{[1,n-1]}^{[1,n-1]}\right)^{-1}.
\end{align*}
with $\tilde C=C_{[1,n-1]}^{[1,n-1]}$.
Consequently, if we choose $B= V X_0$, $C= X_0 N$, then
\begin{equation*}
%\label{eq:Uhat}
\widehat{\UU}
%&:={\hat U}\left (V X_0,X_0 N\right ) \\
= B_{[2,n]}^{[2,n]}
\begin{bmatrix} 
0 &  \one_{n-2}\\
-\frac{(X_{+,0})_{11}}{x_{1n}} &\cdots  -\frac{(X_{+,0})_{1,n-1}}{x_{1n}} 
\end{bmatrix}
 \begin{bmatrix} 
0 &  \one_{n-2}\\
-N_{n1} &\cdots  -N_{n,n-1}
\end{bmatrix}\tilde C^{-1}.
%\\(N^{-1})_{[1,n-1]}^{[1,n-1]}(X^{-1}_0)_{[1,n-1]}^{[1,n-1]}.
\end{equation*}

Denote $M= \left (B_{[2,n]}^{[2,n]}\right )^{-1}\hat{\mathrm U}$, 
and let
\begin{equation}
\label{eq:D12}
D_1=\mbox{diag}(\one_{n-2}, x_{1n}),\quad  D_2=\mbox{diag}(\one_{n-2}, \det X).
\end{equation}
\begin{lemma}
\label{lem:MgoodLaurent}
The  matrices
$\left (M w_0\right )^{-1}_{+}$ and $\left (M w_0\right )^{-1}_{0,-} (D_1 D_2)^{-1}$
 are good Laurent.
\end{lemma}

\begin{proof}
By Lemmas~\ref{lem:X<Laurent} and~\ref{lem:X+1Laurent}, all matrix entries of $M$ are Laurent. Furthermore, after multiplying $M$ by $D_1$ on the left and by $w_0 D_2 w_0^{-1}=\mbox{diag}(\det X, \one_{n-2})$ on the right, all the matrix entries of the resulting matrix $\widetilde M$ become good Laurent.

Further, $\left (M w_0\right )_{0,-} = ( \widehat{\UU}w_0)_{0,-}$, so by~\eqref{eq:Ured},~\eqref{phidef}, 
and~\eqref{badphi},
\begin{align*}
\det(M w_0)_{[i,n-1]}^{[i,n-1]}&=\det\widehat{\UU}_{[i,n-1]}^{[1,n-i]}=\det U(X)_{[i,n-1]}^{[1,n-i]}
= \pm\frac{\phi_{i-1,n-i}(X)}{x_{1n}\det X},\\
\det(\widetilde M w_0)_{[i,n-1]}^{[i,n-1]}&= \det(D_1 M w_0 D_2)_{[i,n-1]}^{[i,n-1]}=\pm\phi_{i-1,n-i}(X).
\end{align*}
Recall that the denominators that enter the Gauss factorization of a matrix are its principal trailing minors, hence
matrix entries of all factors in the Gauss factorization of $M w_0$ are Laurent, and those of $\widetilde M w_0$ are good Laurent. Consequently, $\left (M w_0\right )_{+} =  D_1^{-1}   (\widetilde M w_0)_{+} D_1$ is also good Laurent and so is its inverse.  Finally, $\left (M w_0\right )_{0,-} =  D_1^{-1} (\widetilde M w_0 )_{0,-} D_2^{-1}$ and so
\[
\left (M w_0\right )^{-1}_{0,-} (D_1 D_2)^{-1} = D_2  (\widetilde M w_0)^{-1}_{0,-} D_2^{-1}
\]
is good Laurent.
\end{proof}

In order to restore the matrix entries of $B_{[2,n]}^{[2,n]}$ as Laurent polynomials in~\eqref{cuttlemut}, \eqref{cuttlefroz}, we need to briefly review some technical results from~\cite{GSVdouble}. (Note that the statements and formulas we will need are obtained from their counterparts in~\cite{GSVdouble} via consistently applied conjugation 
by $w_0$ and replacing $n$ with $n-1$.) Let $w_c = E_{12} + \ldots + E_{n-2, n-1} +E_{n-1,1}$ 
be a representative in $SL_{n-1}$ of the Coxeter element $s_{n-2}\cdots s_1\in S_{n-1}$. By~\cite[Lemma 8.2]{GSVdouble}, one can uniquely represent $\widehat{\UU}$ as
$\widehat{\UU} = \hat\beta  (\hat\nu w_c)  {\hat\beta}^{-1}$
with $\hat\beta \in \N_+$, $\hat\nu \in \B_-$ depending rationally on matrix entries of $\widehat{\UU}$.
Moreover, by~\cite[Section 7.3.2]{GSVdouble}, the matrix entries of $\hat\nu=\left (\hat\nu_{ij} \right )_{i,j=1}^{n-1}$ are Laurent polynomials in terms of functions  $\varphi_{kl}(\widehat{\UU})$.
The latter claim follows from the formulas~\cite[equation~(5.8)]{GSVdouble} that can be rewritten as
\begin{equation}
\label{eq:phi_from_nu}
\begin{aligned}
&\varphi_{kl}(\widehat{\UU}) = \pm\det \hat\nu^{[1,l-1]}_{[k+1,k+l-1]} \eta_1\cdots\eta_{n-k-l}, \quad l >1, k+l\leq n-1, \\ 
&\varphi_{k1}(\widehat{\UU}) = \pm\eta_1\cdots\eta_{n-k-1},\quad  k\in [1,n-2],
\end{aligned}
\end{equation}
where
$\eta_s= \hat\nu_{n-1,n-1} \cdots \hat\nu_{n-s,n-s}$ for  $s\in [1,n-2]$.
From~\eqref{eq:phi_from_nu} one obtains monomial expressions in terms of functions $\varphi_{kl}(\widehat{\UU})$ for the diagonal entries and flag minors of $\hat{\nu}_{[1,n-2]}^{[1,n-2]}$ :
\begin{equation}\label{eq:nu_from_phi}
\begin{aligned}
&\hat\nu_{11} = \pm \frac{x_{n1} \varphi_{21}(\widehat{\UU})}{x_{1n} \varphi_{11}(\widehat{\UU})},\\
&\hat\nu_{kk} = \pm\frac{\varphi_{k-1,1}(\widehat{\UU})\varphi_{k+1,1}(\widehat{\UU})}{\varphi_{k1}(\widehat{\UU})^2},\quad k\in [2,n-2],\\ 
&\hat\nu_{n-1,n-1} = \pm \varphi_{n-2,1}(\widehat{\UU}),\\ 
& \det \hat\nu^{[1,l-1]}_{[k+1,k+l-1]}=\pm \frac{\varphi_{kl}(\widehat{\UU})}{\varphi_{k+l-1,1}(\widehat{\UU})}, \quad
 l\leq n-1, k < n-l.
\end{aligned}
\end{equation}
These form an initial cluster for the standard cluster structure on $(n-2)\times(n-2)$ lower triangular matrices which is well known to be complete (see the proof of Lemma~\ref{lem:X<Laurent}).  This ensures that the matrix entries of 
 $\hat{\nu}_{[1,n-2]}^{[1,n-2]}$ are Laurent polynomials in~\eqref{cuttlemut},~\eqref{cuttlefroz}.
Finally, Laurent polynomial expressions for 
the last row of $\hat{\nu}$ were obtained in~\cite[Section 7.3.2]{GSVdouble} by representing the entries 
$\hat{\nu}_{n-1,j}$,  $j\in[1, n-2]$ as polynomials in off-diagonal entries of $\hat{\nu}_{[1,n-2]}^{[1,n-2]}$ and coefficients $c_1(\widehat{\UU}),\ldots, c_{n-2}(\widehat{\UU})$ divided by monomials in $\hat\nu_{kk}$,  $k\in [1,n-2]$.

\begin{lemma}
\label{lem:B+}
The matrix $B_{[2,n]}^{[2,n]}$ in~\eqref{eq:Ured} can be written as 
\[
B_{[2,n]}^{[2,n]}= w_0 (M w_0)^{-1}_{0,-} (\hat\nu w_cw_0)_{0,-} w_0^{-1} (\hat\nu w_cw_0)_+(M w_0)_+^{-1}.
\]
\end{lemma}

\begin{proof} Recall that $\widehat{\UU} = B^{[2,n]}_{[2,n]} M = \hat\beta  (\hat\nu w_c)  {\hat\beta}^{-1}$. Multiplying by $w_0$ on the right and taking the Gauss factorization of both sides we arrive at
\[
B^{[2,n]}_{[2,n]} (M w_0)_+ (M w_0)_{0,-} = \hat\beta  (\hat\nu w_c w_0)_+  (\hat\nu w_c w_0)_{0,-} w_0^{-1}  {\hat\beta}^{-1} w_0.
\]
Comparing the factors in the latter equality, we obtain
\[
\hat\beta^{-1} = w_0 (\hat\nu w_c w_0)_{0,-}^{-1} (M w_0)_{0,-}  w_0^{-1} 
\]
and
\[
B^{[2,n]}_{[2,n]}= \hat\beta (\hat\nu w_c w_0)_+ (M w_0)_+^{-1} = w_0 (M w_0)^{-1}_{0,-} (\hat\nu w_cw_0)_{0,-} w_0^{-1} (\hat\nu w_cw_0)_+(M w_0)_+^{-1}.
\]
\end{proof}

We are now ready to complete the proof that $X$ is good Laurent.
Indeed the first row of $V$ defined in~\eqref{eq:Vfrom7.2} coincides with the first row of $X_+$ , while
$V_{[2,n]}^{[2,n]}=B^{[2,n]}_{[2,n]} \left ((X_0)^{[2,n]}_{[2,n]} \right )^{-1}$. 
%By Lemmas \ref{lem:X+1Laurent},\ref{lem:B+}, 
If $V$ is good Laurent, then the same is true about $X_+= \cdots \gamma^2(V) \gamma(V) V$ and, due to Lemma \ref{lem:X<Laurent}, about $X= X_+ X_{0,-}$.

In view of  Lemmas~\ref{lem:X<Laurent} and~\ref{lem:X+1Laurent}, to verify that $X$ is good Laurent, we only need to establish a similar claim for $B^{[2,n]}_{[2,n]}$. First, we analyze how  the frozen variables $x_{n1}$, $x_{1n}$,
 and $\det X$ enter the Laurent polynomial expressions for the matrix entries $\hat\nu_{ij}$ discussed above. Observe that by~\eqref{eq:UhatU} and~\eqref{badphi}, equations~\eqref{eq:nu_from_phi} can be re-written as
\begin{equation}\label{eq:nu_from_phiX}
\begin{aligned}
&\hat\nu_{11} = \pm \frac{x_{n1} \det X \phi_{21}(X)}{ \phi_{11}(X)},\\
&\hat\nu_{kk} = \pm\frac{\phi_{k-1,1}(X)\phi_{k+1,1}(X)}{\phi_{k1}(X)^2},\quad k\in [2,n-2],\\ 
&\hat\nu_{n-1,n-1} = \pm \frac{\phi_{n-2,1}(X)}{x_{1n}\det X},\\ 
& \det \hat\nu^{[1,l-1]}_{[k+1,k+l-1]}=\pm \frac{\phi_{kl}(X)}{\phi_{k+l-1,1}(X)}, \quad l\leq n-1, k < n-l.
\end{aligned}
\end{equation}
This means, in particular, that expressions for entries of the rows 2 through $n-2$ of $\hat\nu$ do not contain frozen variables at all. 

Let us look more closely at the Laurent polynomial expressions for the entries of the last row of $\hat\nu$, following the strategy of~\cite[Section 7.3.2]{GSVdouble}.
It involves a conjugation of $\hat\nu w_c$ to its companion form performed in two steps. The first is the diagonal conjugation by 
\[
T=\mbox{diag}(t_1,\ldots,t_{n-1}):=\mbox{diag}(\hat\nu_{11}\cdots\hat\nu_{n-2,n-2}, \hat\nu_{22}\cdots\hat\nu_{n-2,n-2},\ldots,\hat\nu_{n-2,n-2}, 1)
\]
that results in $T^{-1} \hat\nu w_c T =\begin{bmatrix}0 & \zeta\\ \varkappa & \xi   \end{bmatrix}$, where 
\[
\varkappa= \hat\nu_{11}\cdots\hat\nu_{n-1,n-1}= \det\hat\nu=\pm \det U(X) = \pm \frac{x_{n1}}{x_{1n}}, \quad
\xi = (\hat\nu_{n-1,j}t_{j+1})_{j=1}^{n-2}, 
\]
and $\zeta$ is a unipotent $(n-2)\times (n-2)$ lower triangular matrix whose off-diagonal entries $\zeta_{ij}= \hat\nu_{ij}t_i^{-1}t_{j+1}$,  $1\le j<i<n-1$, are good Laurent by the second and the fourth equations 
in~\eqref{eq:nu_from_phiX}. The second step is the conjugation by the matrix 
$\eta=\begin{bmatrix} 1 & 0\\ 0 & \zeta^{-1} \gamma(\zeta^{-1}) \gamma^2(\zeta^{-1}) \cdots  \end{bmatrix}$. The resulting matrix
$\eta^{-1}T^{-1} \hat\nu w_c T\eta$ is in a  companion form with coefficients of the characteristic polynomial in the last row. Since $\hat\nu w_c$ is similar to $U(X)$, the last row of this companion matrix is made of functions 
$\pm c_p(U(X))$. We conclude that the matrix entries $\hat\nu_{n-1,j}$,  $j\in[1, n-2]$ are represented by expressions that are linear in $c_i(U(X))$, polynomial in $\zeta_{ij}$, and Laurent polynomial in $\hat\nu_{ii}$, $i\in[2, n-2]$. 
By~\eqref{badc},
$c_p(U(X))=\frac{\ec_p(X)}{x_{1n}\det X}$. Together with the third equation in~\eqref{eq:nu_from_phiX}, this means that every matrix entry of the last row of $\hat\nu$ is a good Laurent function divided by $x_{1n}\det X$. To summarize, the matrix $\bar \nu=D_1 D_2 \hat\nu$, where $D_1, D_2$ are defined in~\eqref{eq:D12}, is good Laurent.

Now, consider the Gauss factorization of $\bar \nu w_c w_0$. The trailing principal minors are
\begin{multline*}
\det(\bar \nu w_c w_0)_{[i,n-1]}^{[i,n-1]} =\pm \bar \nu_{n-1,n-1} \det(\bar \nu)_{[i,n-2]}^{[1,n-i-1]}\\
= \pm x_{1n}\det X \hat\nu_{n-1,n-1} \det(\hat\nu)_{[i,n-2]}^{[1,n-i-1]}
= \pm \phi_{n-2,1}(X)  \frac{\phi_{i-1,n-i}(X)}{\phi_{n-2,1}(X)}.
\end{multline*}
Thus, the factors in the Gauss factorization of $\bar \nu w_c w_0$ are good Laurent, and so is
$(\hat\nu w_c w_0)_+ = (D_1 D_2)^{-1} (\bar\nu w_c w_0)_+ D_1 D_2$, while
 $(\hat\nu w_c w_0)_{0,-} = (D_1 D_2)^{-1} (\bar \nu w_c w_0)_{0,-}$ is Laurent, but not good Laurent. However, we can rewrite the equation for $B_{[2,n]}^{[2,n]}$ from Lemma~\ref{lem:B+} as follows:
 \begin{align*}
 B_{[2,n]}^{[2,n]}&= w_0 (M w_0)^{-1}_{0,-} (\hat\nu w_cw_0)_{0,-} w_0^{-1} (\hat\nu w_cw_0)_+(M w_0)_+^{-1}\\
 &= w_0 \left((M w_0)^{-1}_{0,-} (D_1 D_2)^{-1}\right )  (\bar \nu w_cw_0)_{0,-} w_0^{-1} (\hat\nu w_cw_0)_+(M w_0)_+^{-1}.
 \end{align*}
 Taking into account Lemma~\ref{lem:MgoodLaurent}, we see that in the last expression above all factors are good Laurent, which concludes the proof. \qed

\subsection{Final steps in the proof of Theorem~\ref{initialseed}} To complete the proof of Theorem~\ref{initialseed} we have to check that conditions of Proposition~\ref{twolaurent} are valid. We will apply this proposition to the Laurent representations obtained in Theorems~\ref{thm:InitLaurent} and~\ref{fish_to_cuttlefish}(ii). Therefore, we have to check 
coprimality for any polynomial in the family $\left\{\eg_{ii}(X), i\in[2,n-1]; x_{1n}; x_{n1}; \det X\right\}$, with any 
polynomial in the family 
$\left\{\ek_{ij}(X), 1\le j\le i\le n, (i,j)\ne (n,1), (1,1); \phi_{kl}(X), k,l\in[1,n-2], k+l\le n-1\right\}$.

We start from the following statement.

\begin{proposition}\label{giiirr}
Polynomials $\eg_{ii}(X)$, $i\in[2,n-1]$, are irreducible.
\end{proposition}

\begin{proof} Note that $\eg_{n-1,n-1}(X)$ equals, up to a sign, to the resultant of polynomials 
$\sum_{t=0}^{n-1}x_{1,n-t}z^t$ and $\sum_{t=0}^{n-1}x_{n,n-t}z^t$, which is classically known to be irreducible 
(see, e.~g.,~\cite[Ch.~5.9]{VdW}).

For $i\in [2,n-2]$, assume that $\eg_{ii}(Y)$ are irreducible for $Y$ of size $k\times k$ with $k\in[i+1,n-1]$ (note that for $k=i+1$, $\eg_{ii}(Y)$ is once again a resultant), and consider $\eg_{ii}(X)$. Clearly,
$\eg_{ii}(X)$ is linear in all variables $x_{i+1,l}$ for $l\in[1,n]$. Let $c_{i+1,n}$ denote the coefficient at $x_{i+1,n}$ in $\eg_{ii}(X)$. It is easy to see that $c_{i+1,n}$ restricted to $x_{1n}=x_{nn}=0$ equals
$\eg_{ii}(Y)$ for the matrix $Y$ obtained from $X$ by deleting the $(i+1)$th row and the last column, and hence is irreducible by the assumption, so that $c_{i+1,n}$ itself is irreducible as well. Consequently, if $\eg_{ii}(X)=PQ$ and $P$ depends on $x_{i+1,n}$, then either $P=x_{i+1,n}c_{i+1,n}+\dots$, and hence $Q=1$, or $P=x_{i+1,n}+\dots$, and hence $Q=c_{i+1,n}$. The latter case is impossible, since $x_{i+1,n-1}$ enters $P$ together with $x_{i+1,n}$, so that $PQ$ would contain a monomial with $x_{i+1,n-1}x_{1,n-1}^i$, while $\eg_{ii}(X)$ only contains monomials with
$x_{i+1,n-1}x_{1,n-1}^k$ for $k<i$.
\end{proof}

Note that all polynomials in the first family, except for $x_{n1}$, depend on the entries of the first row of $X$, while all polynomials $\ek_{ij}(X)$ in the second family do not depend on the entries of this row and are not divisible by $x_{n1}$. Consequently, by Proposition~\ref{giiirr}, polynomials $\ek_{ij}(X)$ in the second family are coprime with all polynomials in the first family. Further, polynomials $\phi_{kl}(X)$ are not divisible by any of
$\eg_{ii}(X)$. Indeed, all $\eg_{ii}(X)$ vanish identically on $x_{11}=x_{n1}=0$. On the other hand, none of 
$\phi_{kl}(X)$ have this property: it would imply that there exists a non-trivial linear combination of rows of
$\left.\Phi_{kl}(X)\right|_{x_{11}=x_{n1}=0}$ representing zero, which is impossible, since the entries of each column of this matrix are distinct independent variables. In a similar way, the fact that $\phi_{kl}(X)$ are not divisible
by $\det X$ follows from the fact that the latter vanishes identically for $x_{12}=x_{22}=\dots=x_{n2}=0$, while the former do not. Clearly, $\phi_{kl}(X)$ are not divisible by $x_{1n}$ and $x_{n1}$ either, so the required coprimality 
follows once again from Proposition~\ref{giiirr}.

\begin{remark} In fact, one can use the above reasoning together with~\cite[Lemma 5.5]{Vol} to prove that all functions in both families are irreducible.
\end{remark}

We can now check that the ground ring for the obtained generalized cluster algebra is indeed $\C[x_{1n}, x_{n1}, \det X^{\pm1}]$, as indicated in Remark~\ref{threestruct}. To achieve this goal we will prove that the Laurent representations for the matrix entries obtained in Section~\ref{proofInitLaurent}, in fact, do not contain $x_{1n}$, $x_{n1}$, and $\det X$ in the denominator. This statement is based on the following observation.

\begin{proposition}\label{nofrozen}
If the Laurent representation of an element of a generalized upper cluster algebra in some seed $\Sigma$ does not contain frozen variables in the denominator then the same holds for Laurent representations in all seeds.
\end{proposition}

\begin{proof}  
Clearly, it is enough to prove that the Laurent representation in the adjacent seed $\overline\Sigma$ obtained  by replacing   a cluster variable $f$ by $\bar f$ has this property. Let $R$ and $\bar R$ be the rings of polynomials in the cluster variables of $\Sigma$ and $\overline\Sigma$, respectively. Assume that $x=QM/f^k$, where $k\in\Z$, $Q\in R$ is  not divisible by a monomial and $M$ is a Laurent monomial in cluster variables of $\Sigma$ excluding $f$ and not containing frozen variables in the denominator. Further, let $x=\bar Q\bar M$, where $\bar Q\in\bar R$ is not divisible by a monomial and $\bar M$ is a Laurent monomial in cluster variables of 
$\overline\Sigma$. Recall that $f$ and $\bar f$ satisfy the exchange relation $f\bar f=P$, where $P\in R\cap\bar R$.
Let $m$ be the minimal non-negative integer such that $\tilde Q=\bar f^m \left.Q\right|_{f=P/\bar f}\in\bar R$, then
$\tilde Q \bar f^{k-m}M=\bar Q\bar M P^k$ is an identity in $\bar R$. Since cluster variables in $\overline\Sigma$ are algebraically independent, this identity is only possible if $\tilde Q=\bar Q P^k$ and $\bar M=\bar f^{k-m}M$, therefore, $\bar M$ retains the property of not containing frozen variables in the denominator.
\end{proof}

Taking into account that Laurent representations for the matrix elements in Theorem~\ref{fish_to_cuttlefish} do not contain frozen variables in the denominator, we get the required statement. \qed

\begin{remark}\label{whyunfreeze}
We can now check that unfreezing vertex $D$ was imminent, as indicated in Section~\ref{outline}.4. Indeed, by
Proposition~\ref{Dexchange}, the bracket between the cluster variable $\eg_{n-1,n-1}(X)$ at $D$ and the function $\eg(X)$ that replaces it after the mutation at $D$ is given by
\[
\{\eg_{n-1,n-1}(X),\eg(X)\}=c_1\ef_{2n-3}(X)\eg_{n-2,n-2}(X)+c_2x_{n1}\ef_{2n-4}(X)\eg_{n-1,n-2}(X)
\]
for some constants $c_1$ and $c_2$. Computations in the proof of Proposition~\ref{compatchar} show that $c_2-c_1=1$, so that
\[
\frac{\{\eg_{n-1,n-1}(X),\eg(X)\}}{\eg_{n-1,n-1}(X)}=c_2\eg(X)-\frac{\ef_{2n-3}(X)\eg_{n-2,n-2}(X)}{\eg_{n-1,n-1}(X)}.
\]
By Proposition~\ref{giiirr}, polynomials $\eg_{n-1,n-1}(X)$ and $\eg_{n-2,n-2}(X)$ are irreducible. Further,  $\deg\,\ef_{2n-3}(X)<
\deg\,\eg_{n-1,n-1}(X)$, hence  the expression above is not a regular function on 
$GL_n$, and Proposition~\ref{frozenchar} applies.
\end{remark}
 
\subsection{Proof of Theorem~\ref{initialalmostseed}}
We perform the regular pullback of the generalized structure $\G\CC_n$ on $GL_n$ to $\Mat_n$ as 
described in~\cite{GSVpullback}, taking $g_1=\det X$ as the distinguished polynomial.
The function $\fy_1(r)$ computed via the first relation in~\cite[Eq.~(4.5)]{GSVpullback} equals~1 for $r\in[1,n-2]$ 
and vanishes for $r=0, n-1$, and 
$\tau_1$ computed via the second relation in~\cite[Eq.~(4.5)]{GSVpullback} vanishes. Therefore, we are exactly within
conditions of~\cite[Remark 4.6]{GSVpullback}, which stipulate one arrow between $B$ and $(1,1)$ in both directions in the 
quiver $\widehat Q_n$ and the rules of its mutations as described in Section~\ref{outline}.6.

The compatibility of $\widehat\GCC_n$ with $\Poi$ is an immediate consequence of the compatibility of $\G\CC$ with $\Poi$, since $\det X$ is a Casimir for $\Poi$. Regularity of  $\widehat\GCC_n$ follows from the regularity of $\G\CC_n$. To prove the completeness we invoke~\cite[Theorem 6.2]{GSVpullback}. By Theorem~\ref{fish_to_cuttlefish}(ii), every matrix element 
$x_{ij}$ of $X$ can be written as a Laurent polynomial in the cluster variables of the seed described in part (i) of the same theorem, and these Laurent polynomials do not contain frozen variables, in particular $\det X$, in the denominator.
Therefore,  every $x_{ij}$ can be written as a Laurent polynomial in these cluster variables at a generic point of the
hypersurface $\{\det X=0\}$, and~\cite[Theorem 6.2]{GSVpullback} applies. \qed

\appendix
\renewcommand{\theequation}{A.\arabic{equation}}
\setcounter{equation}{0}
\section*{Appendix: Regular complete cluster structure on the space $\bar \RR_n$ of rational functions}
%\section{Proof of Theorem~\ref{todacs}}\label{prooftodacs}
The goal of this appendix is to prove Theorem~\ref{todacs}. 
We will utilize a map $\rho$ from $\bar \RR_n$ to $\RR_n$ that sends 
\[
\bar M(\lambda) = \frac{\bar q(\lambda)}{p(\lambda)}= \sum_{i=0}^\infty\frac{\bar h_i}{\lambda^i} 
\]
with $\bar q(\lambda)=\sum_{i=0}^{n}\bar q_i\lambda^i$, $p(\lambda)=\lambda^n + \sum_{i=0}^{n-1}p_i\lambda^i$ 
to 
\[
M(\lambda) = \bar M(\lambda) - \bar q_n  = \bar M(\lambda) - \bar h_0 = \sum_{i=0}^\infty\frac{h_i}{\lambda^{i+1}}=\frac{q(\lambda)}{p(\lambda)},
\]
so that $h_n = \bar h_{n+1}$ for $n =0, 1, \ldots$ and $q(\lambda)= \bar q(\lambda) - \bar q_n \lambda^n$.

The map $\rho$ can be realized by a sequence of cluster mutations in the following way. First, note that the cluster variable $\bar t_m^+$ 
attached to the vertex $(m,+)$ in $\bar Q_n^T$ coincides with $t_m^-$ for $m=1,\ldots, n$.
 Apply to the initial seed $(\bar\TE_n, \bar Q_n^T)$ the sequence  $\mu$ of mutations  at vertices 
$(1,-), (2,-), \ldots, (n,-)$. The cluster variable $\bar t_m^-$  attached to the vertex $(m,-)$ then transforms into 
\[
(\bar t_m^-)'=\det (\bar h_{\alpha+\beta})_{\alpha,\beta=1}^m =
\det (\bar H_{m+1}^+)_{\hat 1}^{\hat 1}=
 \det H_m^+ = t_m^+
\] 
for $m=1,\ldots, n-1$ and into 
\[
\frac{t_m^+}{t_m^-} =
\frac{\det (\bar h_{\alpha+\beta})_{\alpha,\beta=1}^m}{\det (\bar h_{\alpha+\beta-1})_{\alpha,\beta=1}^m} 
\]
for $m=n$. 

Indeed, it is straightforward to see that the cluster transformations involved in $\mu$ are given by
\begin{align*}
\bar t_1^- (\bar t_1^-)'&= \bar t_2^- + (\bar t_1^+)^2 , \\
\bar t_m^- (\bar t_m^-)'&= \bar t_{m+1}^-(\bar t_{m-1}^-)' + (\bar t_m^+)^2,\quad  m\in[2, n-1],\\
\bar t_n^- (\bar t_n^-)'&= \frac{\bar t_{m+1}^-}{\bar t_n^+}(\bar t_{m-1}^-)' + \bar t_n^+\ .
\end{align*}
Then the claim follows from the Desnanot--Jacobi identity for Hankel determinants, which in our notations takes a form
\[
\bar t_m^- t_m^+ =   \bar t_{m+1}^-t_{m-1}^+ +  (\bar t_m^+)^2.
\]

We see that cluster variables attached to vertices $(m,\pm)$, $m\in[1,n]$, in the seed $\mu(\bar\TE_n, \bar Q_n^T)$ coincide with cluster variables attached to the same vertices in $(\TE_n, Q_n^T)$. In fact, it is easy to check that the restriction of $\mu(\bar Q_n^T)$ to the full subquiver on vertices $(m,\pm)$, $m\in[1,n]$, coincides with $Q_n^T$, if we consider the vertex $(n,+)$ frozen. Now we can invoke results of~\cite{GSV_Acta} about the cluster structure $\CC_n^T$ on 
$\RR_n$ defined by the seed $(\TE_n, Q_n^T)$ to establish the required properties of $\bar\CC_n^T$.

We start with regularity. Recall that the coefficients $\bar h_i$ in the Laurent expansion of $\bar M(\lambda)$ are polynomials in coefficients of $p(\lambda)$ and $\bar q(\lambda)$, as can be seen from~\eqref{barqviaph}, 
and that $\bar t_n^+$ does not vanish on $\bar\RR_n$ due to the classical criterion of coprimality of $p(\lambda)$ and  
$\bar q(\lambda)$, see, e.g.,~\cite[Theorem 8.39]{Fuh}. Therefore, mutations that constitute our sequence $\mu$ yield regular functions on $\bar\RR_n$. If  $\mu$ is followed up with mutations in the remaining vertices $(1,+),\ldots, (n-1,+)$, the resulting cluster variables are also regular on $\bar\RR_n$ as pullbacks by the map $\rho$ of cluster variables in  the structure $\CC_n^T$ known to be regular. As explained in~\cite[Section 9]{quasichris}, this guarantees regularity of $\bar\CC_n^T$.

Next, we will show that $\bar\CC_n^T$ is complete over the ground ring $\C[\bar t_{n+1}^-/\bar t_n^+, (\bar t_n^+)^{\pm1}]$, as indicated in Remark~\ref{grfortoda}. By~\cite[Lemma 5.3]{GSV_Acta}, $h_i$, $i\in[0, 2n-1]$, are cluster variables in 
$\CC_n^T$, and so $\bar h_i$, $i\in[0,2n]$, are cluster variables in $\bar\CC_n^T$.  We can rewrite equations~\eqref{barqviaph} for $j\in[1,n]$ as 
\begin{equation}
\label{qTodacomplete}
[\bar q_n,\ldots, \bar q_1]= [\bar h_0,\ldots, \bar h_{n-1}] 
\begin{bmatrix}
1 & p_{n-1} & p_{n-2} & \cdots &p_1\\
0 & 1 & p_{n-1} & p_{n-2} & \cdots\\
\vdots & \ddots & \ddots & \ddots & \vdots\\
0 & \cdots & 0 &1 & p_{n-1}\\
0 &  \cdots& 0  & 0 & 1
\end{bmatrix},
\end{equation}
and those for $j\in[-n,0]$ as 
\begin{equation}
\label{pTodacomplete}
[\bar q_0,0,\ldots, 0]= [p_0,\ldots, p_{n-1},1] \bar H^-_{n+1}.
\end{equation}
The latter equation implies that 
\[
[p_0,\ldots, p_{n-1}] = - [\bar h_{n+1},\ldots, \bar h_{2n}] \left (\bar H^+_{n} \right )^{-1}
\]
and $\bar q_0 = \bar t_{n+1}^-/\bar t_n^+$.
Since the entries of $\bar H^-_{n+1}$ are cluster variables in $\bar\CC_n^T$ and $\det \bar H^+_{n}= \bar t_n^+$ is a frozen variable, we conclude that coefficients $p_0,\ldots, p_{n-1}$ belong to the corresponding cluster algebra. The same conclusion for $\bar q_1,\ldots, \bar q_n$ follows from \eqref{qTodacomplete}.

To complete the proof, we need to show compatibility of $\bar\CC_n^T$ with $\Poi^T$. Once again, we will rely on results from~\cite{GSV_Acta}. Since, by~\eqref{todabramoment}, and~\eqref{bartodabramoment}, coefficients $h_i$ and $\bar h_i$ satisfy the same Poisson relations with respect to $\Poi^T$, the same is true for Poisson brackets between Hankel determinants built from $h_i$ and $\bar h_i$ respectively. In the case of $\RR_n$, these were computed in~\cite{GSV_Acta} as a part of the proof of compatibility of $\CC_n^T$ with $\Poi^T$. 
We will now use the following straightforward observation.
Let $Q$ be a quiver in a cluster structure compatible with a Poisson bracket $\Poi$. Suppose we apply a monomial transformation to cluster variables attached to frozen variables in $Q$. Then to adjust $Q$ to a new quiver $\tilde Q$ so that the seed with unchanged mutable variables and the new frozen variables remains compatible with $\Poi$, one needs to change arrows between frozen and mutable vertices in $Q$ in such a way that all $y$-variables at mutable vertices remain unchanged. We apply this observation to $Q^T_{n+1}$ in which functions attached to the vertices $(m,\pm)$ are $\bar t_m^\pm$
instead of $t_m^\pm$.
%are as in the description below Fig.~\ref{Q5T} but with $t_m^\pm$ replaced by $\bar t_m^\pm$. 
We then freeze the vertex $(n,+)$, which detaches the frozen vertex $(n+1,+)$ from the quiver.  The only $y$-variable that involves the frozen vertex $(n+1,-)$
is associated with vertex $(n,-)$ and is equal to
\[ \frac{(\bar t_n^+)^2 \bar t^-_{n-1}}{(\bar t_{n-1}^+)^2 \bar t^-_{n+1}}
=  \frac{\bar t_n^+\bar t^-_{n-1}}{(\bar t_{n-1}^+)^2 (\bar t^-_{n+1}/\bar t_n^+)}.
\] 
Therefore, if one  replaces the frozen variable $\bar t_{n+1}^-$ with $\bar t_{n+1}^-/\bar t_{n}^+$, the requirement of compatibility with $\Poi^T$ leads one to recover the quiver $\bar Q^T_n$. This completes the proof. \qed


\begin{thebibliography}{GSV7}

\bibitem{BD} A.~Belavin and V.~Drinfeld,
\textit{Solutions of the classical Yang-Baxter equation for simple Lie algebras}.
Funktsional. Anal. i Prilozhen. {\bf16} (1982), 1--29.

\bibitem {BFZ}  A.~Berenstein, S.~Fomin, and A.~Zelevinsky,
\textit{Cluster algebras. III. Upper bounds and double Bruhat cells}. 
Duke Math. J. \textbf{126} (2005), 1--52.

\bibitem{CL} P.~Cao and F.~Li,
{\it On some combinatorial properties of generalized cluster algebras}, J. Pure Appl. Algebra {\bf 225} (2021), 106650.

\bibitem{CGGLSS} R.~Casals, E.~Gorsky, M.~Gorsky, I.~Le, L.~Shen, and J.~Simental, \textit{Cluster structures on braid varieties},
J. Amer. Math. Soc. \textbf{38} (2025), 369--479.

\bibitem{CP} V.~Chari and A.~Pressley, \textit{A guide to quantum groups}.
Cambridge University Press, 1994.



\bibitem{FaGe} L.~Faybusovich and M.~Gekhtman,
\textit{Poisson brackets on rational functions and multi-Hamiltonian structure for integrable
lattices}.
Phys.~Lett.~A, \textbf{272} (2000), 236--244.

%\bibitem{FoPy}
%S.~Fomin and P.~Pylyavskyy, {\it Tensor diagrams and cluster algebras}.
%Adv. Math. \textbf{300} (2017), 717--787.

\bibitem{quasichris}  C.~Fraser, {\it Quasi-homomorphisms of cluster algebras}, Adv. in Appl. Math. {\bf 81} (2016), 40--77. 

\bibitem{Fuh} P.A.~Fuhrmann, {\it A Polynomial Approach to Linear Algebra}, Universitext, Springer, New 
York, 1996.

\bibitem {GSV1}  M.~Gekhtman, M.~Shapiro, and A.~Vainshtein,
\textit{Cluster algebras and Poisson geometry}.  
Mosc. Math. J. \textbf{3} (2003), 899--934.

\bibitem{GSVb}  M.~Gekhtman, M.~Shapiro, and A.~Vainshtein,
\textit{Cluster algebras and Poisson geometry}.
Mathematical Surveys and Monographs, 167. American Mathematical Society, Providence, RI, 2010.

\bibitem{GSV_Acta}  M.~Gekhtman, M.~Shapiro, and A.~Vainshtein,
{\it Generalized B\"acklund--Darboux transformations for Coxeter--Toda flows from cluster algebra perspective}, Acta Mathematica {\bf 206} (2011), no.2, 245--310.


\bibitem{GSVM}  M.~Gekhtman, M.~Shapiro, and A.~Vainshtein,  
\textit{Cluster structures on simple complex Lie groups and Belavin--Drinfeld classification},
Mosc. Math. J. \textbf{12} (2012), 899--934.


\bibitem{GSVPNAS}  M.~Gekhtman, M.~Shapiro, and A.~Vainshtein,  
\textit{Cremmer--Gervais cluster structure on $SL_n$}.  
Proc.~Natl.~Acad.~Sci. \textbf{111} (2014), 9688--9695.

\bibitem{GSVMem}  M.~Gekhtman, M.~Shapiro, and A.~Vainshtein,  
\textit{Exotic cluster structures on $SL_n$: the Cremmer--Gervais case}.  
Memoirs of the AMS \textbf{246} (2017), no.~1165, 94pp.

\bibitem{GSVdouble} M.~Gekhtman, M.~Shapiro, and A.~Vainshtein,  \textit{Drinfeld double of $GL_n$ and generalized cluster structures}, Proc. London Math. Soc. \textbf{116} (2018), 429--484.

\bibitem{GSVnewdouble} M.~Gekhtman, M.~Shapiro, and A.~Vainshtein,  \textit{Generalized cluster structures related to the Drinfeld double of $GL_n$}, J. London Math. Soc. \textbf{105} (2022), 1601--1633.

\bibitem{GSVpest}  M.~Gekhtman, M.~Shapiro, and A.~Vainshtein, 
\textit{Periodic staircase matrices and generalized cluster structures},
Int. Math. Res. Notes (2022), 4181--4221.

\bibitem{GSVple}  M.~Gekhtman, M.~Shapiro, and A.~Vainshtein,  
\textit{A plethora of cluster structures on $GL_n$}, 
Memoirs of the AMS {\bf 297} (2024), no. 1486, 104 pp.

\bibitem{GSVuni} M.~Gekhtman, M.~Shapiro, and A.~Vainshtein, \textit{A unified approach to exotic cluster structures on simple Lie groups}, arXiv:2308.16701, 40 pages.

\bibitem{GSVpullback} M.~Gekhtman, M.~Shapiro, and A.~Vainshtein, \textit{Regular pullback of generalized cluster structures}, arXiv:26xx.xxxx, 15 pages.


\bibitem{GV}   M.~Gekhtman and D.~Voloshyn, \textit{Generalized cluster structures related to Poisson duals of $SL_n$}, arXiv:2312.04859, 90 pages.

\bibitem{r-sts}  A.~Reyman and M.~Semenov-Tian-Shansky,
\textit{Integrable systems. A group-theoretical approach}. Moscow--Izhevsk, 2003, ISBN 5-93972-262-8.

\bibitem{Scott} J.~Scott, \textit{Grassmannians and cluster algebras}, Proc. London Math. Soc. {\bf 92} (2006),  345--380.

\bibitem{VdW} B.~van der Waerden, \textit{Algebra.~Vol.~I}. Springer-Verlag, New York, 1991.

\bibitem{Vol} D.~Voloshyn, {\it Starfish lemma via birational quasi-isomorphisms},  J. Pure Appl. Algebra 
\textbf{229} (2025), 108--127.


\bibitem{Ya} M.~Yakimov, 
{\it Symplectic leaves of complex reductive Poisson-Lie groups}, Duke Math. J. \textbf{112} (2002),
453--509.

\end{thebibliography}
\end{document}